\documentclass[10pt,a4paper]{amsart}

\usepackage{amssymb}
\usepackage[utf8]{inputenc}
\usepackage[english]{babel}
\usepackage{amsfonts}
\usepackage{amsmath}
\usepackage{comment}
\DeclareFontFamily{U}{mathx}{\hyphenchar\font45}
\DeclareFontShape{U}{mathx}{m}{n}{
      <5> <6> <7> <8> <9> <10>
      <10.95> <12> <14.4> <17.28> <20.74> <24.88>
      mathx10
      }{}
\DeclareSymbolFont{mathx}{U}{mathx}{m}{n}
\DeclareFontSubstitution{U}{mathx}{m}{n}
\DeclareMathAccent{\widecheck}{0}{mathx}{"71}
\DeclareMathAccent{\wideparen}{0}{mathx}{"75}

\usepackage{amsthm}
\usepackage{color}
\usepackage{enumitem}
\usepackage[nameinlink]{cleveref}

\usepackage{dsfont}

\newtheorem{main-theorem}{Theorem}
\newtheorem{proposition}{Proposition}[section]
\newtheorem{theorem}[proposition]{Theorem}

\newtheorem{corollary}[proposition]{Corollary}
\newtheorem{lemma}[proposition]{Lemma}
\theoremstyle{remark}
\newtheorem{remark}[proposition]{Remark}

\theoremstyle{definition}
\newtheorem{definition}[proposition]{Definition}
\newtheorem*{acknowledgements}{Acknowledgements}

\DeclareMathOperator{\supp}{supp}

\DeclareMathOperator{\sign}{sign}

\DeclareMathOperator*{\esssup}{ess\,sup}
\DeclareMathOperator{\res}{res}

\newcommand{\R}{\mathbb{R}}
\newcommand{\C}{\mathbb{C}}
\newcommand{\Z}{\mathbb{Z}}
\newcommand{\N}{\mathbb{N}}
\newcommand{\Sph}{\mathbb{S}}
\newcommand{\dd}{\mathrm{d}}
\newcommand{\tm}{\mathrm{t}}
\newcommand{\x}{\mathrm{x}}
\newcommand{\Id}{\mathrm{Id}}
\newcommand{\Orth}{\mathrm{O}}
\newcommand{\T}{\mathsf{T}}

\title[Unbounded potentials]{The initial-to-final-state inverse problem with unbounded potentials and Strichartz estimates}

\author{Pedro Caro}
\address{Pedro Caro\\
Ikerbasque \&
Basque Center for Applied Mathematics, Bilbao, Spain\\}
\email{pcaro@bcamath.org}

\author{Alberto Ruiz}
\address{Alberto Ruiz\\
Universidad Aut\'onoma de Madrid, 28049 Cantoblanco, Spain}
\email{alberto.ruiz@uam.es}

\begin{document}

\begin{abstract}
The initial-to-final-state inverse problem consists in determining a 
quantum Hamiltonian assuming the knowledge of the state of the system at some fixed time, for 
every initial state. We formulated this problem to establish a theoretical framework that would explain the viability of data-driven prediction in quantum mechanics.
In a previous work, we analysed this inverse problem for 
Hamiltonians of the form $-\Delta + V$ with an electric potential $V = V(\tm, \x)$, and we showed
that uniqueness holds whenever the potentials are bounded and decay super-exponentially at 
infinity.
In this paper, we extend this result for unbounded potentials.
One of the key steps consists in proving a family of suitable Strichartz estimates---including
the corresponding endpoint of Keel and Tao.

In the context of the inverse Calderón problem this family of inequalities corresponds to the
Carleman inequality proved by Kenig, Ruiz and Sogge. Haberman showed that
this inequality can be also retrieved as an embedding of a suitable Bourgain space. 
The corresponding Bourgain space in our context do not capture the mixed-norm Lebesgue spaces of
Strichartz inequalities. In this paper, we give a counterexample that justifies this fact, and 
shows the limitations of Bourgain spaces to address the initial-to-final-state inverse problem.
\end{abstract}

\date{\today}


\maketitle

\section{Introduction}
The Schrödinger equation determines the evolution of a wave function
\[ u: (t, x) \in [0, T] \times \R^n \longmapsto u(t, x) \in \C \]
throughout an interval of time $[0, T]\subset \R$, with $n \in \N$.
Wave functions provide a mathematical description of 
the dynamical state of quantum particles.

If the motion takes place 
under the influence of an electric potential $V = V(\tm, \x)$ and the initial 
state $u(0,  \centerdot)$ is prescribed by $f$, then the corresponding 
wave function is the solution of
the initial-value problem for the Schr\"odinger equation
\begin{equation}
	\left\{
		\begin{aligned}
		& i\partial_\tm u = - \Delta u + V u & & \textnormal{in} \enspace (0,T) \times \R^n, \\
		& u(0, \centerdot) = f &  & \textnormal{in} \enspace \R^n.
		\end{aligned}
	\right.
	\label{pb:IVP}
\end{equation}

Whenever $V$ has suitable integrability properties one can show that there exists a unique
$u \in C([0,T]; L^2(\R^n))$ solving the problem \eqref{pb:IVP}, and the linear map 
\[ \mathcal{U} : f \in L^2(\R^n) \longmapsto u \in C([0, T]; L^2(\R^n)) \]
is bounded. Solutions of this type are usually referred to as \emph{physical solutions}.

In this context, we formulated in \cite{10.1063/5.0152372} the \textit{initial-to-final-state
inverse problem} aiming at deciding when is possible to determine $V = V(\tm, \x)$ from the 
\textit{initial-to-final-state map}---given by the bounded linear map defined as
\[\mathcal{U}_T : f \in L^2(\R^n) \longmapsto u(T, \centerdot) \in L^2(\R^n)\]
with $u$ denoting the solution of the problem \eqref{pb:IVP}.

The motivation of this inverse problem was to establish 
a theoretical framework that would allow to understand rigorously 
the viability of data-driven prediction in quantum mechanics.
Roughly speaking, data-driven prediction aims at computing an exact or approximative inverse of
\[ \mathcal{U} \longmapsto \mathcal{U}_T \]
ignoring the exact interaction produced by the presence of $V$.


In our previous work \cite{10.1063/5.0152372}, we established that the potential 
$V = V(\tm, \x)$ is uniquely determined by $\mathcal{U}_T$. This result, which implies 
the unique determination of $\mathcal{U}$ itself,
holds for potentials $V \in L^1((0, T); L^\infty (\R^n))$ with $n \geq 2$ that exhibit 
super-exponential decay. Specifically, we required that 
$ e^{\rho |\x|} V \in L^\infty((0, T) \times \R^n) $ for every $\rho > 0$.
The necessity of this strong decay condition stems from the application of specialized 
exponentially-growing solutions in our proof.

The main goal of this paper is to generalize this uniqueness result to the case of 
unbounded potentials. In order to state precisely 
our contributions we need to introduce some notation.
For $R > 0$, let $\mathbf{1}_{>R}$ denote the 
characteristic function of the set $\{ x \in \R^n : |x| > R \}$. For $s \in \R$ and 
$\omega \in \Sph^{n - 1}$, let $H_{s, \omega}$ denote the hyperplane
\[ H_{\omega, s} = \{ x \in \R^n : \omega \cdot x = s \}. \]

\begin{main-theorem}\label{th:uniqueness} \sl
Consider $V_1 $ and $ V_2 $ in $ L^a((0, T); L^b (\R^n))$
with $n \geq 2$ and $(a, b) \in [1, \infty] \times [1, \infty]$
satisfying
\[2 - \frac{2}{a} = \frac{n}{b} \enspace \textnormal{with} \enspace (n, a, b) \neq (2, \infty, 1).\]
For the endpoint $(a, b) = (\infty, n/2)$ with $n \geq 3$, we also suppose that
$V_1 $ and $ V_2 $ belongs to $C ([0, T]; L^{n/2}(\R^n))$.

Let $\mathcal{U}_T^1$ and $\mathcal{U}_T^2$ denote 
the initial-to-final-state map associated to the Hamiltonians
$-\Delta + V_1$ and $-\Delta + V_2$.

Assume that $V_j$  with $j \in \{ 1, 2 \}$ satisfies the following conditions:
\begin{enumerate}[label=\textnormal{(\alph*)}, ref=\alph*]
\item $e^{\rho |\x|} V_j \in L^a((0, T); L^b(\R^n))$ for all $\rho > 0$,
\label{ite:super-exp_decay_non-end}
\item there exists $R_j > 0$ such that
\[ \sup_{\omega \in \Sph^{n-1}} \int_\R \| \mathbf{1}_{>R_j} V_j \|_{L^\infty ((0,T) \times H_{\omega, s})} \, \dd s < \infty.\]
\label{ite:tempered_decay_non-end}
\end{enumerate}

Then,
\[ \mathcal{U}_T^1 = \mathcal{U}_T^2 \, \Rightarrow  \, V_1 = V_2. \] 
\end{main-theorem}

We adopt a unified approach to prove both the case $(a, b) \neq (\infty, n/2)$ and
the endpoint $(a, b) = (\infty, n/2)$ with $n\geq 3$ of the \cref{th:uniqueness}. 
At some point we need to construct a suitable approximation of the potentials. At that stage, we 
must rely on the assumed time continuity at the endpoint.
Nevertheless, this continuity in time is already required to guarantee the 
well-posedness of the initial-value problem \eqref{pb:IVP}, see the work of Ionescu and Kenig
\cite{zbMATH02204588}. An alternative to the continuity assumption, in the context of 
well-posedness, is to impose a smallness condition on 
$\| V \|_{L^\infty ((s, s+\delta); L^{n/2}(\R^n))}$ for some
$\delta > 0$ and all $s \in [0, T-\delta]$, see again \cite{zbMATH02204588}.

As noted earlier, the super-exponential decay condition in \eqref{ite:super-exp_decay_non-end} of
the \cref{th:uniqueness} was originally established in \cite{10.1063/5.0152372}.
This requirement stems from the use of a family of exponentially-growing solutions.
Building on the work of \cite{zbMATH04103631,zbMATH00554321,zbMATH02208390}, one could 
potentially relax this to a weaker exponential decay assumption.
However, it appears unlikely that any form of exponential decay can be entirely removed 
\cite{zbMATH00850580}.

In \cite{10.1088/1361-6420/ae15b8}, Cañizares \textit{et al} restricted their analysis of this
problem to Hamiltonians with time-independent electric potentials $V = V(\x)$.
In that scenario, they showed uniqueness in dimension $n \geq 2$ for bounded integrable 
potentials exhibiting only a super-linear decay at infinity.
This improvement was possible because they avoided the use of exponentially-growing solutions. 
Instead, they relied on the construction of stationary states at different 
energies with an explicit leading term plus a correction term that would vanish as the energy 
grew.

Various inverse problems for the dynamical Schrödinger equation have been previously formulated 
and analysed, see for example 
\cite{zbMATH01886353,zbMATH05549395,zbMATH05655673,zbMATH05839237,zbMATH06733553,zbMATH06864429}). 
Many of these studies even considered time-dependent Hamiltonians 
\cite{zbMATH05379127,zbMATH06516179,zbMATH06769718,zbMATH07033617,zbMATH07242805}.
Despite this extensive body of work, none have addressed an inverse problem capable of tackling 
the data-driven prediction question in quantum mechanics. Typically, these studies examine 
inverse problems where the data consists of a dynamical Dirichlet-to-Neumann map, defined on the 
boundary of a domain that confines the Hamiltonian's non-constant parts. This setup differs 
significantly from our own, as the non-constant part of our Hamiltonians is spread throughout the 
entire space. Furthermore, our approach only requires knowledge of the initial and final states 
as data.


\subsection{General scheme to prove the \cref{th:uniqueness}}
Writing $\Sigma = (0,T)\times\R^n$, we derive from the identity 
$\mathcal{U}_T^1 = \mathcal{U}_T^2$ an orthogonality relation:
\begin{equation}
\int_\Sigma (V_1 - V_2)u_1 \overline{v_2} \, = 0
\label{eq:ortho_intro}
\end{equation}
for all physical solutions $u_1$ and $v_2$
of the equations
\[ (i\partial_\tm + \Delta - V_1) u_1 = 0  \enspace \textnormal{and} \enspace (i\partial_\tm + \Delta - \overline{V_2}) v_2 = 0 \enspace \textnormal{in} \enspace \Sigma. \]

With the identity \eqref{eq:ortho_intro} in mind, we construct solutions of the form
\[ u_1 = e^{i |\nu|^2 \tm - \nu \cdot \x} (u_1^\sharp + u_1^\flat), \qquad v_2 = e^{i |\nu|^2 \tm + \nu \cdot \x} (v_2^\sharp + v_2^\flat)\]
with $\nu \in \R^n \setminus \{ 0 \}$, for the equations 
\[ (i\partial_\tm + \Delta - V_1^{\rm ext}) u_1 = 0 \enspace \textnormal{and} \enspace  (i\partial_\tm + \Delta - \overline{V_2^{\rm ext}}) v_2 = 0 \enspace \textnormal{in} \enspace \R \times \R^n, \]
where $V_j^{\rm ext}$ denotes either the extension by zero outside $\Sigma$ if 
$(a, b) \neq (\infty, n/2)$, or a suitable continuous extension with support in 
$[T^\prime ,T^{\prime \prime}] \times \R^n$
with $T^\prime < 0 < T < T^{\prime \prime}$ if $(a, b) = (\infty, n/2)$ with $n\geq 3$.

These type of exponentially-growing 
solutions are called complex geometrical optics (CGO for short) and will satisfy
\[ \lim_{|\nu| \to \infty} \Big( \Big|\int_\Sigma (V_1 - V_2)u_1^\sharp \overline{v_2^\flat}\, \Big| + \Big| \int_\Sigma (V_1 - V_2)u_1^\flat \overline{v_2^\sharp}\, \Big| + \Big| \int_\Sigma (V_1 - V_2)u_1^\flat \overline{v_2^\flat}\, \Big| \Big) = 0. \]

The functions $ u_1^\sharp $ and $ v_2^\sharp $ will be essentially chosen as $e^{-i |\eta|^2 \tm + i \eta \cdot \x}$ and $e^{-i |\kappa|^2 \tm + i \kappa \cdot \x}$ with 
$\eta, \kappa \in H_{\hat{\nu}}$ where $H_{\hat{\nu}} = H_{\hat{\nu}, 0}$ and 
$\hat{\nu} = \nu / |\nu|$. 

If plugging these CGO solutions 
into the orthogonality relation \eqref{eq:ortho_intro} were allowed, we would conclude, by making $|\nu|$ tend to 
infinity, that
\[\int_\Sigma (V_1 - V_2) e^{-i (|\eta|^2 - |\kappa|^2) \tm - i (\kappa - \eta) \cdot \x} \, = 0 \qquad \forall \eta, \kappa \in H_{\hat{\nu}} .\]
Then, making suitable choices of $\eta$ and $\kappa$, we would prove that
$ V_1 (t, x) = V_2(t, x) $ for  almost every $(t, x) \in \Sigma$.

This approach presents two non-trivial challenges: first to construct the corrections terms 
$u_1^\flat$ and $v_2^\flat$ of the CGO solutions for the Schr\"odinger equation, 
and then showing that the orthogonality relation is satisfied
for CGOs---note that their restrictions to $\Sigma$ grow exponentially 
in certain regions of the space and, in consequence, they are not physical.
These challenges where already addressed in \cite{10.1063/5.0152372} for the case of bounded
potentials, however, the extension to the unbounded case is highly non trivial.

In order to construct the correction term $u^\flat$ of a CGO in the form
\[u = e^{i|\nu|^2 \tm + \nu \cdot \x} (u^\sharp + u^\flat)\]
satisfying the Schrödinger equation 
\[(i\partial_\tm + \Delta - V) u = 0 \enspace \textnormal{in} \enspace \R \times \R^n,\]
one needs to solve
\[(i\partial_\tm + \Delta + 2 \nu \cdot \nabla - V) u^\flat = V u^\sharp \enspace \textnormal{in} \enspace \R \times \R^n.\]

As in \cite{10.1063/5.0152372}, a key step consists of proving suitable estimates for the 
Fourier multiplier $S_\nu$ defined so that if $u = S_\nu f $ with
$f \in \mathcal{S}(\R \times \R^n)$ then
\[ (i\partial_\tm + \Delta + 2 \nu \cdot \nabla) u = f \enspace \textnormal{in} \enspace \R \times \R^n. \]
In our previous work, we proved an inequality for $S_\nu$ with a gain of $|\nu|^{1/2}$ that 
allowed to deal with the bounded case, see \cite[inequality (27)]{10.1063/5.0152372}. 

In order to extend our analysis for unbounded potentials, we need to estimate 
$S_\nu$ in mixed-norm Lebesgue spaces. This requires to exploit properties of the symbol
of $S_\nu$ that were completely ignored in \cite{10.1063/5.0152372}. The types of estimates 
we need here are exact analogues of the classical Strichartz inequalities for the free 
Schrödinger propagator---including the endpoint of Keel and Tao \cite{zbMATH01215570}.

In the same way as classical Strichartz estimates, our 
family of inequalities is scale invariant and therefore it is independent of $|\nu|$.
This forces us to use not only the analogues of Strichartz estimates
but also another inequality with certain gain in $|\nu|$.


While the estimate in \cite[inequality (27)]{10.1063/5.0152372} provides a gain and is an
obvious candidate, we prove a stronger result here. Our improved version achieves the full power $|\nu|$, which seems the natural gain for this problem.

In order to be able to gather the gain in $|\nu|$ with the versatility of our Strichartz 
inequalities for $S_\nu$, we use a trick due to Lavine and Nachman.

It might be convenient for the
reader to mention here that the construction of the reminder term $u^\flat$ only requires a 
tempered decay for the potential, actually one only needs the decay stated in 
\eqref{ite:tempered_decay_non-end} of the \cref{th:uniqueness}.


The second challenge we address is extending the orthogonality relation in 
\eqref{eq:ortho_intro} to exponentially-growing solutions, as CGOs. While this 
constitutes a technical step, its core ideas are adaptations from our previous work. 

It is important to emphasize that the super-exponential decay condition
\eqref{ite:super-exp_decay_non-end} of the \cref{th:uniqueness}
is specifically required for this extension to admit CGO solutions.

\subsection{Bourgain spaces for the initial-to-final-state inverse problem}
The family of Strichartz inequalities required to deal with unbounded potentials in
our context corresponds, in the framework of the Calderón problem, to the Carleman inequality 
proved by Kenig, Ruiz and Sogge \cite{zbMATH04050093}---only for the scale-invariant situation.

Haberman showed using Stein--Tomas that
this inequality can be also retrieved as an embedding of a suitable Bourgain space 
\cite{zbMATH06490961}. This Bourgain
space also recovers the classical Sylvester--Uhlmann inequality \cite{zbMATH04015323},
which has the same role as our inequality with a gain of $|\nu|$. 

Therefore, the corresponding Bourgain space already gathers
the gain of the Sylvester--Uhlmann inequality with the versatility of Kenig--Ruiz--Sogge.
This provides an alternative path to avoid the Lavine--Nachman trick
when working on the Calderón problem. 

Unfortunately, this alternative approach
is not available for the initial-to-final-state inverse problem.
In fact, the corresponding Bourgain space in our context do not capture the mixed-norm Lebesgue 
spaces of Strichartz inequalities.

In this paper, we construct a counterexample that shows that
the corresponding Bourgain spaces is not embedded in the mixed-norm Lebesgue 
spaces appearing in the Strichartz inequalities. This highlights some obvious limitations
of Bourgain spaces in the context of the initial-to-final-state inverse problem.

\subsection{Contents} This paper consists of other six sections plus an appendix with two sections. In the \cref{sec:CGO} we construct the CGO solutions assuming several
inequalities that will be proved in the \cref{sec:boundedness,sec:BSch_op}. In the
\cref{sec:boundedness} we prove the Strichartz inequalities for $S_\nu$ as well as the
one with the full gain of $|\nu|$. The \cref{sec:BSch_op} is devoted to prove a bound
for the Birman--Schwinger operator which is key to make the Lavine--Nachman trick work.
In the \cref{sec:identical_hamiltonians} we first prove the orthogonality relation 
\eqref{eq:ortho_intro}, and then we extend it to exponentially-growing solutions. This extension
requires two results that are straightforward generalizations of our previous work, for that 
reason, these are postponed to the \cref{app:integration_by_parts,app:transference}.
The section \ref{sec:uniqueness}
is devoted to prove the \cref{th:uniqueness}.
In the \cref{sec:bourgain} we compare the inverse Calderón problem with the
initial-to-final-state inverse problem,
and we exhibit the counterexample that prevents us of using Bourgain spaces to
avoid the Lavine--Nachman trick.

\section{Complex geometrical optics solutions}\label{sec:CGO}

In this section we construct the exponentially-growing solutions that will be plugged into
the orthogonality relation \eqref{eq:ortho_intro}. We start by a preliminary discussion
that motivates the adopted approach. The arguments presented here are based on several
key inequalities that will be proved in the \cref{sec:boundedness,sec:BSch_op}.

\subsection{Preliminary discussion}\label{sec:pre-discussion}
We follow the construction of CGO 
carried out in \cite{10.1063/5.0152372}. As 
motivated in there, given a suitable potential $V$ defined in $\R \times \R^n$,
we look for solutions of
\[ 
(i\partial_\tm + \Delta - V) u = 0 \enspace \textnormal{in} \enspace \R \times \R^n, \]
in the form
\[ 
u = e^\varphi (u^\sharp + u^\flat), \]
where
\begin{equation}
\label{def:complex_phase}
\varphi(t,x) = i |\nu|^2 t + \nu \cdot x \qquad \forall (t, x) \in \R \times \R^n
\end{equation}
with $\nu \in \R^n \setminus \{ 0 \}$, and
$u^\sharp$ chosen so that $e^\varphi u^\sharp$ is a solution of 
$(i\partial_\tm + \Delta) (e^\varphi u^\sharp) = 0$ in $\R \times \R^n$,
or equivalently $(i\partial_\tm + \Delta + 2 \nu \cdot \nabla) u^\sharp = 0 $ in
$\R \times \R^n$. These choices force $u^\flat$ to satisfy
\begin{equation}
(i\partial_\tm + \Delta + 2 \nu \cdot \nabla - V) u^\flat = V u^\sharp \enspace \textnormal{in} \enspace \R \times \R^n.
\label{eq:remainder-equation}
\end{equation}

At this point, we will guess the mixed-norm Lebesgue spaces where the potential $V$
might belong to, in order for us, to construct $u^\flat$ solving 
\eqref{eq:remainder-equation}. To do so, we need to introduce the Fourier multiplier 
$S_\nu$ such that if $u = S_\nu f $ with
$f \in \mathcal{S}(\R \times \R^n)$ then
\begin{equation}
\label{eq:constant_coefficients}
(i\partial_\tm + \Delta + 2 \nu \cdot \nabla) u = f \enspace \textnormal{in} \enspace \R \times \R^n.
\end{equation}

We could construct $u^\flat$ as a solution for the equation
\begin{equation}
\label{eq:Neumann}
(\Id - S_\nu \circ M_V ) u^\flat =  S_\nu (V u^\sharp),
\end{equation}
where $M_V $ denotes the operator 
\[ M_V : u \longmapsto V u \]
---since applying the differential 
operator $i\partial_\tm + \Delta + 2 \nu \cdot \nabla$ to both sides of the equation 
\eqref{eq:Neumann} we would see that $u^\flat$ satisfies the identity 
\eqref{eq:remainder-equation}.

A way to solve \eqref{eq:Neumann} consists in finding 
a Banach space where the operator $S_\nu \circ M_V$ is a contraction. This drives us to 
the need of understanding the boundedness of $S_\nu$ in spaces that capture the behaviour
of $V$ that we want to allow.

In the \cref{sec:boundedness}, we prove that, for
$n \in \N$ and $(q, r) \in [1, 2] \times [1, 2]$ satisfying
\begin{equation}
\label{id:Strichartz_indeces}
2 - \frac{2}{q} = \frac{n}{r} - \frac{n}{2} \enspace \textnormal{with} \enspace (n, q, r) \neq (2, 2, 1),
\end{equation}
the Fourier multiplier $S_\nu$ is a bounded operator from the mixed-norm Lebesgue space
$L^q (\R; L^r (\R^n))$ to 
$L^{q^\prime} (\R; L^{r^\prime} (\R^n))$ with a norm that is less or equal than a 
constant $C$ that only depends on $n$ and $(q, r)$---see the \cref{th:non-gain_Snu}.

Here $(q^\prime, r^\prime) \in [2, \infty] \times [2, \infty]$ is the pair 
the conjugate exponents of $q$ and $r$, and satisfies the relation
\begin{equation}
\label{id:Strichartz_dual-indeces}
\frac{2}{q^\prime} = \frac{n}{2} - \frac{n}{r^\prime} \enspace \textnormal{with} \enspace (n, q^\prime, r^\prime) \neq (2, 2, \infty).
\end{equation}

Consequently, the operator $S_\nu \circ M_V$ is bounded on $L^{q^\prime} (\mathbb{R}; L^{r^\prime} (\mathbb{R}^n))$, provided that $M_V$ is bounded from $L^{q^\prime} (\mathbb{R}; L^{r^\prime} (\mathbb{R}^n))$ to $L^{q} (\mathbb{R}; L^{r} (\mathbb{R}^n))$. Then, if
$V \in L^a (\R; L^b (\R^n))$ with 
$(a, b) \in [1, \infty] \times [1, \infty]$ satisfying
\begin{equation}
\label{id:indices4Vab}
\frac{1}{q} - \frac{1}{q^\prime} = \frac{1}{a} \enspace \textnormal{and} \enspace \frac{1}{r} - \frac{1}{r^\prime} = \frac{1}{b}
\end{equation}
one can check---by Hölder's inequality---that $M_V$ is bounded from
$L^{q^\prime} (\R; L^{r^\prime} (\R^n))$ to $L^q (\R; L^r (\R^n))$.

Additionally, one can see
that such $(a, b) \in [1, \infty] \times [1, \infty]$ has to satisfy the relation
\begin{equation}
\label{cond:V_ab} 
2 - \frac{2}{a} = \frac{n}{b} \enspace \textnormal{with} \enspace (n, a, b) \neq \big(2, \infty, 1 \big).
\end{equation}
%

Furthermore, if its norm 
\[ \| M_V \|_{\mathcal{L}(L^{q^\prime} (\R; L^{r^\prime} (\R^n)); L^q (\R; L^r (\R^n)))} = \|  V \|_{L^a (\R; L^b (\R^n))} < 1/C\]
---with $C$ an upper bound for the norm of $S_\nu$,
then $S_\nu \circ M_V$ is a contraction 
on the space $L^{q^\prime} (\R; L^{r^\prime} (\R^n))$. Hence,
we can construct an inverse for
$\Id - S_\nu \circ M_V$.

However, in order to solve \eqref{eq:Neumann} we need to ensure
that $S_\nu (V u^\sharp) \in L^{q^\prime} (\R; L^{r^\prime} (\R^n))$. By the boundedness
of $S_\nu$ from $L^q (\R; L^r (\R^n))$ to $L^{q^\prime} (\R; L^{r^\prime} (\R^n))$,
we need to guarantee that $V u^\sharp \in L^q (\R; L^r (\R^n))$.

In our previous work \cite{10.1063/5.0152372}, $u^\sharp$ was chosen so that
$(i\partial_\tm + \Delta) u^\sharp = 0 $ and $ \nu \cdot \nabla  u^\sharp = 0$ in
$\R \times \R^n$. Assuming that $V$ is supported in $\overline{\Sigma}$ and it has an extra 
decay, we have by H\"older's inequality that
\begin{align*}
\| V u^\sharp \|_{L^q (\R; L^r (\R^n))} & \leq \| \langle \hat{\nu} \cdot \x \rangle^\rho V \|_{L^a (\R; L^b (\R^n))} 
\| \langle \hat{\nu} \cdot \x \rangle^{-\rho} u^\sharp \|_{L^{q^\prime} ((0,T); L^{r^\prime} (\R^{n}))}\\
& \leq \| \langle \hat{\nu} \cdot \x \rangle^\rho V \|_{L^a (\R; L^b (\R^n))} 
\| \langle \centerdot \rangle^{-\rho} \|_{L^{r^\prime} (\R)} \| u^\sharp \|_{L^{q^\prime} ((0,T); L^{r^\prime} (H_{\hat{\nu}}))},
\end{align*}
where
$\langle \hat{\nu} \cdot \x \rangle = (1 + |\hat{\nu} \cdot \x|^2)^{1/2}$ and 
$\rho > 1/r^\prime$.

A weight as $\langle \hat{\nu} \cdot \x \rangle^{-\rho}$ has to be introduced because
$u^\sharp$ does not depend on $ \hat{\nu} \cdot \x $. Moreover, this is enough
since $u^\sharp$ can be chosen to have this behaviour on $(0, T) \times H_{\hat{\nu}}$---for 
further details see the proof of the \cref{th:CGO}.
Thus, we would have 
$V u^\sharp \in L^q (\R; L^r (\R^n))$.

This argument would allows us to 
solve \eqref{eq:Neumann} under the conditions
\[\sup_{\omega \in \Sph^{n-1}}\| \langle \omega \cdot \x \rangle^\rho V \|_{L^a (\R; L^b (\R^n))} < \infty,\]
and $\|  V \|_{L^a (\R; L^b (\R^n))}$ is sufficiently small.

Fortunately, using a trick that involves the Birman--Schwinger operator we will be able to
solve \eqref{eq:Neumann} for potentials $V$ supported in $\overline{\Sigma}$ so that
$\|  V \|_{L^a (\R; L^b (\R^n))}$ is arbitrary and
$\sup_{\omega \in \Sph^{n-1}} \| \mathbf{1}_{>R} \langle \omega \cdot \x \rangle^\rho V \|_{L^\infty (\R \times \R^n)} < \infty$
for some $R>0$ and $\rho>1$.

In order to make this trick work, we will need the boundedness properties of $S_\nu$ stated in
the \cref{th:non-gain_Snu} as well as the one in the \cref{th:gain_Snu}.

\subsection{The Lavine--Nachman trick}
Assume for a moment that we have chosen $u^\sharp$ so that 
$(i\partial_\tm + \Delta + 2 \nu \cdot \nabla) u^\sharp = 0 $ 
in $\R \times \R^n$ with $\nu \in \R^n \setminus\{ 0 \}$, we look for $u^\flat$ solving
\begin{equation}
\label{eq:remainder-equation_LNtrick}
(i\partial_\tm + \Delta + 2 \nu \cdot \nabla - M_V) u^\flat = M_V u^\sharp \enspace \textnormal{in} \enspace \R \times \R^n.
\end{equation}
In order to solve this equation, we will use a trick due to Lavine and Nachman 
that we learnt in \cite{zbMATH06156446}.

Start by observing that
\[ (i\partial_\tm + \Delta + 2 \nu \cdot \nabla - M_V) \circ S_\nu \circ M_{|V|^{1/2}} = M_{|V|^{1/2}} - M_V \circ S_\nu \circ M_{|V|^{1/2}}. \]
Then, defining
\begin{equation}
\label{def:W}
W: (t,x) \in \R \times \R^n \longmapsto \left\{
\begin{aligned}
& \frac{V(t,x)}{|V(t,x)|^{1/2}} & V(t,x)\neq 0\\
& \qquad 0 & V(t,x)=0
\end{aligned}
\right. 
\end{equation}
and using that $|V|^{1/2} = |W|$ and $V = |W| W$, we can write
\[ (i\partial_\tm + \Delta + 2 \nu \cdot \nabla - M_V) \circ S_\nu \circ M_{|V|^{1/2}} = M_{|W|} \circ ( \Id - M_W \circ S_\nu \circ M_{|W|}). \]

Then, $u^\flat$ in the form $u^\flat = S_\nu (|W| v)$ with $v$ satisfying
\begin{equation}
\label{eq:v_L^2}
( \Id - M_W \circ S_\nu \circ M_{|W|}) v = M_W u^\sharp
\end{equation}
solves the equation \eqref{eq:remainder-equation_LNtrick}. Indeed,
\[ (i\partial_\tm + \Delta + 2 \nu \cdot \nabla - M_V) S_\nu (|W| v) = M_{|W|} \circ ( \Id - M_W \circ S_\nu \circ M_{|W|}) v = M_{|W|} \circ M_W u^\sharp. \]
Since $M_{|W|} \circ M_W = M_V$, we have that $u^\flat = S_\nu (|W| v)$ solves 
\eqref{eq:remainder-equation_LNtrick}.

Therefore, our goal is to find $v$ solving \eqref{eq:v_L^2}. To that end, we 
will show that the Birman--Schwinger operator 
$M_W \circ S_\nu \circ M_{|W|}$ is a contraction in $L^2(\R \times \R^n)$ whenever $|\nu|$ is
sufficiently large. In the \cref{prop:MSM}, we provide conditions on $W$ that guarantee that
\[ \lim_{|\nu| \to \infty} \| M_W \circ S_\nu \circ M_{|W|} \|_{\mathcal{L}[L^2(\R \times \R^n)]} = 0. \]

Thus, for $|\nu|$ sufficiently large, the operator $\Id - M_W \circ S_\nu \circ M_{|W|}$
has a bounded inverse in $L^2(\R \times \R^n)$, and $v$ can be chosen as
\[v = ( \Id - M_W \circ S_\nu \circ M_{|W|})^{-1} (W u^\sharp). \]

In the rest of the section we state the conditions on $V$ that make it possible,
first, to invert $\Id - M_W \circ S_\nu \circ M_{|W|}$ and, second, to construct
$u^\flat$.

\begin{proposition}\label{prop:LNtrick}\sl
Consider $S, T \in \R$ such that $S<T$ and $V \in L^a(\R; L^b (\R^n))$ with
$\supp V \subset [S, T]\times \R^n$ and $(a, b) \in [1, \infty] \times [1, \infty]$ 
satisfying \eqref{cond:V_ab}.

Assume that there is $R > 0$ so that
\begin{equation}
\label{cond:decay_lines_for_V}
\sup_{\omega \in \Sph^{n-1}} \int_\R \| \mathbf{1}_{>R} V \|_{L^\infty (\R \times H_{\omega, s})} \, \dd s < \infty,
\end{equation}
and $\mathbf{1}_{\leq R} V \in C(\R; L^{n/2} (\R^n))$ if $(a, b) = (\infty, n/2)$ for
$n\geq 3$.

Then, there exists a constant $c > 0$ that only depends on $n$, $(a,b)$ and $R$ as well as on
$\| V \|_{L^a (\R; L^b(\R^n))}$ and the quantity in \eqref{cond:decay_lines_for_V}, such that 
for every $\nu \in \R^n$ with $|\nu| \geq c$ the operator
\[ \Id - M_W \circ S_\nu \circ M_{|W|},\]
with $W$ as in \eqref{def:W}, has a bounded inverse in $L^2(\R\times \R^n)$.
\end{proposition}

Before proving this proposition,
we note that while $|W|^2 \in L^a(\R; L^b (\R^n))$, 
we must also ensure that, at the endpoint, the time continuity
of $V$ is inherited by $W$.

\begin{lemma}\sl
Consider $V \in L^\infty (\R; L^{n/2} (\R^n))$ such that 
$ V \in C(\R; L^{n/2} (\R^n))$. Then, $ W $ defined as in \eqref{def:W} belongs to 
$ C(\R; L^n (\R^n))$.
\end{lemma}

\begin{proof}
Consider $s \in \R$ such that $W(s, \centerdot) \neq 0$. Then, 
$V(t, \centerdot) \neq 0$ for all $t$ sufficiently close to $s$.
By \eqref{def:W}, we have that $W(t, \centerdot) \neq 0$
for all $t$ sufficiently close to $s$.
Then, for almost every $x \in \R^n$ we have that
\begin{align*}
W (t, & x) - W(s, x) = \frac{V(t,x)}{|V(t,x)|} |V(t,x)|^{1/2} - \frac{V(s,x)}{|V(s,x)|} |V(s,x)|^{1/2} \\
& = \Big( \frac{V(t,x)}{|V(t,x)|} - \frac{V(s,x)}{|V(s,x)|} \Big)|V(t,x)|^{1/2} + \frac{V(s,x)}{|V(s,x)|} \big(|V(t,x)|^{1/2} - |V(s,x)|^{1/2}\big).
\end{align*}
Hence,
\begin{align*}
\| W (t,  \centerdot) - W(s, \centerdot) \|_{L^n(\R^n)} & \leq \Big\| \Big| \frac{V(t,\centerdot)}{|V(t,\centerdot)|} - \frac{V(s,\centerdot)}{|V(s,\centerdot)|} \Big|^2 |V(t,\centerdot)| \Big\|_{L^{n/2}(\R^n)}^{1/2} \\
& \qquad + \big\| \big||V(t,\centerdot)|^{1/2} - |V(s,\centerdot)|^{1/2}\big|^2 \big\|_{L^{n/2}(\R^n)}^{1/2}.
\end{align*}

On the one hand, observe that for almost every $x \in \R^n$ we have
\begin{align*}
& \Big| \frac{V(t,x)}{|V(t,x)|} - \frac{V(s,x)}{|V(s,x)|} \Big|^2 |V(t,x)| \leq 2 \Big| \frac{V(t,x)}{|V(t,x)|} - \frac{V(s,x)}{|V(s,x)|} \Big| |V(t,x)| \\
& \leq 2\big( |V(t,x) - V(s,x)| + \big| |V(s,x)| - |V(t,x)| \big| \big) \leq 4 |V(t,x) - V(s,x)|
\end{align*}
for almost every $x \in \R^n$. Hence,
\[ \Big\| \Big| \frac{V(t,\centerdot)}{|V(t,\centerdot)|} - \frac{V(s,\centerdot)}{|V(s,\centerdot)|} \Big|^2 |V(t,\centerdot)| \Big\|_{L^{n/2}(\R^n)}^{1/2} \leq 2  \| V(t,\centerdot) - V(s,\centerdot) \|_{L^{n/2}(\R^n)}^{1/2} \]

On the other hand,
since $|a - b|^2 \leq |a^2 - b^2|$ for all positive constants $a$ and $b$, we have
\begin{align*}
\big||V(t,x)|^{1/2} - |V(s,x)|^{1/2}\big|^2 \leq \big| |V(s,x)| - |V(t,x)| \big| \leq |V(t,x) - V(s,x)|.
\end{align*}
Hence
\[ \big\| \big||V(t,\centerdot)|^{1/2} - |V(s,\centerdot)|^{1/2}\big|^2 \big\|_{L^{n/2}(\R^n)}^{1/2} \leq \| V(t,\centerdot) - V(s,\centerdot) \|_{L^{n/2}(\R^n)}^{1/2}. \]

The inequality $|a - b|^2 \leq |a^2 - b^2|$ follows from
\[ (a - b)^2 + b^2 \leq (a - b)^2 + b^2 + 2 (a - b)b = a^2 \]
in the case $a\geq b \geq 0$.

Gathering these inequalities we have that
\[\| W (t,  \centerdot) - W(s, \centerdot) \|_{L^n(\R^n)} \leq 3 \| V(t,\centerdot) - V(s,\centerdot) \|_{L^{n/2}(\R^n)}^{1/2}, \]
which means that $t \in \R \mapsto W(t, \centerdot) \in L^n(\R^n) $ is continuous at $s$,
by the continuity of $t \in \R \mapsto V(t, \centerdot) \in L^{n/2}(\R^n) $.

In order to prove the lemma, it remains to see what happens in the case that 
$W(s, \centerdot) = 0$. Note that
\begin{align*}
\| W (t,  \centerdot) - W(s, \centerdot) \|_{L^n(\R^n)} &= \| W (t,  \centerdot) \|_{L^n(\R^n)} = \| V (t,  \centerdot) \|_{L^{n/2}(\R^n)}^{1/2} \\
&= \| V (t,  \centerdot) - V(s, \centerdot) \|_{L^{n/2}(\R^n)}^{1/2}
\end{align*}
since $V(s, \centerdot) = 0$. Hence, again by the continuity of 
$t \in \R \mapsto V(t, \centerdot) \in L^{n/2}(\R^n) $, we can conclude that 
$t \in \R \mapsto W(t, \centerdot) \in L^n(\R^n) $ is 
continuous at $s$ even when $W(s, \centerdot) = 0$.
\end{proof}

\begin{proof}[Proof of the \cref{prop:LNtrick}]
The fact $ \Id - M_W \circ S_\nu \circ M_{|W|}$
has a bounded inverse in $L^2(\R\times \R^n)$ for every $\nu \in \R^n$ with $|\nu|$ 
sufficiently large, is a consequence of
the Neumann-series theorem together with the fact that the Birman--Schwinger operator 
$M_W \circ S_\nu \circ M_{|W|}$ is a contraction in $L^2(\R \times \R^n)$---see the
\cref{prop:MSM}.
\end{proof}

We finish this section by stating the theorem which establishes the existence and properties
of the CGO solutions $u = e^\varphi (u^\sharp + u^\flat)$.
In order to present our choice for $u^\sharp$, it is convenient make some comments.

For $s \in \R$ and 
$\omega \in \Sph^{n - 1}$, the hyperplane 
$ H_{\omega, s} = \{ x \in \R^n : \omega \cdot x = s \} $
is endowed with the measure $\sigma_{\omega, s}$. This measure is
defined as the push-forward of the 
Lebesgue measure in $\R^{n-1}$ via the map
\[ y^\prime \in \R^{n-1} \longmapsto y=(y^\prime, s) \in \R^{n-1} \times \{ s \} \longmapsto Q y \in H_{\omega, s}, \]
were $Q$ is any matrix in $\Orth(n)$ such that $\omega = Q e_n$ with 
$\{ e_1, \dots , e_n \}$ the standard basis of $\R^n$. The definition of 
$\sigma_{\omega, s}$ is 
independent of the choice of $Q$. We then denote by $L^p(H_{\omega, s})$, 
for $p \in [1, \infty]$, the usual Lebesgue spaces on $H_{\omega, s}$ with respect to the 
measure $\sigma_{\omega, s}$.

Now we introduce our choice of $u^\sharp$. For $\psi \in \mathcal{S} (H_{\hat{\nu}})$, we define
\begin{equation}
\label{id:ushrap}
u^\sharp (t, x) = \frac{1}{(2\pi)^\frac{n-1}{2}} \int_{H_{\hat{\nu}}} e^{ix \cdot \xi} e^{-i t |\xi|^2} {\psi}(\xi) \, \dd \sigma_{\hat{\nu}} (\xi) \quad \forall (t, x) \in \R \times \R^n,
\end{equation}
where $\sigma_{\hat{\nu}} = \sigma_{\hat{\nu}, 0}$.

\begin{theorem}\label{th:CGO}\sl
Consider $T>0$ and $V \in L^a((0,T); L^b (\R^n))$ with $(a, b) \in [1, \infty] \times [1, \infty]$ 
satisfying \eqref{cond:V_ab}.

Assume that there is $R > 0$ so that
\begin{equation}
\label{cond:decay_lines_for_V_(0,T)}
\sup_{\omega \in \Sph^{n-1}} \int_\R \| \mathbf{1}_{>R} V \|_{L^\infty ((0,T) \times H_{\omega, s})} \, \dd s < \infty,
\end{equation}
and $\mathbf{1}_{\leq R} V \in C([0,T]; L^{n/2} (\R^n))$ if $(a, b) = (\infty, n/2)$ for
$n\geq 3$.

For $\nu \in \R^n \setminus \{ 0 \}$ and $\psi \in \mathcal{S}(H_{\hat{\nu}})$,
consider $\varphi$ as in \eqref{def:complex_phase} and $u^\sharp$ as in \eqref{id:ushrap}.

Then, there is a positive 
constant $c > 0$ that only depends on $n$, $(a,b)$ and $R$ as well as on
$\| V \|_{L^a ((0,T); L^b(\R^n))}$ and the quantity in
\eqref{cond:decay_lines_for_V_(0,T)}, such that for every $\nu \in \R^n$, with 
$|\nu| \geq c$, and every $\psi \in \mathcal{S}(H_{\hat{\nu}})$
there exists $u^\flat \in L^{q^\prime} (\R; L^{r^\prime} (\R^n))$ so that
\[ u = e^\varphi (u^\sharp + u^\flat) \]
solves the equation
\[ i\partial_\tm u = - \Delta u + V^{\rm ext} u \enspace \textnormal{in} \enspace \R \times \R^n. \]
Here $(q^\prime, r^\prime) \in [2, \infty] \times [2, \infty]$
satisfying
\begin{equation}
\label{id:indicesVprimes}
1 - \frac{2}{q^\prime} = \frac{1}{a} \enspace \textnormal{and} \enspace 1 - \frac{2}{r^\prime} = \frac{1}{b},
\end{equation}
and $V^{\rm ext}$ denotes either the extension by zero outside $\Sigma$ if 
$(a, b) \neq (\infty, n/2)$, or a suitable continuous extension with support in 
$[T^\prime ,T^{\prime \prime}] \times \R^n$
with $T^\prime < 0 < T < T^{\prime \prime}$ if $(a, b) = (\infty, n/2)$ with $n\geq 3$.

Additionally, for any measurable function $\tilde{W}$ in $\R\times\R^n$ supported in 
$[T^\prime ,T^{\prime \prime}] \times \R^n$ such that $|\tilde{W}|^2 \in L^a(\R; L^b (\R^n))$,
\begin{equation}
\label{cond:decay_lines_for_Wtilde}
\sup_{\omega \in \Sph^{n-1}} \int_\R \| \mathbf{1}_{>R} |\tilde W|^2 \|_{L^\infty (\R \times H_{\omega, s})} \, \dd s < \infty
\end{equation}
and $\mathbf{1}_{\leq R} |\tilde W|^2 \in C(\R; L^{n/2} (\R^n))$ in the endpoint 
$(a, b) = (\infty, n/2)$ for $n \geq 3$,
we have that
\begin{equation}
\label{boun:CGO-leading}
\| \tilde{W} u^\sharp \|_{L^2(\R\times \R^n)} \leq C \| {\psi} \|_{L^2(H_{\hat{\nu}})},
\end{equation}
and
\begin{equation}
\label{lim:CGO-remainder}
\lim_{|\nu| \to \infty} \| \tilde{W} u^\flat \|_{L^2(\R\times \R^n)} = 0.
\end{equation}

The constant $C > 0$ of the inequality \eqref{boun:CGO-leading} 
only depends on $n$, $(a,b)$, $T$ and $R$, as well as on
$\| |\tilde W|^2 \|_{L^a (\R; L^b(\R^n))}$ and the quantity 
in \eqref{cond:decay_lines_for_Wtilde}.
\end{theorem}

\begin{proof}
Start by proving the inequality \eqref{boun:CGO-leading}, which in particular
shows that $ W u^\sharp \in L^2(\R \times \R^n)$ with $W$ as in \eqref{def:W} so that 
$|W| W = V^{\rm ext}$. We have that
\begin{align*}
\| \tilde W u^\sharp \|_{L^2(\R \times \R^n)} \leq \| \mathbf{1}_{\leq R} \tilde W u^\sharp \|_{L^2(\R \times \R^n)} + \| \mathbf{1}_{> R} \tilde W u^\sharp \|_{L^2(\R \times \R^n)},
\end{align*}
with $\mathbf{1}_{\leq R}$ denoting the characteristic 
function of $\{ x \in \R^n : |x| \leq R \}$. 

By Hölder's inequality we have that
\begin{align*}
& \| \mathbf{1}_{\leq R} \tilde W u^\sharp \|_{L^2(\R \times \R^n)} \leq (2R)^{1/r^\prime} \| |\tilde W|^2 \|_{L^a(\R; L^b (\R^n))}^{1/2} \| u^\sharp \|_{L^{q^\prime}((T^\prime,T^{\prime \prime}); L^{r^\prime} (H_{\hat{\nu}}))}, \\
& \| \mathbf{1}_{> R} \tilde W u^\sharp \|_{L^2(\R \times \R^n)} \leq S^{1/2} \Big( \int_\R \|\mathbf{1}_{> R} |\tilde W|^2 \|_{L^\infty (\R \times H_{\hat{\nu}, s})} \, \dd s \Big)^{1/2} \sup_{t \in \R} \| u^\sharp(t, \centerdot) \|_{L^2 (H_{\hat{\nu}})}.
\end{align*}
Here and below $S $ stands for $ T^{\prime \prime} -T^\prime$.

Note that there is $s \in [q^\prime, \infty]$ so that $(s, r^\prime)$ is an admissible Strichartz pair for dimension $n- 1$:
\[\frac{2}{s} = \frac{n-1}{2} - \frac{n-1}{r^\prime}  \enspace \textnormal{with} \enspace (n-1, s, r^\prime)\neq (2, 2, \infty).\]
Then, by Hölder's and Strichartz's inequalities, we have that
\[ \| u^\sharp \|_{L^{q^\prime}((T^\prime,T^{\prime\prime}); L^{r^\prime} (H_{\hat{\nu}}))} \leq S^{1/q^\prime - 1/s} \| u^\sharp \|_{L^s(\R; L^{r^\prime} (H_{\hat{\nu}}))} \lesssim S^{1/q^\prime - 1/s} \| {\psi} \|_{L^2(H_{\hat{\nu}})}. \]

Furthermore, by Plancherel's identity we have that
\[\sup_{t \in \R} \| u^\sharp(t, \centerdot) \|_{L^2 (H_{\hat{\nu}})} \leq \| {\psi} \|_{L^2(H_{\hat{\nu}})}.\]
This concludes the proof \eqref{boun:CGO-leading}.

The \cref{prop:LNtrick} justifies the existence and boundedness of the operator
$(\Id - M_W \circ S_\nu \circ M_{|W|})^{-1}$ in $L^2(\R \times \R^n)$.

Moreover, for $v \in L^2(\R \times \R^n)$, we have that $|W| v \in L^q (\R; L^r (\R^n))$, since Hölder's inequality implies
\[\| |W| v \|_{L^q(\R; L^r (\R^n))} \leq \| |W|^2 \|_{L^a(\R; L^b (\R^n))}^{1/2} \| v \|_{L^2(\R \times \R^n)}\]
with
$(q, r) \in [1, 2] \times [1, 2]$ and $(a, b) \in [1, \infty] \times [1, \infty]$ 
satisfying the relation 
\begin{equation}
\label{id:indices4Vqr}
\frac{2}{q} - 1 = \frac{1}{a} \enspace \textnormal{and} \enspace  \frac{2}{r} - 1 = \frac{1}{b}
\end{equation}
or equivalently \eqref{id:indices4Vab}.

Then, since 
$ W u^\sharp \in L^2(\R \times \R^n)$ and the Fourier multiplier
$S_\nu$ is a bounded operator from 
$L^q (\R; L^r (\R^n))$ to $L^{q^\prime} (\R; L^{r^\prime} (\R^n))$, we have that
\[ u^\flat = S_\nu \circ M_{|W|} \circ ( \Id - M_W \circ S_\nu \circ M_{|W|})^{-1} (W u^\sharp) \]
belongs to $ L^{q^\prime} (\R; L^{r^\prime} (\R^n))$ and is solution of 
\eqref{eq:remainder-equation_LNtrick}.
On the one hand, the fact that $u^\flat$ is solutions of
\eqref{eq:remainder-equation_LNtrick} was discussed when
explaining the Lavine--Nachman trick. On the other hand, $u^\flat $ belongs to $ L^{q^\prime} (\R; L^{r^\prime} (\R^n))$ since naming
\begin{equation}
\label{id:v_LNtrick}
v = ( \Id - M_W \circ S_\nu \circ M_{|W|})^{-1} (W u^\sharp) \in L^2(\R \times \R^n),
\end{equation}
we have that $u^\flat = S_\nu (|W| v)$.
This concludes the proof that $u^\flat \in L^{q^\prime} (\R; L^{r^\prime} (\R^n))$.

Additionally, the limit 
\eqref{lim:CGO-remainder} holds thanks to the bound of the operator 
$M_{\tilde{W}} \circ S_\nu \circ M_{|W|}$ given in the \cref{prop:MSM},
since $\tilde{W} u^\flat = M_{\tilde{W}} \circ S_\nu \circ M_{|W|} v$ with
$v$ as in \eqref{id:v_LNtrick} satisfying
\[ \| v \|_{L^2(\R \times \R^n)} \lesssim \| W u^\flat \|_{L^2(\R \times \R^n)} \lesssim \| {\psi} \|_{L^2(H_{\hat{\nu}})}.\]

Last chain of inequalities holds because of the boundedness of
$(\Id - M_W \circ S_\nu \circ M_{|W|})^{-1}$ in $L^2(\R \times \R^n)$, and
\eqref{boun:CGO-leading}.
\end{proof}

\section{Boundedness properties of $S_\nu$}\label{sec:boundedness}
Here we focus on inverting the constant-coefficient differential operator 
\begin{equation}
\label{term:constant-coefficient_operator}
i\partial_\tm + \Delta + 2 \nu \cdot \nabla,
\end{equation}
where $\nu \in \R^n \setminus \{ 0 \}$,
and finding bounds for its inverse on different spaces. The symbol of 
\eqref{term:constant-coefficient_operator} is the polynomial function
\begin{equation}
\label{id:symbol}
p_\nu (\tau, \xi) = - \tau - |\xi|^2 + i2 \nu \cdot \xi \qquad \forall (\tau, \xi) \in \R \times \R^n.
\end{equation}

For $\nu \in \R^n \setminus \{ 0 \}$, introduce the set 
\[\Gamma_\nu = \{ (\tau, \xi) \in \R \times \R^n : p_\nu (\tau, \xi) = 0 \}.\]
Then, the function 
\[(\tau, \xi) \in (\R \times \R^n) \setminus \Gamma_{\nu} \longmapsto \frac{1}{p_\nu (\tau, \xi)} \in \C\]
can be extended to $\R \times \R^n$ as a locally integrable function. 

For 
$f \in \mathcal{S}(\R \times \R^n)$, we define
\begin{equation}
\label{id:multiplier_definition}
S_\nu f (t, x) = \frac{1}{(2\pi)^\frac{n+1}{2}} \int_{\R \times \R^n} e^{i (t \tau + x \cdot \xi)} \frac{1}{p_\nu (\tau, \xi)} \widehat{f}(\tau, \xi) \, \dd (\tau, \xi)
\end{equation}
for all $(t, x) \in \R \times \R^n$.

It might be convenient for the reader to mention at this point that we choose to define the
Fourier transform in $\R^d$, with $d \in \N$, so that for functions $f \in \mathcal{S}(\R^d)$,
it reads as 
\[\widehat{f} (\xi) = \frac{1}{(2\pi)^{d/2}} \int_{\R^d} e^{-i\xi \cdot x} f(x) \, \dd x,\]
while its inverse is given by
\[\widecheck{f} (\xi) = \frac{1}{(2\pi)^{d/2}} \int_{\R^d} e^{i\xi \cdot x} f(x) \, \dd x.\]

The goal of this section is to prove the following inequalities:

\begin{theorem}\label{th:gain_Snu}\sl
Consider $n \in \N$. There exists an absolute constant $C > 0$ such that
\[ \sup_{s \in \R} \| S_\nu f \|_{L^2(\R \times H_{\hat{\nu}, s})} \leq \frac{C}{|\nu|} \int_\R \| f \|_{L^2(\R \times  H_{\hat{\nu}, s})} \, \dd s \]
for all $f \in \mathcal{S}(\R \times \R^n)$.
\end{theorem}

\begin{theorem}\label{th:non-gain_Snu} \sl
Consider $n \in \N$ and $(q, r) \in [1, 2] \times [1, 2]$ satisfying \eqref{id:Strichartz_indeces}.
There exists a constant $C > 0$ that only depends on $n$, $q$ and $r$ such that
\[ \| S_\nu f \|_{L^{q^\prime} (\R; L^{r\prime}(\R^n))} \leq C \| f \|_{L^q (\R; L^r(\R^n))} \]
for all $f \in \mathcal{S}(\R \times \R^n)$.
\end{theorem}

The first inequality recalls a local smoothing, while the second is the counterpart of the
classical Strichartz inequalities.

Since the operator $S_\nu$ has certain {homogeneity}, we prove the 
\cref{th:gain_Snu,th:non-gain_Snu}
studying a {normalized} version of it.

Consider the polynomial function
\begin{equation}
\label{id:normalized_symbol}
p(\tau, \xi) = \tau - |\xi|^2 + i\xi_n \quad \forall\, (\tau, \xi) \in \R \times \R^n,
\end{equation}
where $\xi_n = \xi \cdot e_n$ and $\{e_1,\ldots,e_n\}$ is
the standard basis of $\R^n$.

Setting
\[\Gamma = \{ (\tau, \xi) \in \R \times \R^n : p (\tau, \xi) = 0 \},\]
the function 
\[(\tau, \xi) \in (\R \times \R^n) \setminus \Gamma \longmapsto \frac{1}{p (\tau, \xi)} \in \C\]
can be extended to $\R \times \R^n$ as a locally integrable function. 

For 
$f \in \mathcal{S}(\R \times \R^n)$, we define
\begin{equation}
\label{def:S}
S f (t, x) = \frac{1}{(2\pi)^\frac{n+1}{2}} \int_{\R \times \R^n} e^{i (t \tau + x \cdot \xi)} \frac{1}{p (\tau, \xi)} \widehat{f}(\tau, \xi) \, \dd (\tau, \xi) \quad \forall (t, x) \in \R \times \R^n.
\end{equation}

In order to relate the operators $S$ and $S_\nu$, we choose 
$Q \in \Orth(n)$ so that $\nu = |\nu| Q e_n$ and make the change
$ \sigma = -|\nu|^2 4 \tau$ and $ \eta = 2 |\nu| Q \xi$ that implies
\begin{equation}
\label{id:rescaling}
p_\nu (\sigma, \eta) = 4 |\nu|^2 p(\tau, \xi).
\end{equation}

One can check that if $f \in \mathcal{S}(\R \times \R^n)$ and we define
\begin{equation}
\label{id:scaling_function}
g (s, y) = f \Big(-\frac{s}{4 |\nu|^2}, \frac{Q y}{2 |\nu|} \Big) \quad \forall (s, y) \R \times \R^n,
\end{equation}
then
\[ \widehat{g}(\tau, \xi) = (2 |\nu|)^{n+2} \widehat{f} (- 4 |\nu|^2 \tau, 2 |\nu| Q \xi) \qquad \forall (\tau, \xi) \in \R \times \R^n, \]
and consequently
\begin{equation}
\label{id:scaling_operator}
S_\nu f (s, y) = \frac{1}{4|\nu|^2} S g (-4 |\nu|^2 s, 2 |\nu| Q^\T y) \quad \forall (s, y) \R \times \R^n.
\end{equation}

\subsection{Proof of the \cref{th:gain_Snu}}
After the relations \eqref{id:scaling_operator} and \eqref{id:scaling_function}, it is a 
simple task to check that
\begin{align*}
& \sup_{s \in \R} \| S_\nu f \|_{L^2(\R \times H_{\hat{\nu}, s})} = \frac{1}{4|\nu|^2} \frac{1}{(2 |\nu|)^{(n+1)/2}} \sup_{x_n \in \R} \| S g (\centerdot, x_n) \|_{L^2(\R \times \R^{n-1})},\\
& \int_{\R} \| g (\centerdot, y_n) \|_{L^2(\R \times \R^{n-1})} \, \dd y_n = 2|\nu| (2 |\nu|)^{(n+1)/2} \int_\R \| f \|_{L^2(\R \times  H_{\hat{\nu}, s})} \, \dd s.
\end{align*}

Then, the \cref{th:gain_Snu} is an immediate consequence of the following inequality
for the multiplier $S$.

\begin{proposition}\label{prop:L2_S} \sl
Consider $n \in \N$.
There exists an absolute constant $C > 0$ such that
\[ \sup_{x_n \in \R} \| S f (\centerdot, x_n) \|_{L^2(\R \times \R^{n-1})} \leq C \int_{\R} \| f (\centerdot, y_n) \|_{L^2(\R \times \R^{n-1})} \, \dd y_n \]
for all $f \in \mathcal{S}(\R \times \R^n)$.
\end{proposition}

\begin{proof}
Start by noticing that for $(t, x_n) \in \R \times \R$ we have that the Fourier transform 
of $x^\prime \in \R^{n-1} \mapsto Sf(t, x^\prime , x_n)$ can be written as
\[ \mathcal{F}[Sf(t, \centerdot, x_n)] (\xi^\prime) = \frac{e^{it|\xi^\prime|^2}}{2\pi} \int_{\R \times \R} e^{i (t \sigma + x_n \xi_n)} \frac{1}{\sigma - \xi_n^2 + i \xi_n} \widehat{f}(\sigma+|\xi^\prime|^2, \xi^\prime, \xi_n) \, \dd (\sigma, \xi_n) \]
for all $\xi^\prime \in \R^{n - 1}$. For that, we have performed the translation 
$\sigma = \tau - |\xi^\prime|^2$.

Applying Plancharel's identity in $\R$,
we have for $(\xi^\prime, x_n) \in \R^{n - 1} \times \R$ that
\[ \int_{\R} | \mathcal{F}[Sf(t, \centerdot, x_n)] (\xi^\prime) |^2 \, \dd t = \int_{\R} | u (\sigma, \xi^\prime, x_n) |^2 \, \dd \sigma \]
where
\[ u (\sigma, \xi^\prime, x_n) = \frac{1}{(2\pi)^{1/2}} \int_{\R} e^{i x_n \xi_n} \frac{1}{\sigma - \xi_n^2 + i \xi_n} \widehat{f}(\sigma+|\xi^\prime|^2, \xi^\prime, \xi_n) \, \dd \xi_n. \]

Plancharel's identity now in $\R^{n - 1}$ implies that
\[ \int_{\R} \| Sf(t, \centerdot, x_n) \|_{L^2(\R^{n - 1})}^2 \, \dd t = \int_{\R} \| \mathcal{F}[Sf(t, \centerdot, x_n)] \|_{L^2(\R^{n - 1})}^2 \, \dd t = \| u (\centerdot, x_n) \|_{L^2(\R \times \R^{n - 1})}^2  \]
for every $x_n \in \R$.

Before going on with the argument, we introduce a family of convolution operators.
For $\sigma \in \R \setminus \{ 0 \}$, the function
\[ \eta \in \R \longmapsto \frac{1}{\sigma - \eta^2 + i \eta} \in \C \]
is integrable---since the polynomial function in the denominator does not vanish in $\R$,
so we can define
\begin{equation}\label{id:kernel_sigma}
K_\sigma (y) = \frac{1}{2 \pi} \int_{\R} e^{i y \eta} \frac{1}{\sigma - \eta^2 + i \eta} \, \dd \eta \quad \forall y \in \R.
\end{equation}
Then, we introduce the family
\[ T_\sigma g = K_\sigma \ast g \quad \forall g \in \mathcal{S} (\R). \]

One can check that
\[ u (\sigma, \xi^\prime, x_n) = T_\sigma \mathcal{F}^{-1} [ \widehat{f} (\sigma+|\xi^\prime|^2, \xi^\prime, \centerdot) ] (x_n), \]
and consequently, for $(\sigma, \xi^\prime, x_n) \in (\R \setminus \{ 0 \}) \times \R^{n-1} \times \R$, we obtain
\[ |u (\sigma, \xi^\prime, x_n)| \leq \| K_\sigma \|_{L^\infty (\R)} \| \mathcal{F}^{-1} [ \widehat{f} (\sigma+|\xi^\prime|^2, \xi^\prime, \centerdot) ] \|_{L^1(\R)}. \]

Using the \cref{lem:kernel_sigma} and performing the translation 
$\tau = \sigma + |\xi^\prime|^2$, we can conclude that
\[ \| u (\centerdot, x_n) \|_{L^2(\R \times \R^{n - 1})}^2 \lesssim \int_{\R \times \R^{n - 1}} \| \mathcal{F}^{-1} [ \widehat{f} (\tau, \xi^\prime, \centerdot) ] \|_{L^1(\R)}^2 \, \dd (\tau, \xi^\prime). \]
The implicit constant of this inequality is absolute.

Since
\[ \mathcal{F}^{-1} [ \widehat{f} (\tau, \xi^\prime, \centerdot) ] (y_n) = \mathcal{F} [ f (\centerdot, y_n) ](\tau, \xi^\prime)\quad \forall (\tau, \xi^\prime, y_n) \in \R \times \R^{n-1} \times \R, \]
Young's inequalities for integrals, and then Plancherel's identity in 
$\R \times \R^{n - 1}$ yield
\begin{gather*}
\Big( \int_{\R \times \R^{n - 1}} \| \mathcal{F}^{-1} [ \widehat{f} (\tau, \xi^\prime, \centerdot) ] \|_{L^1(\R)}^2 \, \dd (\tau, \xi^\prime) \Big)^{1/2} \leq \int_{\R} \| \mathcal{F} [ f (\centerdot, y_n) ] \|_{L^2(\R \times \R^{n - 1})} \, \dd y_n \\
= \int_{\R} \| f (\centerdot, y_n) \|_{L^2(\R \times \R^{n - 1})} \, \dd y_n.
\end{gather*}
Summing up, we have shown that
\[ \| Sf (\centerdot, x_n) \|_{L^2(\R \times \R^{n - 1})} =  \| u (\centerdot, x_n) \|_{L^2(\R \times \R^{n - 1})} \lesssim \int_{\R} \| f (\centerdot, y_n) \|_{L^2(\R \times \R^{n - 1})} \, \dd y_n. \]
Taking supremum in $x_n \in \R$, we obtain the inequality of the statement.
\end{proof}

We now state and prove the lemma we used in the proof of the \cref{prop:L2_S}. In the
statement, $\mathbf{1}_I$ denotes the characteristic function of the interval $I \subset \R$.

\begin{lemma}\label{lem:kernel_sigma} \sl
For almost every $\sigma \in \R \setminus \{ 0 \}$, the function $K_\sigma$ defined in 
\eqref{id:kernel_sigma} is explicitly given by
\begin{align*}
K_\sigma (x) = & - \mathbf{1}_{(1/4, \infty)} (\sigma) \mathbf{1}_{(-\infty, 0)} (x) \frac{2e^{x/2}}{|4\sigma - 1|^{1/2}} \sin (|4\sigma - 1|^{1/2} x /2) \\
& - \mathbf{1}_{(0, 1/4)} (\sigma) \mathbf{1}_{(-\infty, 0)} (x) \frac{2e^{x/2}}{|4\sigma - 1|^{1/2}} \sinh (|4\sigma - 1|^{1/2} x /2) \\
& - \mathbf{1}_{(-\infty, 0)} (\sigma) \mathbf{1}_{(-\infty, 0)} (x) \frac{1}{|4\sigma - 1|^{1/2}} e^{(1 + |4\sigma - 1|^{1/2}) x /2} \\
& + \mathbf{1}_{(-\infty, 0)} (\sigma) \mathbf{1}_{(0, \infty)} (x) \frac{1}{|4\sigma - 1|^{1/2}} e^{-(|4\sigma - 1|^{1/2} - 1) x /2}
\end{align*}
for almost every $x \in \R$.

Consequently,
\[ \esssup_{\sigma \in \R \setminus \{ 0 \}} \| K_\sigma \|_{L^\infty (\R)} < \infty. \]
\end{lemma}

\begin{proof}
Start by defining the complex polynomial function
\[p_\sigma (\zeta) = - \zeta^2 + i \zeta + \sigma \quad \forall \zeta \in \C, \]
and observe that
\begin{equation}
\label{id:integral}
K_\sigma (x) = \int_\R f_{(\sigma, x)} (\xi) \, \dd \xi,
\end{equation}
where
\[f_{(\sigma, x)} (\zeta) = \frac{1}{2\pi} \frac{e^{-ix\zeta}}{p_\sigma(\zeta)} \quad \forall \zeta \in \C \setminus \{ \eta_\sigma, \kappa_\sigma \} \]
with $\eta_\sigma$ and $\kappa_\sigma$ denoting the roots of $p_\sigma$.

One can check that, whenever $\sigma > 1/4$, $p_\sigma$ vanishes at
\[\eta_\sigma = \frac{|4\sigma - 1|^{1/2} + i}{2} \quad \textnormal{and} \quad \kappa_\sigma = \frac{- |4\sigma - 1|^{1/2} + i}{2}. \]
For $\sigma = 1/4$, the polynomial function vanishes to second order at $i/2$. However, 
whenever $\sigma < 1/4$, it vanishes at the imaginary points
\[\eta_\sigma = i \Big[ \frac{|4\sigma - 1|^{1/2} + 1}{2} \Big] \quad \textnormal{and} \quad \kappa_\sigma = i \Big[ \frac{1 - |4\sigma - 1|^{1/2}}{2} \Big]. \]

The idea to prove the first part of this lemma is to compute the integral on the right-hand
side of \eqref{id:integral} for $\sigma \in \R \setminus \{ 0, 1/4 \}$, avoiding in this way
the zero of order $2$ at $\sigma = 1/4$. To do so, we will use Cauchy's residue theorem.

The residues of $f_{(\sigma,x)}$ at the points $\eta_\sigma$ and $\kappa_\sigma$ are 
respectively
\[ \res_{\eta_\sigma} f_{(\sigma,x)} = \frac{1}{2\pi} \frac{e^{-ix\eta_\sigma}}{\kappa_\sigma - \eta_\sigma}, \qquad \res_{\kappa_\sigma} f_{(\sigma,x)} = \frac{1}{2\pi} \frac{e^{-ix\kappa_\sigma}}{\eta_\sigma - \kappa_\sigma}. \]

For $\sigma > 1/4$ and $x \leq 0$, we choose a contour $\gamma_R$ so that
\begin{equation}
\label{id:upper_contour}
\int_{\gamma_R} f_{(\sigma, x)} (\zeta) \, \dd \zeta = \int_{-R}^R f_{(\sigma, x)} (\xi) \, \dd \xi + \int_0^\pi f_{(\sigma, x)} (Re^{i\theta}) i R e^{i\theta} \dd \theta.
\end{equation}

On the one hand, we have that
\[ \Big| \int_0^\pi f_{(\sigma, x)} (Re^{i\theta}) i R e^{i\theta} \dd \theta \Big| \lesssim \frac{R e^{Rx}}{|R - |\eta_\sigma|| |R - |\kappa_\sigma||}, \]
so that
\begin{equation}
\label{lim:upper_contour}
\lim_{R \to \infty} \Big| \int_0^\pi f_{(\sigma, x)} (Re^{i\theta}) i R e^{i\theta} \dd \theta \Big| = 0
\end{equation}
since $x \leq 0$.

On the other hand, for $R$ sufficiently large, Cauchy's residue theorem 
ensures that
\[\int_{\gamma_R} f_{(\sigma, x)} (\zeta) \, \dd \zeta = 2\pi i [ \res_{\eta_\sigma} f_{(\sigma,x)} + \res_{\kappa_\sigma} f_{(\sigma,x)} ] = - \frac{2e^{x/2}}{|4\sigma - 1|^{1/2}} \sin (|4\sigma - 1|^{1/2} x /2).\]

Therefore, whenever $\sigma > 1/4$, we have that
\[ K_\sigma (x) = - \frac{2e^{x/2}}{|4\sigma - 1|^{1/2}} \sin (|4\sigma - 1|^{1/2} x /2) \quad \forall x \in (-\infty, 0]. \]

For $\sigma > 1/4$ and $x \geq 0$, we choose a contour $\gamma_R$ so that
\begin{equation}
\label{id:lower_contour}
\int_{\gamma_R} f_{(\sigma, x)} (\zeta) \, \dd \zeta = \int_{-R}^R f_{(\sigma, x)} (\xi) \, \dd \xi + \int_0^\pi f_{(\sigma, x)} (Re^{-i\theta}) (-i) R e^{-i\theta} \dd \theta.
\end{equation}

On the one hand, we have that
\[ \Big| \int_0^\pi f_{(\sigma, x)} (Re^{-i\theta}) (-i) R e^{-i\theta} \dd \theta \Big| \lesssim \frac{R e^{-Rx}}{|R - |\eta_\sigma|| |R - |\kappa_\sigma||}, \]
so that
\begin{equation}
\label{lim:lower_contour}
\lim_{R \to \infty} \Big| \int_0^\pi f_{(\sigma, x)} (Re^{-i\theta}) (-i) R e^{-i\theta} \dd \theta \Big| = 0
\end{equation}
since $x \geq 0$.

On the other hand, for every $R > 0$ the Cauchy--Goursat theorem 
implies that
\[\int_{\gamma_R} f_{(\sigma, x)} (\zeta) \, \dd \zeta = 0.\]

Therefore, whenever $\sigma > 1/4$, we have that
\[ K_\sigma (x) = 0 \quad \forall x \in [0, \infty). \]

For $0 < \sigma < 1/4$ and $x \leq 0$, we choose a contour $\gamma_R$ as in 
\eqref{id:upper_contour}.
Since \eqref{lim:upper_contour} holds, Cauchy's residue theorem ensures that, for $R$ large 
enough, we know that
\[\int_{\gamma_R} f_{(\sigma, x)} (\zeta) \, \dd \zeta = 2\pi i [ \res_{\eta_\sigma} f_{(\sigma,x)} + \res_{\kappa_\sigma} f_{(\sigma,x)} ] = - \frac{2e^{x/2}}{|4\sigma - 1|^{1/2}} \sinh (|4\sigma - 1|^{1/2} x /2).\]

Therefore, whenever $0 < \sigma < 1/4$, we have that
\[ K_\sigma (x) = - \frac{2e^{x/2}}{|4\sigma - 1|^{1/2}} \sinh (|4\sigma - 1|^{1/2} x /2) \quad \forall x \in (-\infty, 0]. \]

For $0 < \sigma < 1/4$ and $x \geq 0$, we choose a contour $\gamma_R$ as in 
\eqref{id:lower_contour}.
Since \eqref{lim:lower_contour} holds, the Cauchy--Goursat theorem implies that, for all 
$R > 0$, we have that
\[\int_{\gamma_R} f_{(\sigma, x)} (\zeta) \, \dd \zeta = 0.\]

Therefore, whenever $0 < \sigma < 1/4$, we have that
\[ K_\sigma (x) = 0 \quad \forall x \in [0, \infty). \]

For $\sigma < 0$ and $x \leq 0$, we choose a contour $\gamma_R$ as in 
\eqref{id:upper_contour}.
Since \eqref{lim:upper_contour} holds, Cauchy's residue theorem ensures that, for $R$ large 
enough, we know that
\[\int_{\gamma_R} f_{(\sigma, x)} (\zeta) \, \dd \zeta = 2\pi i \res_{\eta_\sigma} f_{(\sigma,x)} = - \frac{1}{|4\sigma - 1|^{1/2}} e^{(1 + |4\sigma - 1|^{1/2}) x /2}.\]

Therefore, whenever $\sigma < 0$, we have that
\[ K_\sigma (x) = - \frac{1}{|4\sigma - 1|^{1/2}} e^{(1 + |4\sigma - 1|^{1/2}) x /2} \quad \forall x \in (-\infty, 0]. \]

For $\sigma < 0$ and $x \geq 0$, we choose a contour $\gamma_R$ as in \eqref{id:lower_contour}.
Since \eqref{lim:lower_contour} holds, Cauchy's residue theorem ensures that, for $R$ large 
enough, we know that
\[\int_{\gamma_R} f_{(\sigma, x)} (\zeta) \, \dd \zeta = 2\pi i \res_{\kappa_\sigma} f_{(\sigma,x)} = \frac{1}{|4\sigma - 1|^{1/2}} e^{-(|4\sigma - 1|^{1/2}- 1) x /2}.\]

Therefore, whenever $\sigma < 0$, we have that
\[ K_\sigma (x) = \frac{1}{|4\sigma - 1|^{1/2}} e^{-(|4\sigma - 1|^{1/2}- 1) x /2} \quad \forall x \in [0, \infty). \]

This concludes the computation of $K_\sigma$, which corresponds to the first part of this 
lemma. The second part consists of a rudimentary verification.
\end{proof}

\subsection{Proof of the \cref{th:non-gain_Snu}} After the relations in 
\eqref{id:scaling_operator} and \eqref{id:scaling_function}, one can check that
\begin{align*}
& \| S_\nu f \|_{L^{q^\prime} (\R; L^{r\prime}(\R^n))} = \frac{1}{4|\nu|^2} \frac{1}{(2 |\nu|)^{2/q^\prime + n/r^\prime}} \| S g \|_{L^{q^\prime} (\R; L^{r\prime}(\R^n))},\\
& \| g \|_{L^q (\R; L^r(\R^n))} = (2 |\nu|)^{2/q + n/r} \| f \|_{L^q (\R; L^r(\R^n))}.
\end{align*}

Then, since
\[ \frac{2}{q} + \frac{n}{r} - 2 - \frac{2}{q^\prime} - \frac{n}{r^\prime} = 2 \Big[ \frac{2}{q} - 2 + \frac{n}{r} - \frac{n}{2} \Big] = 0,  \]
the \cref{th:non-gain_Snu} is a consequence of the following inequality for $S$.

\begin{proposition} \label{prop:scale_invariant} \sl
Consider $n \in \N$ and $(q, r) \in [1, 2] \times [1, 2]$ satisfying 
\eqref{id:Strichartz_indeces}.
There exists a constant $C > 0$ that only depends on $n$, $q$ and $r$ such that
\[ \| S f \|_{L^{q^\prime} (\R; L^{r\prime}(\R^n))} \leq C \| f \|_{L^q (\R; L^r(\R^n))} \]
for all $f \in \mathcal{S}(\R \times \R^n)$.
\end{proposition}

Before proceeding with the proof, it might be convenient for the reader to recall the notation
of the Schrödinger propagator $e^{-i\tm\Delta}$ in $\R^d$:
\[ e^{-it \Delta} f (x) = \frac{1}{(2\pi)^{d/2}} \int_{\R^d} e^{i\xi \cdot x} e^{it|\xi|^2} \widehat{f}(\xi) \, \dd \xi \quad \forall (t, x) \in \R \times \R^d, \]
as well as the inequality
\begin{equation}
\label{in:dispersive}
\| e^{-it \Delta} f \|_{L^\infty(\R^d)} \lesssim \frac{1}{|t|^{d/2}} \| f \|_{L^1(\R^d)},
\end{equation}
where the implicit constant only depends on $d$.

\begin{proof}[Proof of the \cref{prop:scale_invariant} when $(q, r) \neq (2, 2n/(n+2))$]
In order to prove the inequality, it will be useful to write $Sf$ in the following form: 
\begin{equation}
\label{id:Sf_u}
Sf(t, x) = \int_\R U_s [f(t-s, \centerdot )] (x) \, \dd s \quad
\forall (t, x) \in \R \times \R^n.
\end{equation}
Here $U_s$ denotes the Fourier multiplier given by
\[ U_s \phi (x) = i \sign(s) \frac{1}{(2\pi)^{n/2}} \int_{\R^n} e^{ix \cdot \xi} e^{is|\xi|^2} \mathbf{1}_{(-\infty, 0)} (s \xi_n) e^{s \xi_n} \widehat{\phi}(\xi) \, \dd \xi \quad \forall x \in \R^n, \]
for $s \in \R$ and $\phi \in \mathcal{S}(\R^n)$.
 
To check that \eqref{id:Sf_u} holds, start by performing the translation $\sigma = \tau - |\xi|^2$,
and writing
\[S f (t, x) = \frac{1}{(2\pi)^\frac{n+1}{2}} \int_{\R \times \R^n} e^{i (t \sigma + x \cdot \xi)} \frac{e^{i t |\xi|^2}}{\sigma + i \xi_n} \widehat{f}(\sigma + |\xi|^2, \xi) \, \dd (\sigma, \xi).\]

Observe that
\[ S f (t, x) = \frac{1}{(2\pi)^\frac{n+1}{2}} \int_{\R^n} e^{i x \cdot \xi} \Big( \int_\R \frac{1}{\sigma + i \xi_n} e^{i t \sigma} e^{i t |\xi|^2}  \widehat{f}(\sigma + |\xi|^2, \xi) \, \dd \sigma \Big) \, \dd \xi \]
with
\[e^{i t \sigma} e^{i t |\xi|^2} \widehat{f}(\sigma + |\xi|^2, \xi) = \frac{1}{(2\pi)^{1/2}} \int_\R e^{i \sigma s} e^{is|\xi|^2} \mathcal{F}[f(t-s, \centerdot)] (\xi) \, \dd s.\]

In the view of these two identities, it is natural to compute
\[\int_\R \frac{1}{\sigma + i \xi_n} \widecheck{\phi}(\sigma) \, \dd \sigma \]
for $\phi \in \mathcal{S}(\R)$.
We could proceed using Cauchy's residue theorem as in the \cref{lem:kernel_sigma}, but 
instead, we will use the well-known formula
\[ \mathcal{F}^{-1} \Big( \frac{1}{c + i\, \centerdot } \Big) (\xi) = (2\pi)^\frac{1}{2} e^{- c \xi} \mathbf{1}_{(0, \infty)} (\xi) \quad \forall \xi \in \R, \]
with $c \in (0, \infty)$.

For $\xi_n \in \R \setminus \{ 0 \}$, we have that
\begin{align*}
& \int_\R \frac{1}{\sigma + i \xi_n} \widecheck{\phi}(\sigma) \, \dd \sigma = \int_\R \frac{i}{- \xi_n + i \sigma} \widecheck{\phi}(\sigma) \, \dd \sigma \\
& = \mathbf{1}_{\R_-} (\xi_n)  \int_\R \mathcal{F}^{-1} \Big( \frac{i}{|\xi_n| + i\, \centerdot } \Big) (s) \phi(s) \, \dd s - \mathbf{1}_{\R_+} (\xi_n) \int_\R \mathcal{F}^{-1} \Big( \frac{i}{|\xi_n| + i\, \centerdot } \Big) (s) \phi(-s) \, \dd s \\
& = i (2\pi)^{1/2} \int_\R e^{s \xi_n}  \big[ \mathbf{1}_{\R_-} (\xi_n) \mathbf{1}_{\R_+} (s) - \mathbf{1}_{\R_+} (\xi_n) \mathbf{1}_{\R_-} (s) \big] \phi(s)  \, \dd s.
\end{align*}
Here $\R_-$ and  $\R_+$ denote the intervals $(-\infty, 0)$ and $(0, \infty)$ respectively.

Hence,
\[\int_\R \frac{1}{\sigma + i \xi_n} \widecheck{\phi}(\sigma) \, \dd \sigma = i (2\pi)^{1/2} \int_\R \sign(s) \mathbf{1}_{(-\infty, 0)} (s \xi_n) e^{s \xi_n} \phi(s)  \, \dd s \]
whenever $\xi_n \in \R \setminus \{ 0 \}$.

All these computations show that
\[ S f (t, x) = \frac{i}{(2\pi)^\frac{n}{2}} \int_{\R^n} e^{i x \cdot \xi} \Big( \int_\R \sign(s) \mathbf{1}_{(-\infty, 0)} (s \xi_n) e^{s \xi_n} e^{is|\xi|^2} \mathcal{F}[f(t-s, \centerdot)] (\xi) \, \dd s \Big) \, \dd \xi. \]
Then, changing the order of integration, we obtain the expression \eqref{id:Sf_u}.

Now, we start with the actual boundedness of $S$. For that, we will bound
the norm of the Fourier multiplier $U_s$ from $L^r(\R^n)$ to $L^{r^\prime} (\R^n)$ with
with $r \in [1, 2]$ and $r^\prime$ denoting its conjugate exponent. We proceed by 
interpolating the following two inequalities:
\begin{align}
\label{in:L2L2}
& \| U_s \phi \|_{L^2(\R^n)} \leq \| \phi \|_{L^2(\R^n)},\\
\label{in:LinftyL1}
& \| U_s \phi \|_{L^\infty(\R^n)} \lesssim \frac{1}{|s|^{n/2}} \|\phi \|_{L^1 (\R^n)},
\end{align}
for all $\phi \in \mathcal{S}(\R^n)$ and $s \in \R \setminus \{0\}$.
The implicit constant of the second inequality only depends on $n$.

The first inequality follows from Plancherel's identity, while to prove the second inequality, we observe 
that
\[ U_s \phi (x) = \frac{i \sign(s)}{2\pi} \int_\R K_s(x_n - y_n) e^{-is\Delta} [\phi(\centerdot, y_n)](x^\prime) \, \dd y_n \]
with $x^\prime = (x \cdot e_1, \dots , x \cdot e_{n-1})$ and
\[ K_s (y) = \int_{\R} e^{iy \eta} e^{is\eta^2} \mathbf{1}_{(-\infty, 0)} (s \eta) e^{s \eta}  \, \dd \eta \quad \forall y \in \R. \]

Then,
\[ \| U_s \phi \|_{L^\infty(\R^n)} \lesssim \| K_s \|_{L^\infty(\R)} \int_\R \| e^{-is\Delta} [\phi(\centerdot, y_n)] \|_{L^\infty(\R^{n-1})} \, \dd y_n. \]

From the estimate \eqref{in:dispersive} with $d=n-1$, we have that \eqref{in:LinftyL1} holds if
\begin{equation}
\label{in:s^1/2}
\| K_s \|_{L^\infty(\R)} \lesssim \frac{1}{|s|^{1/2}}.
\end{equation}

The inequality $\| K_s \|_{L^\infty(\R)} \leq 1/|s|$ follows directly from the definition of $K_s$
since
\[ \| K_s \|_{L^\infty(\R)} \leq \int_{\R} \mathbf{1}_{(-\infty, 0)} (s \eta) e^{s \eta}  \, \dd \eta. \]
However, this inequality is weaker than \eqref{in:s^1/2} whenever $|s|<1$. In order to obtain 
\eqref{in:s^1/2}, we need to consider the effect coming from the oscillations of 
$\eta \mapsto e^{is\eta^2}$.

Note that changing $\eta = \xi/s$ in the integral defining $K_s$, evaluated in $2y$, we have
\begin{align*}
K_s (2y) &= \frac{1}{|s|} \int_{\R} e^{i2y \xi/s} e^{i\xi^2/s} \mathbf{1}_{(-\infty, 0)} (\xi) e^{\xi}  \, \dd \xi = \frac{e^{-iy^2/s}}{|s|} \int_{\R} e^{i(\xi + y)^2/s} \mathbf{1}_{(-\infty, 0)} (\xi) e^{\xi}  \, \dd \xi\\
&= \frac{e^{-iy^2/s}}{|s|} e^{-y} \int_{-\infty}^y e^{i\eta^2/s} e^{\eta}  \, \dd \eta.
\end{align*}

Using the bound in the \cref{lem:oscillatory_integral}, we have that \eqref{in:s^1/2} holds, which confirms that \eqref{in:LinftyL1} also holds.

From the inequalities \eqref{in:L2L2} and \eqref{in:LinftyL1},
we deduce by the Riesz--Thorin interpolation that
\begin{equation}\label{in:Lr'Lr}
\| U_s \phi \|_{L^{r^\prime}(\R^n)} \lesssim \frac{1}{|s|^{n/r - n/2}}  \| \phi \|_{L^r(\R^n)} \quad \forall s \in \R \setminus \{ 0 \},
\end{equation}
where $r \in [1, 2]$ and $r^\prime$ is its conjugate 
exponent. The implicit constant of this inequality only depends on $n$ and $r$.

From the inequality \eqref{in:Lr'Lr} we derive the the wanted estimate as follows:
\[ \| S f (t, \centerdot) \|_{L^{r^\prime} (\R^n)} \leq \int_\R \| U_s [f(t-s, \centerdot )] \|_{L^{r^\prime}(\R^n)} \, \dd s \lesssim \int_\R \frac{\| f (t-s, \centerdot) \|_{L^r(\R^n)}}{|s|^{n/r - n/2}}   \, \dd s \quad \forall t \in \R. \]

Applying the Hardy--Littlewood--Sobolev inequality in $\R$ for
$n/r - n/2 = 1 - n(1/2 + 1/n - 1/r) $ with $1/q^\prime = 1/q - n(1/2 + 1/n - 1/r)$ we have
\[ \| S f \|_{L^{q^\prime} (\R; L^{r\prime}(\R^n))} \lesssim \| f \|_{L^q (\R; L^r(\R^n))} \]
whenever $0 < n(1/2 + 1/n - 1/r) < 1$. This yields the relation between $q$, $r$ and $n$.

Note that the condition $0 < n(1/2 + 1/n - 1/r)$ prevents us of reaching the case 
$(q, r) = \big( 2, 2n/(n+2) \big)$.
\end{proof}

\begin{lemma}\label{lem:oscillatory_integral} \sl For every $y \in \R$ and $s \in \R \setminus \{ 0 \}$, 
we have that
\[ \Big| \int_{-\infty}^y e^{i\frac{\xi^2}{s}} e^{\xi} \, \dd \xi \Big| \lesssim e^y |s|^{1/2}. \]
\end{lemma}

\begin{proof}
Take $\chi \in \mathcal{S}(\R)$ such that 
$\supp \chi \subset \{ \xi \in \R : |\xi| \leq 2 \}$ and
$\chi (\xi) = 1$ whenever $|\xi| \leq 1$. For $r \in (0, \infty)$, we define
\begin{align*}
& \chi_{\leq r} (\xi) = \chi (\xi/r) \quad \forall \xi \in \R,\\
& \chi_{>r}(\xi) = 1 - \chi_{\leq r} (\xi)\quad \forall \xi \in \R.
\end{align*}

On the one hand,
\begin{equation}
\label{in:leq_r}
\Big| \int_{-\infty}^y e^{i\frac{\xi^2}{s}} e^{\xi} \chi_{\leq r} (\xi) \, \dd \xi \Big| \leq r e^y \| \chi \|_{L^1(\R)}.
\end{equation}

On the other hand
\[\int_{-\infty}^y e^{i\frac{\xi^2}{s}} e^{\xi} \chi_{> r} (\xi) \, \dd \xi = -i \frac{s}{2} \int_{-\infty}^y \partial [ e^{i\frac{\, \centerdot^2}{s}} ] (\xi) \frac{e^{\xi}}{\xi} \chi_{> r} (\xi) \, \dd \xi \]
since
\[ \partial [ e^{i\frac{\, \centerdot^2}{s}} ] (\xi) = i \frac{2}{s} \xi e^{i\frac{\xi^2}{s}} \quad \forall \xi \in \R. \]

Integrating by parts, we obtain
\begin{align*}
\int_{-\infty}^y \partial [ e^{i\frac{\, \centerdot^2}{s}} ] (\xi) \frac{e^{\xi}}{\xi} \chi_{> r} (\xi) \, \dd \xi &= e^{i\frac{\, y^2}{s}} e^{y} \frac{\chi_{> r} (y)}{y} \\
& \quad -\int_{-\infty}^y e^{i\frac{\, \xi^2}{s}} e^{\xi} \Big[ \frac{\chi_{> r} (\xi)}{\xi} - \frac{1}{r} \frac{\chi^\prime_{\leq r} (\xi)}{\xi} - \frac{\chi_{> r} (\xi)}{\xi^2} \Big] \, \dd \xi.
\end{align*}
Here $\chi^\prime_{\leq r} (\xi) = \chi^\prime (\xi/r) $ for all $\xi \in \R$, with 
$\chi^\prime$ denoting the first-order derivative of $\chi$. 

The integral in the 
right-hand side of the previous identity can be bounded as follows:
\begin{align*}
& \Big| \int_{-\infty}^y e^{i\frac{\, \xi^2}{s}} e^{\xi} \Big[ \frac{\chi_{> r} (\xi)}{\xi} - \frac{1}{r} \frac{\chi^\prime_{\leq r} (\xi)}{\xi} - \frac{\chi_{> r} (\xi)}{\xi^2} \Big] \, \dd \xi \Big| \\
& \qquad \leq \sup_{\xi \in \R} \frac{|\chi_{> r} (\xi)|}{|\xi|} \int_{-\infty}^y e^{\xi} \, \dd \xi + e^y \Big[ \frac{1}{r} \int_\R \frac{|\chi^\prime_{\leq r} (\xi)|}{|\xi|} \, \dd \xi + \int_\R \frac{|\chi_{> r} (\xi)|}{\xi^2} \, \dd \xi \Big] \\
& \qquad \leq \frac{e^y}{r} \Big[ \sup_{\eta \in \R} \frac{|1 - \chi (\eta)|}{|\eta|} + \int_\R \frac{|\chi^\prime (\eta)|}{|\eta|} \, \dd \eta + \int_\R \frac{|1 - \chi (\eta)|}{\eta^2} \, \dd \eta \Big].
\end{align*}
Therefore, we have that
\[\Big| \int_{-\infty}^y \partial [ e^{i\frac{\, \centerdot^2}{s}} ] (\xi) \frac{e^{\xi}}{\xi} \chi_{> r} (\xi) \, \dd \xi \Big| \leq \frac{e^y}{r} \Big[ 2 \sup_{\eta \in \R} \frac{|1 - \chi (\eta)|}{|\eta|} + \int_\R \Big(\frac{|\chi^\prime (\eta)|}{|\eta|} + \frac{|1 - \chi (\eta)|}{\eta^2} \Big) \, \dd \eta \Big].  \]

Consequently,
\begin{equation}
\label{in:grt_r}
\Big| \int_{-\infty}^y e^{i\frac{\xi^2}{s}} e^{\xi} \chi_{> r} (\xi) \, \dd \xi \Big| \lesssim \frac{|s|}{r} e^y. 
\end{equation}

From the inequalities \eqref{in:leq_r} and \eqref{in:grt_r}, we have that
\[ \Big| \int_{-\infty}^y e^{i\frac{\xi^2}{s}} e^{\xi} \, \dd \xi \Big| \lesssim e^y \Big( r + \frac{|s|}{r} \Big). \]
Choosing $r\in (0, \infty)$ so that $r = |s|/r$, we obtain the bound stated.
\end{proof}

In order to prove the \cref{prop:scale_invariant} in the endpoint $(q, r) = (2, 2n/(n+2))$
with $n \geq 3$, we need some preliminary results.
Recall the definition of the Fourier multiplier $U_s$ defined in the proof of the 
\cref{prop:scale_invariant}
\[ U_s \phi (x) = i \sign(s) \frac{1}{(2\pi)^{n/2}} \int_{\R^n} e^{ix \cdot \xi} e^{is|\xi|^2} \mathbf{1}_{(-\infty, 0)} (s \xi_n) e^{s \xi_n} \widehat{\phi}(\xi) \, \dd \xi \quad \forall x \in \R^n, \]
for $s \in \R$ and $\phi \in \mathcal{S}(\R^n)$.
Define the map
\[ Tf(x) = \int_\R U_s [f(s, \centerdot )] (x) \, \dd s \quad
\forall x \in \R^n, \]
for all $f \in \mathcal{S}(\R \times \R^n)$.

\begin{lemma} \label{lem:T_boundedness} \sl
Consider $n \in \N$ and $(q, r) \in [1, 2] \times [1, 2]$ satisfying 
\eqref{id:Strichartz_indeces} with $(q, r) \neq (2, 2n/(n+2))$.
There exists a constant $C > 0$ that only depends on $n$, $q$ and $r$ such that
\[ \| T f \|_{L^2 (\R^n)} \leq C \| f \|_{L^q (\R; L^r(\R^n))} \]
for all $f \in \mathcal{S}(\R \times \R^n)$.
\end{lemma}

\begin{proof}
Start by noticing that
\begin{align*}
\| T f \|_{L^2 (\R^n)}^2 
& = \int_{\R \times \R} \Big[ \int_{\R^n} U_s [f(s, \centerdot )] (x) \overline{U_t [f(t, \centerdot )] (x)} \, \dd x \Big] \, \dd (s,t) \\
& = \int_{\R \times \R^n} \Big[\int_{\R} U_t^\ast  U_s [f(s, \centerdot )] (x) \, \dd s \Big] \overline{f(t, x)} \, \dd (t, x),
\end{align*}
where
\[ U_t^\ast \phi (x) = -i \sign(t) \frac{1}{(2\pi)^{n/2}} \int_{\R^n} e^{ix \cdot \xi} e^{-it|\xi|^2} \mathbf{1}_{(-\infty, 0)} (t \xi_n) e^{t \xi_n} \widehat{\phi}(\xi) \, \dd \xi \quad \forall x \in \R^n, \]
and
\[ U_t^\ast U_s \phi (x) = \frac{\mathbf{1}_{\R_+}(st)}{(2\pi)^{n/2}} \int_{\R^n} e^{ix \cdot \xi} e^{i(s-t)|\xi|^2} \mathbf{1}_{\R_-} ((s+t) \xi_n) e^{(s+t) \xi_n} \widehat{\phi}(\xi) \, \dd \xi \quad \forall x \in \R^n, \]
with $(s, t) \in \R \times \R $ and $\phi \in \mathcal{S}(\R^n)$.
Here $\R_-$ and  $\R_+$ denote the intervals $(-\infty, 0)$ and $(0, \infty)$ respectively,
and we have used that
\[ \sign(s) \sign(t) \mathbf{1}_{(-\infty, 0)} (s \xi_n) \mathbf{1}_{(-\infty, 0)} (t \xi_n) = \mathbf{1}_{(0, \infty)}(st) \mathbf{1}_{(-\infty, 0)} ((s+t) \xi_n). \]

Let $L$ denote the linear operator given by
\[ L f(t, x) = \int_{\R} U_t^\ast  U_s [f(s, \centerdot )] (x) \, \dd s  \quad \forall (t, x) \in  \R \times \R^n. \]

The inequality announced in the statement will follow from the next one
\begin{equation}
\label{in:TastT=L}
\| L f \|_{L^{q^\prime} (\R; L^{r\prime}(\R^n))} \leq C \| f \|_{L^q (\R; L^r(\R^n))}
\end{equation}
for all $f \in \mathcal{S}(\R \times \R^n)$, with $q^\prime$ and $r^\prime$ the conjugate
exponents of $q$ and $r$. Indeed,
\[ \| T f \|_{L^2 (\R^n)}^2 \leq \| L f \|_{L^{q^\prime} (\R; L^{r\prime}(\R^n))} \| f \|_{L^q (\R; L^r(\R^n))} \lesssim \| f \|_{L^q (\R; L^r(\R^n))}^2.  \]

It can be convenient for the reader to mention here that one can check that $L = T^\ast T$
with $T^\ast$ denoting the adjoint of $T$.

We now focus on showing that \eqref{in:TastT=L} holds. This is pretty similar to the proof
of the \cref{prop:scale_invariant}. Start by bounding
the norm of $U_t^\ast U_s$ from $L^r(\R^n)$ to $L^{r^\prime} (\R^n)$.

We proceed again by 
interpolating the following two inequalities:
\begin{align}
\label{in:L2L2_TastT}
& \| U_t^\ast U_s \phi \|_{L^2(\R^n)} \leq \| \phi \|_{L^2(\R^n)},\\
\label{in:LinftyL1_TastT}
& \| U_t^\ast U_s \phi \|_{L^\infty(\R^n)} \lesssim \frac{1}{|s-t|^{n/2}} \|\phi \|_{L^1 (\R^n)},
\end{align}
for all $\phi \in \mathcal{S}(\R^n)$ and $(s,t) \in \R \times \R $ such that $s \neq t$.
The implicit constant of the second inequality only depends on $n$.

The first inequality follows from Plancherel's identity, while to prove the second inequality, we observe 
that
\[ U_t^\ast U_s \phi (x) = \frac{\mathbf{1}_{\R_+}(st)}{2\pi} \int_\R K_{(s,t)}(x_n - y_n) e^{-i(s-t)\Delta} [\phi(\centerdot, y_n)](x^\prime) \, \dd y_n \]
with $x^\prime = (x \cdot e_1, \dots , x \cdot e_{n-1})$ and
\[ K_{(s,t)} (y) = \int_{\R} e^{iy \eta} e^{i(s-t)\eta^2} \mathbf{1}_{(-\infty, 0)} ((s+t) \eta) e^{(s+t)\eta}  \, \dd \eta \quad \forall y \in \R. \]

Then,
\[ \| U_t^\ast U_s \phi \|_{L^\infty(\R^n)} \lesssim \| K_{(s,t)} \|_{L^\infty(\R)} \int_\R \| e^{-i(s-t)\Delta} [\phi(\centerdot, y_n)] \|_{L^\infty(\R^{n-1})} \, \dd y_n. \]
From the estimate \eqref{in:dispersive} with $d=n-1$,
we have that \eqref{in:LinftyL1_TastT} holds if
\begin{equation}
\label{in:s^1/2_TastT}
\| K_{(s,t)} \|_{L^\infty(\R)} \lesssim \frac{1}{|s-t|^{1/2}}
\end{equation}
for $(s,t) \in \R \times \R $ such that $s\neq t$.

Since $st > 0$, we have that $s \neq - t$ and, consequently, we can perform the change variables  
$\eta = \xi/(s+t)$ in the integral defining $K_{(s,t)}$. Evaluating this in $2y$, we have
\begin{align*}
K_{(s,t)} (2y) &= \frac{1}{|s+t|} \int_{\R} e^{i\frac{2y}{s+t} \xi} e^{i\frac{s-t}{(s+t)^2} \xi^2} \mathbf{1}_{(-\infty, 0)} (\xi) e^{\xi}  \, \dd \xi \\
& = \frac{e^{-i\frac{y^2}{s-t}}}{|s+t|} \int_{\R} e^{i\frac{s-t}{(s+t)^2}\big(\xi + \frac{s+t}{s-t}y \big)^2} \mathbf{1}_{(-\infty, 0)} (\xi) e^{\xi}  \, \dd \xi \\
&= \frac{e^{-i\frac{y^2}{s-t}}}{|s+t|} e^{-\frac{s+t}{s-t}y} \int_{-\infty}^{\frac{s+t}{s-t}y} e^{i\frac{s-t}{(s+t)^2} \eta^2} e^{\eta}  \, \dd \eta.
\end{align*}

Using the bound in the \cref{lem:oscillatory_integral}, we have that \eqref{in:s^1/2_TastT} holds, which confirms that \eqref{in:LinftyL1_TastT} also holds.

From the inequalities \eqref{in:L2L2_TastT} and \eqref{in:LinftyL1_TastT},
we deduce by the Riesz--Thorin interpolation that
\begin{equation}\label{in:Lr'Lr_TastT}
\| U_t^\ast U_s \phi \|_{L^{r^\prime}(\R^n)} \lesssim \frac{1}{|s - t|^{n/r - n/2}}  \| \phi \|_{L^r(\R^n)} \quad \forall (s, t) \in \R \times \R : s\neq t,
\end{equation}
where $r \in [1, 2]$ and $r^\prime$ is its conjugate 
exponent. The implicit constant of this inequality only depends on $n$ and $r$.

From the inequality \eqref{in:Lr'Lr_TastT} we derive the the wanted estimate as follows:
\[ \| L f (t, \centerdot) \|_{L^{r^\prime} (\R^n)} \leq \int_\R \| U_t^\ast U_s [f(s, \centerdot )] \|_{L^{r^\prime}(\R^n)} \, \dd s \lesssim \int_\R \frac{\| f (s, \centerdot) \|_{L^r(\R^n)}}{|s-t|^{n/r - n/2}}   \, \dd s \quad \forall t \in \R. \]

Applying the Hardy--Littlewood--Sobolev inequality in $\R$ for
$n/r - n/2 = 1 - n(1/2 + 1/n - 1/r) $ with $1/q^\prime = 1/q - n(1/2 + 1/n - 1/r)$ we have
\[ \| L f \|_{L^{q^\prime} (\R; L^{r\prime}(\R^n))} \lesssim \| f \|_{L^q (\R; L^r(\R^n))} \]
whenever $0 < n(1/2 + 1/n - 1/r) < 1$. This implies the relation between $q$, $r$ and 
$n$.

This concludes the proof of the inequality \eqref{in:TastT=L}, which was the key 
inequality to ensure the lemma.
\end{proof}


For convenience, let $r_n \in [1, 2]$ denote $r_n =  2n/(n+2)$.

\begin{proof}[Proof of the \cref{prop:scale_invariant} when $(q, r) = (2, r_n)$ with $n\geq 3$]
For $j \in \Z$ consider
\[S_jf(t, x) = \int_\R \mathbf{1}_{2^j} (s) U_s [f(t-s, \centerdot )] (x) \, \dd s \quad \forall (t, x) \in \R \times \R^n, \]
where $\mathbf{1}_{2^j}$ denotes the characteristic function of 
$\{ s \in \R : 2^{j-1} < |s| \leq 2^j \}$.

Following the approach adopted by Keel and Tao in \cite{zbMATH01215570}, we will prove
that
\begin{equation}
\label{in:summation_j}
\sum_{j \in \Z} \Big| \int_{\R \times \R^n} S_j f(t, x) \overline{g(t, x)} \, \dd (t,x) \Big| \lesssim \| f \|_{L^2 (\R; L^{r_n}(\R^n))} \| g \|_{L^2 (\R; L^{r_n}(\R^n))}.
\end{equation}

Observe that \eqref{in:summation_j} implies by the triangle inequality that
\[ \Big| \int_{\R \times \R^n} S f(t, x) \overline{g(t, x)} \, \dd (t,x) \Big| \lesssim \| f \|_{L^2 (\R; L^{r_n}(\R^n))} \| g \|_{L^2 (\R; L^{r_n}(\R^n))}, \]
and by duality we obtain the inequality announced in the proposition \cref{prop:scale_invariant}
when $(q, r) = (2, 2n/(n+2))$ with $n\geq 3$.

As in \cite{zbMATH01215570}, our first goal will be to prove 
the two-parameter family of inequalities:
\begin{equation}
\label{in:2-parameter_ineq}
\| S_j f \|_{L^2 (\R; L^{b^\prime}(\R^n))} \lesssim 2^{-j\frac{n}{2}\big( \frac{1}{a} - \frac{1}{r_n} + \frac{1}{b} - \frac{1}{r_n} \big)} \| f \|_{L^2 (\R; L^a(\R^n))}
\end{equation}
for all $j \in \Z$ and $(1/a, 1/b)$ in a neighbourhood of $(1/r_n, 1/r_n)$. The implicit 
constant only depends on $n$ and the neighbourhood around $(1/r_n, 1/r_n)$.

Observe that still here we cannot sum for $j \in \Z$ when $(a, b) = (r_n, r_n)$,
since in this case the exponent is $0$.

The inequality \eqref{in:2-parameter_ineq} follows by interpolating the next three inequalities:
\begin{align}
\label{in:dispersive_L2t}
& \| S_j f \|_{L^2 (\R; L^\infty (\R^n))} \lesssim 2^{-j\frac{n}{2}\big( 1 - \frac{1}{r_n} + 1 - \frac{1}{r_n} \big)} \| f \|_{L^2 (\R; L^1(\R^n))},\\
\label{in:non-enpoint_T_L2t}
& \| S_j f \|_{L^2 (\R \times \R^n)} \lesssim 2^{-j\frac{n}{2}\big( \frac{1}{r} - \frac{1}{r_n} + \frac{1}{2} - \frac{1}{r_n} \big)} \| f \|_{L^2 (\R; L^r(\R^n))}, \\
\label{in:non-enpoint_T_L2t_dual}
& \| S_j f \|_{L^2 (\R; L^{r^\prime}(\R^n))} \lesssim 2^{-j\frac{n}{2}\big( \frac{1}{2} - \frac{1}{r_n} + \frac{1}{r} - \frac{1}{r_n} \big)} \| f \|_{L^2 (\R \times \R^n)},
\end{align}
for all $f \in \mathcal{S}(\R \times \R^n)$ and $1 \leq r < r_n $. The implicit constants 
here only depends on $n$, as well as on $r$ for those in \eqref{in:non-enpoint_T_L2t} and
\eqref{in:non-enpoint_T_L2t_dual}.

Let us show how to derive these three inequalities.

Using \eqref{in:LinftyL1}, we have that
\begin{align*}
\| S_j f (t, \centerdot) \|_{L^\infty (\R^n)} & \leq \int_\R \mathbf{1}_{2^j} (s) \| U_s [f(t-s, \centerdot )] \|_{L^\infty (\R^n)} \, \dd s \\
& \lesssim \int_\R \frac{\mathbf{1}_{2^j} (s)}{|s|^{n/2}} \| f(t-s, \centerdot ) \|_{L^1 (\R^n)} \, \dd s.
\end{align*}
By Young's inequality we have that
\[ \| S_j f \|_{L^2 (\R; L^\infty (\R^n))} \lesssim 2^{-j(n/2 - 1)} \| f \|_{L^2 (\R; L^1(\R^n))}, \]
which proves \eqref{in:dispersive_L2t}.

In order to prove \eqref{in:non-enpoint_T_L2t},
it is convenient to realize that
\[ S_j f(t, x) = T g_t (x) \quad \forall (t, x) \in \R \times \R^n \]
with $T$ as in the \cref{lem:T_boundedness} and $g_t^j(s,y) = \mathbf{1}_{2^j} (s) f(t-s, y)$.
Hence,
\begin{align*}
\| S_j f (t, \centerdot) \|_{L^2 (\R^n)} &= \| T g_t^j \|_{L^2 (\R^n)} \\
&\lesssim \| g_t^j \|_{L^q (\R; L^r(\R^n))} = \Big( \int_\R \mathbf{1}_{2^j}(s) \| f(t-s, \centerdot ) \|_{L^r (\R^n)}^q \, \dd s \Big)^{1/q}.
\end{align*}

Using Young's inequality for $2/q>1$ we have that
\[ \| S_j f \|_{L^2 (\R \times \R^n)}^q \lesssim \| \mathbf{1}_{2^j} \|_{L^1(\R)} \| f \|_{L^2 (\R; L^r(\R^n))}^q, \]
which can be re-written as
\[ \| S_j f \|_{L^2 (\R \times \R^n)} \lesssim 2^{j/q} \| f \|_{L^2 (\R; L^r(\R^n))} = 2^{-j\big[\frac{1}{2} \big(\frac{n}{r} - \frac{n}{2} \big) - 1\big]} \| f \|_{L^2 (\R; L^r(\R^n))}, \]
by using the relation \eqref{id:Strichartz_indeces}. This proves the inequality 
\eqref{in:non-enpoint_T_L2t} since
\[ \frac{1}{2} \Big(\frac{n}{r} - \frac{n}{2} \Big) - 1 = \frac{n}{2}\Big( \frac{1}{r} - \frac{1}{2} - \frac{2}{n} \Big) = \frac{n}{2}\Big( \frac{1}{r} + \frac{1}{2} - \frac{2}{r_n} \Big). \]

Finally, the inequality \eqref{in:non-enpoint_T_L2t_dual} is 
straightforward after \eqref{in:non-enpoint_T_L2t} and the identity
\begin{equation}
\label{id:quasi-adjoint}
\int_{\R \times \R^n} S_j f(t, x) \overline{g(t, x)} \, \dd (t,x) = \int_{\R \times \R^n}  \tilde{f}(t, x) \overline{S_j \tilde{g}(t, x)} \, \dd (t,x),
\end{equation}
where $\tilde{f}(t, x) = f(t, x^\prime, -x_n)$ and 
$\tilde{g}(t, x) = g(t, x^\prime, -x_n)$.

Before checking that the previous identity holds, let us use it, together with the inequality \eqref{in:non-enpoint_T_L2t}, to derive \eqref{in:non-enpoint_T_L2t_dual}:
%
%
\begin{align*}
 \Big| \int_{\R \times \R^n} & S_j f(t, x) \overline{g(t, x)} \, \dd (t,x) \Big| \\
& \leq \| \tilde f \|_{L^2 (\R \times \R^n)} \| S_j \tilde{g} \|_{L^2 (\R \times \R^n)}  \lesssim 2^{j/q} \| \tilde{f} \|_{L^2 (\R \times \R^n)} \| \tilde{g} \|_{L^2 (\R; L^r(\R^n))}\\
& = 2^{-j\big[\frac{1}{2} \big(\frac{n}{r} - \frac{n}{2} \big) - 1\big]} \| f \|_{L^2 (\R \times \R^n)} \| g \|_{L^2 (\R; L^r(\R^n))}.
\end{align*}
Then, by duality we conclude that \eqref{in:non-enpoint_T_L2t_dual} holds.

Let us check that \eqref{id:quasi-adjoint} holds. Start by noticing that
\begin{align*}
\int_{\R \times \R^n} S_j f(t, x) \overline{g(t, x)} & \, \dd (t,x) \\
& = \int_{\R \times \R^n} f(s, x) \Big( \overline{\int_\R \mathbf{1}_{2^j} (t-s) U_{t-s}^\ast [g(t, \centerdot)] (x) \, \dd t}  \Big) \, \dd (s,x)\\
& = \int_{\R \times \R^n} f(s, x) \Big( \overline{\int_\R \mathbf{1}_{2^j} (t) U_{-t}^\ast [g(s-t, \centerdot)] (x) \, \dd t}  \Big) \, \dd (s,x).
\end{align*}
Furthermore, the identity
$U_{-t}^\ast [g(s-t, \centerdot)] (x) = U_t [\tilde{g}(s-t, \centerdot)] (x^\prime, -x_n)$ 
implies
\begin{align*}
\int_{\R \times \R^n} f(s, x) \Big( \overline{\int_\R \mathbf{1}_{2^j} (t) U_t [\tilde{g}(s-t, \centerdot)] (x^\prime, -x_n) \, \dd t} &  \Big) \, \dd (s,x)\\
& = \int_{\R \times \R^n}  \tilde{f}(s, x) \overline{S_j \tilde{g}(s, x)} \, \dd (s,x).
\end{align*}
From the previous two identities we can conclude that \eqref{id:quasi-adjoint} holds.

Now that we know that \eqref{in:dispersive_L2t}, \eqref{in:non-enpoint_T_L2t} and 
\eqref{in:non-enpoint_T_L2t_dual} hold, by interpolation we obtain 
\eqref{in:2-parameter_ineq}. 
At this point, we could just ask the reader to follow the interpolation 
argument in \cite[\S 5]{zbMATH01215570} to derive \eqref{in:summation_j} from
\eqref{in:2-parameter_ineq}. However, for the sake of completeness, we reproduce the argument
here.

In \cite[Lemma 5.1]{zbMATH01215570}, Keel and Tao proved an atomic decomposition 
functions in $L^p(\R^d)$, which allows us to write
\[ f(t, x) = \sum_{k \in \Z} \alpha_k(t) f_k(t, x) \quad g(t, x) = \sum_{k \in \Z} \beta_k(t) g_k(t, x) \qquad \forall (t, x) \in \R \times \R^n,\]
where $ \| f_k \|_{L^\infty(\R \times \R^n)} + \| g_k \|_{L^\infty(\R \times \R^n)} \lesssim 2^{-k/r_n} $,
$|\supp f_k \cap \supp f_l| = |\supp g_k \cap \supp g_l| = 0$,
$|\supp f_k| + |\supp g_k| \lesssim 2^k$, and
\[ \Big( \sum_{k \in \Z} |\alpha_k(t)|^{r_n} \Big)^{1/r_n} \lesssim \| f(t, \centerdot) \|_{L^{r_n}(\R^n)}, \quad \Big( \sum_{k \in \Z} |\beta_k(t)|^{r_n} \Big)^{1/r_n} \lesssim \| g(t, \centerdot) \|_{L^{r_n}(\R^n)}.\]

Note that, using \eqref{in:2-parameter_ineq} for the atoms $\alpha_k f_k$ and $\beta_l g_l$, we have
\begin{align*}
\Big| \int_{\R \times \R^n} & S_j [\alpha_k f_k](t, x) \overline{\beta_l (t) g_l(t, x)} \, \dd (t,x) \Big| \\
& \lesssim 2^{-j\frac{n}{2}\big( \frac{1}{a} - \frac{1}{r_n} + \frac{1}{b} - \frac{1}{r_n} \big)} \| \alpha_k \|_{L^2 (\R)} 2^{k\big(\frac{1}{a} - \frac{1}{r_n}\big)} \| \beta_l \|_{L^2 (\R)} 2^{l\big(\frac{1}{b} - \frac{1}{r_n}\big)}  \\
& = 2^{(k-j\frac{n}{2}) \big(\frac{1}{a} - \frac{1}{r_n}\big)} 2^{(l-j\frac{n}{2})\big(\frac{1}{b} - \frac{1}{r_n}\big)} \| \alpha_k \|_{L^2 (\R)} \| \beta_l \|_{L^2 (\R)}.
\end{align*}

Since the previous inequality holds for every $(1/a, 1/b)$ 
in a neighbourhood of $(1/r_n, 1/r_n)$,
one can choose $a$ and $b$ so that $1/a - 1/r_n = -\varepsilon \sign(k - jn/2)$ and 
$1/b - 1/r_n = -\varepsilon \sign(l - jn/2)$ to obtain
\[ \Big| \int_{\R \times \R^n}  S_j [\alpha_k f_k](t, x) \overline{\beta_l (t) g_l(t, x)} \, \dd (t,x) \Big| \lesssim 2^{-\varepsilon (|k-j\frac{n}{2}| + |l-j\frac{n}{2}|)} \| \alpha_k \|_{L^2 (\R)} \| \beta_l \|_{L^2 (\R)}. \]

Then,
\[\sum_{j \in \Z} \Big| \int_{\R \times \R^n} S_j f(t, x) \overline{g(t, x)} \, \dd (t,x) \Big| \lesssim \sum_{j, k, l} 2^{-\varepsilon (|k-j\frac{n}{2}| + |l-j\frac{n}{2}|)} \| \alpha_k \|_{L^2 (\R)} \| \beta_l \|_{L^2 (\R)}. \]

Since
\[ \Big| k-j\frac{n}{2} \Big| + \Big| l-j\frac{n}{2} \Big| \geq \frac{1}{2} (|k - l| + |nj-(k+l)|), \]
and
\[\sum_j 2^{-\frac{\varepsilon}{2} |nj-(k+l)|} \lesssim \sum_j 2^{-\frac{\varepsilon}{2} |j|} \lesssim 1,\]
we have that
\[\sum_{j \in \Z} \Big| \int_{\R \times \R^n} S_j f(t, x) \overline{g(t, x)} \, \dd (t,x) \Big| \lesssim \sum_{k, l} 2^{-\frac{\varepsilon}{2} |k-l|}  \| \alpha_k \|_{L^2 (\R)} \| \beta_l \|_{L^2 (\R)}. \]

Using Cauchy--Schwarz in the sum in $k$, and then Young's inequality to deal with the 
convolution, we obtain that
\[\sum_{j \in \Z} \Big| \int_{\R \times \R^n} S_j f(t, x) \overline{g(t, x)} \, \dd (t,x) \Big| \lesssim \sum_j 2^{-\frac{\varepsilon}{2} |j|} \Big( \sum_k \| \alpha_k \|_{L^2 (\R)}^2 \Big)^\frac{1}{2} \Big( \sum_l \| \beta_l \|_{L^2 (\R)}^2 \Big)^\frac{1}{2}.\]

Finally,
\begin{align*}
\sum_{j \in \Z} \Big| \int_{\R \times \R^n} & S_j f(t, x) \overline{g(t, x)} \, \dd (t,x) \Big| \lesssim \big\| \big( \sum_k |\alpha_k|^2  \big)^{1/2} \big\|_{L^2 (\R)} \big\| \big( \sum_l |\beta_j|^2  \big)^{1/2} \big\|_{L^2 (\R)} \\
& \lesssim \big\| \big( \sum_k |\alpha_k|^{r_n}  \big)^{1/r_n} \big\|_{L^2 (\R)} \big\| \big( \sum_l |\beta_j|^{r_n}  \big)^{1/r_n} \big\|_{L^2 (\R)} \\
& \lesssim \| f \|_{L^2 (\R; L^{r_n}(\R^n))} \| g \|_{L^2 (\R; L^{r_n}(\R^n))},
\end{align*}
which prove the inequality \eqref{in:summation_j}.
\end{proof}

\section{Bounds for the Birman--Schwinger operator}\label{sec:BSch_op}
The goal of this section is to prove that the norm of the Birman--Schwinger operator
$M_W \circ S_\nu \circ M_{|W|}$, with $W$ given by \eqref{def:W}, tends to vanish as
$|\nu|$ increase. Actually, we will
prove a more general result for an operator of the same form 
$M_{W_1} \circ S_\nu \circ M_{W_2}$---with $W_1$ and $W_2$ unrelated.

\begin{proposition}\label{prop:MSM}\sl Consider $S, T \in \R$ such that $S<T$.
Let $W_1$ and $W_2$ be two measurable functions in 
$\R\times\R^n$ with support in $[S,T] \times \R^n$.

For $j \in \{1, 2\}$, assume that
\begin{enumerate}[label=\textnormal{(\alph*)}, ref=\alph*]
\item $|W_j|^2 \in L^a(\R; L^b (\R^n))$ with $(a, b) \in [1, \infty] \times [1, \infty]$ satisfying \eqref{cond:V_ab},

\item there is $R_j > 0$ so that
\begin{equation}
\label{cond:decay_lines}
\sup_{\omega \in \Sph^{n-1}} \int_\R \| \mathbf{1}_{>R_j} |W_j|^2 \|_{L^\infty (\R \times H_{\omega, s})} \, \dd s < \infty,
\end{equation}
and $\mathbf{1}_{\leq R_j} W_j \in C(\R; L^n (\R^n))$ if $(a, b) = (\infty, n/2)$ for
$n\geq 3$.
\end{enumerate}

Then,
\[ \lim_{|\nu| \to \infty} \| M_{W_1} \circ S_\nu \circ M_{W_2} \|_{\mathcal{L}[L^2(\R \times \R^n)]} = 0. \]
\end{proposition}

\begin{remark} If we fix a direction $\hat{\nu} \in \Sph^{n-1}$, and we let $\nu$ denote any 
vector in $\R^n \setminus \{ 0 \}$ parallel to $\hat{\nu}$, then condition 
\eqref{cond:decay_lines} is only needed for $\omega = \hat{\nu}$.
\end{remark}

Before proving the \cref{prop:MSM}, it is convenient to make some observations.

As a consequence of the Cauchy--Schwarz inequality we have that
\begin{equation}
\label{in:L2_to_rays}
\int_\R \| W u \|_{L^2 (\R \times H_{\omega, s})} \, \dd s \leq \Big( \int_\R \| |W|^2 \|_{L^\infty (\R \times H_{\omega, s})} \, \dd s \Big)^{1/2} \| u \|_{L^2(\R \times \R^n)}.
\end{equation}
Then, by duality
\begin{equation}
\label{in:dual_of_L2_to_rays}
\| W u \|_{L^2(\R \times \R^n)} \leq \Big( \int_\R \| |W|^2 \|_{L^\infty (\R \times H_{\omega, s})} \, \dd s \Big)^{1/2} \esssup_{s \in \R} \| u \|_{L^2 (\R \times H_{\omega, s})}.
\end{equation}

It follows from Hölder's inequality that
\begin{equation}
\label{in:L2_to_LqLr}
\| W u \|_{L^q(\R; L^r (\R^n))} \leq \| |W|^2 \|_{L^a(\R; L^b (\R^n))}^{1/2} \| u \|_{L^2(\R \times \R^n)},
\end{equation}
with $(q, r) \in [1, 2] \times [1, 2]$ and $(a, b) \in [1, \infty] \times [1, \infty]$ 
satisfying the relation \eqref{id:indices4Vqr}, 
or equivalently \eqref{id:indices4Vab}. 
Once again, by duality we obtain 
\begin{equation}
\label{in:dual_of_L2_to_LqLr}
\| W u \|_{L^2(\R \times \R^n)} \leq \| |W|^2 \|_{L^a(\R; L^b (\R^n))}^{1/2} \| u \|_{L^{q^\prime}(\R; L^{r^\prime} (\R^n))}
\end{equation}
with $q^\prime$ and $r^\prime$ denoting 
the conjugate exponents of $q$ and $r$.

\begin{proof}[Proof of the \cref{prop:MSM}]
We will decompose $W_j$ as $W_j = W_j^\sharp + W_j^\flat$, where the choices of $W_j^\sharp$ and 
$W_j^\flat$ depend on whether $(a, b) \neq (\infty, n/2)$ for $n \geq 2$ or 
$(a, b) = (\infty, n/2)$ for $n \geq 3$. Then,
\begin{equation}\label{in:sharp_flat}
\begin{aligned}
\| W_1 S_\nu (W_2 u) \|_{L^2(\R \times \R^n)} \leq & \| W_1^\sharp S_\nu (W_2^\sharp u) \|_{L^2(\R \times \R^n)} + \| W_1^\sharp S_\nu (W_2^\flat u) \|_{L^2(\R \times \R^n)} \\
& + \| W_1^\flat S_\nu (W_2 u) \|_{L^2(\R \times \R^n)}.
\end{aligned}
\end{equation}

On the one hand, the inequalities \eqref{in:L2_to_rays} and \eqref{in:dual_of_L2_to_rays}, 
together with the 
\cref{th:gain_Snu}, imply that
\begin{align*}
& \| W_1^\sharp S_\nu (W_2^\sharp u) \|_{L^2(\R \times \R^n)} \\ & \lesssim \frac{1}{|\nu|} \Big( \int_\R \| |W_1^\sharp|^2 \|_{L^\infty (\R \times H_{\hat{\nu}, s})} \, \dd s \Big)^{1/2} \Big( \int_\R \| |W_2^\sharp|^2 \|_{L^\infty (\R \times H_{\hat{\nu}, s})} \, \dd s \Big)^{1/2} \| u \|_{L^2(\R \times \R^n)}.
\end{align*}

On the other hand, the inequalities \eqref{in:L2_to_LqLr} and \eqref{in:dual_of_L2_to_LqLr}, 
together with the \cref{th:non-gain_Snu}, imply that
\begin{align*}
 \| W_1^\sharp S_\nu (W_2^\flat u) \|_{L^2(\R \times \R^n)} & + \| W_1^\flat S_\nu (W_2 u) \|_{L^2(\R \times \R^n)} \\ 
& \lesssim \| |W_1^\sharp|^2 \|_{L^a(\R; L^b (\R^n))}^{1/2}  \| |W_2^\flat|^2 \|_{L^a(\R; L^b (\R^n))}^{1/2} \| u \|_{L^2(\R \times \R^n)} \\
& \quad + \| |W_1^\flat|^2 \|_{L^a(\R; L^b (\R^n))}^{1/2}  \| |W_2|^2 \|_{L^a(\R; L^b (\R^n))}^{1/2} \| u \|_{L^2(\R \times \R^n)}.
\end{align*}

Whenever $(a, b) \neq (\infty, n/2)$ with $n \geq 2$, we choose,
for $\lambda > 0$ and $j \in \{1, 2\}$, the functions
\[ W_j^\sharp = \mathbf{1}_{\leq R_j} \mathbf{1}_{|W_j|\leq \lambda} W_j + \mathbf{1}_{> R_j} W_j, \qquad W_j^\flat = W_j - W_j^\sharp, \]
where $\mathbf{1}_{|W_j|\leq \lambda}$ denote the characteristic function of
$\{ (t, x) \in \R \times \R^n : |W_j(t,x)| \leq \lambda \}$ respectively.

With this choice, one can check that
\[ \| |W_j^\sharp|^2 \|_{L^\infty (\R \times H_{\hat{\nu}, s})} \leq \lambda^2 \mathbf{1}_{[-R_j, R_j]}(s) + \| \mathbf{1}_{>R_j} |W_j|^2 \|_{L^\infty (\R \times H_{\hat{\nu}, s})} \]
which in turns implies that
\[ \Big( \int_\R \| |W_j^\sharp|^2 \|_{L^\infty (\R \times H_{\hat{\nu}, s})} \, \dd s \Big)^{1/2} \leq \lambda (2R_j)^{1/2} + \Big(  \int_\R \| \mathbf{1}_{>R_j} |W_j|^2 \|_{L^\infty (\R \times H_{\hat{\nu}, s})} \, \dd s \Big)^{1/2}. \]

Furthermore, $|W_1^\sharp (t,x)| \leq |W_1 (t,x)|$ for almost every 
$(t,x) \in \R \times \R^n$, and
\[\| |W_j^\flat|^2 \|_{L^a(\R; L^b (\R^n))}^{1/2} = \| \mathbf{1}_{\leq R_j} \mathbf{1}_{|W_j|> \lambda} |W_j|^2 \|_{L^a(\R; L^b (\R^n))}^{1/2}. \]
As suggested by the notation, $\mathbf{1}_{|W_j|> \lambda}$ denotes the
characteristic function of the set 
$\{ (t, x) \in \R \times \R^n : |W_j(t,x)| > \lambda \}$.


Gathering these upper bounds and plugging into the \eqref{in:sharp_flat}, we obtain that 
there exists a constant $C > 0$ that only depends on $n$, $(a, b)$ and $R_j$,
as well as on $\| |W_j|^2 \|_{L^a(\R; L^b (\R^n))}$ and the quantity 
in \eqref{cond:decay_lines} for $j \in \{1, 2\}$, such that
\[ \| M_{W_1} \circ S_\nu \circ M_{W_2} \|_{\mathcal{L}[L^2(\R \times \R^n)]} \leq C \Big( \frac{\lambda^2 + 1 }{|\nu|} + \sum_{j \in \{ 1,2 \}} \| \mathbf{1}_{|W_j|> \lambda} |W_j|^2 \|_{L^a(\R; L^b (\R^n))}^{1/2} \Big) \]
for all $\nu \in \R^n \setminus \{ 0 \}$ and $\lambda > 0$.

Finally, to see that
\[ \lim_{|\nu| \to \infty} \| M_{W_1} \circ S_\nu \circ M_{W_2} \|_{\mathcal{L}[L^2(\R \times \R^n)]} = 0 \]
it is enough to choose $\lambda^2 = |\nu|^{1/2}$ and make $|\nu|$ tend to $\infty$.

We now turn our attention to the endpoint $(a, b) = (\infty, n/2)$ for $n \geq 3$.
For any $\varepsilon > 0$, one can find a partition 
$S = t_0 < t_1 < \dots < t_m < t_{m+1} = T$ for which 
the piecewise constant-in-time approximation satisfies
\[
\sup_{t \in \R} \Big\| \mathbf{1}_{\leq R_j} \Big( W_j(t, \centerdot) - \sum_{k=1}^m \mathbf{1}_{(t_{k-1}, t_k]}(t) W_j(t_k, \centerdot) \Big) \Big\|_{L^n(\R^n)} \leq \varepsilon.
\]

For $\lambda > 0$ and $j \in \{1, 2\}$, we introduce the functions
\[ W_j^\sharp =  \mathbf{1}_{\leq R_j} \sum_{k = 1}^m \mathbf{1}_{(t_{k-1}, t_k]} \mathbf{1}_{|W_j(t_k, \centerdot)| \leq \lambda} W_j(t_k, \centerdot) + \mathbf{1}_{> R_j} W_j, \qquad W_j^\flat = W_j - W_j^\sharp, \]
where $\mathbf{1}_{|W_j(t_k, \centerdot)| \leq \lambda}$ denotes the
indicator function of $\{ x \in \R^n : |W_j(t_k,x)| \leq \lambda \}$.

With this choice, we have again that
\[ \| |W_j^\sharp|^2 \|_{L^\infty (\R \times H_{\hat{\nu}, s})} \leq \lambda^2 \mathbf{1}_{[-R_j, R_j]}(s) + \| \mathbf{1}_{>R_j} |W_j|^2 \|_{L^\infty (\R \times H_{\hat{\nu}, s})} \]
which in turns implies that
\[ \Big( \int_\R \| |W_j^\sharp|^2 \|_{L^\infty (\R \times H_{\hat{\nu}, s})} \, \dd s \Big)^{1/2} \leq \lambda (2R_j)^{1/2} + \Big(  \int_\R \| \mathbf{1}_{>R_j} |W_j|^2 \|_{L^\infty (\R \times H_{\hat{\nu}, s})} \, \dd s \Big)^{1/2}. \]

Furthermore $\| |W_1^\sharp|^2 \|_{L^\infty(\R; L^{n/2} (\R^n))} \leq \| |W_1|^2 \|_{L^\infty(\R; L^{n/2} (\R^n))}$, and
\[ W_j^\flat = \mathbf{1}_{\leq R_j} \Big( W_j - \sum_{k = 1}^m \mathbf{1}_{(t_{k-1}, t_k]} W_j(t_k, \centerdot)\Big) + \mathbf{1}_{\leq R_j} \sum_{k = 1}^m \mathbf{1}_{(t_{k-1}, t_k]} \mathbf{1}_{|W_j(t_k, \centerdot)| > \lambda} W_j(t_k, \centerdot) \]
where $\mathbf{1}_{|W_j(t_k, \centerdot)| > \lambda}$ denotes the
indicator function of $\{ x \in \R^n : |W_j(t_k,x)| > \lambda \}$.

Hence,
\begin{align*}
\| |W_j^\flat|^2 \|_{L^\infty(\R; L^{n/2} (\R^n))}^{1/2} & \leq \Big\| \mathbf{1}_{\leq R_j} \Big( W_j - \sum_{k=1}^m \mathbf{1}_{(t_{k-1}, t_k]} W_j(t_k, \cdot) \Big) \Big\|_{L^\infty(\R; L^n (\R^n))} \\
& + \Big\| \mathbf{1}_{\leq R_j} \Big| \sum_{k = 1}^m \mathbf{1}_{(t_{k-1}, t_k]} \mathbf{1}_{|W_j(t_k, \centerdot)| > \lambda} W_j(t_k, \centerdot) \Big|^2 \Big\|_{L^\infty(\R; L^{n/2} (\R^n))}^{1/2}.
\end{align*}

Observe that
\begin{align*}
\Big\| \mathbf{1}_{\leq R_j} \Big| \sum_{k = 1}^m \mathbf{1}_{(t_{k-1}, t_k]} & \mathbf{1}_{|W_j(t_k, \centerdot)| > \lambda} W_j(t_k, \centerdot) \Big|^2 \Big\|_{L^\infty(\R; L^{n/2} (\R^n))}^{1/2} \\
& = \max_{k \in\{ 1, \dots, m \}} \| \mathbf{1}_{\leq R_j} \mathbf{1}_{|W_j(t_k, \centerdot)| > \lambda} | W_j(t_k, \centerdot) |^2 \|_{L^{n/2}(\R^n)}^{1/2},
\end{align*}
and $\lambda$ can be chosen so that the right-hand side of the previous identity becomes
\[\max_{k \in\{ 1, \dots, m \}} \| \mathbf{1}_{\leq R_j} \mathbf{1}_{|W_j(t_k, \centerdot)| > \lambda} | W_j(t_k, \centerdot) |^2 \|_{L^{n/2}(\R^n)}^{1/2} < \varepsilon.\]
Therefore,
\[ \| |W_j^\flat|^2 \|_{L^\infty(\R; L^{n/2} (\R^n))}^{1/2} \leq 2 \varepsilon.\]


Gathering these upper bounds and plugging into the \eqref{in:sharp_flat}, we obtain that 
there exists a constant $C > 0$ that only depends on $n$ and $R_j$,
as well as on the norm of $|W_j|^2$ in $L^\infty(\R; L^{n/2} (\R^n))$ and the quantity 
in \eqref{cond:decay_lines} for $j \in \{1, 2\}$, such that
\[ \| M_{W_1} \circ S_\nu \circ M_{W_2} \|_{\mathcal{L}[L^2(\R \times \R^n)]} \leq C \Big( \frac{\lambda^2 + 1 }{|\nu|} + \varepsilon \Big) \]
for all $\nu \in \R^n \setminus \{ 0 \}$.

Choosing $\nu$ so that $|\nu|$ is sufficiently large,
we obtain that
\[ \| M_{W_1} \circ S_\nu \circ M_{W_2} \|_{\mathcal{L}[L^2(\R \times \R^n)]} \lesssim \varepsilon. \]

This already ensures that the statement also holds for the endpoint with the continuity as an 
extra assumption.
\end{proof}

\section{Hamiltonians with identical initial-to-final-state maps}
\label{sec:identical_hamiltonians}

We begin this section by recalling key facts about the well-posedness of the initial-value 
problem \eqref{pb:IVP}. Following this, we prove an integral identity that establishes the 
orthogonality relation \eqref{eq:ortho_intro}. Finally, we derive a dual description of 
\eqref{eq:ortho_intro}, a formulation essential for later extending the orthogonality relation to 
CGO solutions.


Consider
\[V = V^\sharp + V^\flat\] 
with $V^\sharp \in L^1 ((0, T); L^\infty(\R^n)) $ and 
$V^\flat \in L^a ((0, T); L^b(\R^n))$ with $n \in \N$ and
$(a, b) \in [1, \infty] \times [1, \infty]$ satisfying
\eqref{cond:V_ab}. At the endpoint $(a, b) = (\infty, n/2)$ with $n \geq 3$, we further assume
that $V^\flat \in C ([0, T]; L^{n/2}(\R^n))$.

Then, for every $f \in L^2(\R^n)$, there exists a unique
$u \in C([0,T]; L^2(\R^n)) \cap L^{q^\prime} ((0,T); L^{r^\prime} (\R^n))$ 
solving the problem \eqref{pb:IVP}---see for example 
\cite{zbMATH02222227,zbMATH02204588,zbMATH05013664}\footnote{This solution
is obtained by solving an integral equation that is a consequence of Duhamel's principle.
One can check that this solution $u$ satisfies the equation 
$(i\partial_\tm  + \Delta - V) u = 0$ in $(0, T) \times \R^n$ in the sense of 
distributions.}---with $(q^\prime, r^\prime) \in [2, \infty] \times [2, \infty]$
satisfying \eqref{id:indicesVprimes}.

Moreover, the linear map 
\[ f \in L^2(\R^n) \longmapsto u \in C([0, T]; L^2(\R^n)) \cap L^{q^\prime} ((0,T); L^{r^\prime} (\R^n)) \]
is bounded. Consequently, for every $t \in [0,T]$, the linear maps
\[ \mathcal{U}_t : f \in L^2(\R^n) \mapsto u(t, \centerdot) \in L^2(\R^n) \]
is also bounded.

One can check that the pair 
$(q^\prime, r^\prime) \in [2, \infty] \times [2, \infty]$
is an admissible Strichartz pair that
satisfies the relation \eqref{id:Strichartz_dual-indeces}.

The assumption $ V^\flat \in C ([0, T]; L^{n/2}(\R^n))$ can be replaced
by a smallness condition the quantity
$\| V^\flat \|_{L^\infty ((s, s+\delta); L^{n/2}(\R^n))}$
for some $\delta > 0$ and all $s \in [0, T-\delta]$---see \cite{zbMATH02204588}.

In the view of the relation between the functional spaces involved in the initial-value 
problem \eqref{pb:IVP}, it seems convenient to introduce some notation:

Let $X$ denote the set
\[ X = \{ F^\sharp + F^\flat : F^\sharp \in L^1((0,T); L^2(\R^n)) \enspace \textnormal{and} \enspace F^\flat \in L^q((0,T); L^r(\R^n)) \},\]
where $(q, r) \in [1, 2] \times [1, 2]$ satisfies \eqref{id:Strichartz_indeces}. Endowed
with the norm
\[ \| F \|_X = \inf \{ \| F^\sharp \|_{L^1((0,T); L^2(\R^n))} + \| F^\flat \|_{L^q((0,T); L^r(\R^n))} : F = F^\sharp + F^\flat \},\]
the normed space $(X, \| \centerdot \|_X)$ becomes Banach.

Let $X^\star$ denote the set
\[X^\star = C([0,T]; L^2(\R^n)) \cap L^{q^\prime} ((0,T); L^{r^\prime} (\R^n)) \]
where $q^\prime$ and $r^\prime$ denote the conjugate exponents of $q$ and $r$. Endowed with
the norm
\[ \| u \|_{X^\star} = \max \Big( \sup_{t \in [0, T]} \| u(t, \centerdot) \|_{L^2(\R^n)}, \| u \|_{L^{q^\prime} ((0,T); L^{r^\prime} (\R^n))} \Big), \]
the normed space $(X^\star, \| \centerdot \|_{X^\star})$ becomes Banach and can be identified
with a subspace of the dual of $(X, \| \centerdot \|_X)$.

Consider $(a, b) \in [1, \infty] \times [1, \infty]$ satisfies \eqref{cond:V_ab}.
Let $Y$ denote the set
\[ Y = \{ V^\sharp + V^\flat : V^\sharp \in L^1((0,T); L^\infty(\R^n)) \enspace \textnormal{and} \enspace V^\flat \in L^a((0,T); L^b(\R^n))  \}\]
if $(a, b) \neq (\infty, n/2)$, or
\[ Y = \{ V^\sharp + V^\flat : V^\sharp \in L^1((0,T); L^\infty(\R^n)) \enspace \textnormal{and} \enspace V^\flat \in C([0,T]; L^{n/2}(\R^n))  \}\]
if $(a, b) = (\infty, n/2)$ and $n \geq 3$. Endowed
with the norm
\[ \| V \|_Y = \inf \{ \| V^\sharp \|_{L^1((0,T); L^\infty(\R^n))} + \| V^\flat \|_{L^a((0,T); L^b(\R^n))} : V = V^\sharp + V^\flat \},\]
the normed space $(Y, \| \centerdot \|_Y)$ becomes Banach. 

Throughout this section,
the pairs $(a, b)$, $(q, r)$ and $(q^\prime, r^\prime)$ corresponding to $Y$, $X$ and $X^\star$
respectively, are always related by the identity 
\eqref{id:indices4Vab}.

\subsection{An integral identity for physical solutions}

The main goal here is to state an integral formula that relates the 
initial-to-final-state maps with their corresponding potentials. The
identity consists of a straightforward generalization of the one proved
in \cite{10.1063/5.0152372} for potentials $V_1$ and $V_2$ in $L^1((0,T); L^\infty(\R^n))$.

\begin{proposition}\label{prop:orthogonality} \sl Let $V_1$ and $V_2$ belong to $Y$.
For every $f, g \in L^2(\R^n)$ we have that
\begin{equation}
i \int_{\R^n} (\mathcal{U}_T^1 - \mathcal{U}_T^2) f \, \overline{g} \; = \int_\Sigma (V_1 - V_2)u_1 \overline{v_2}\,,
\label{id:integral_identity}
\end{equation}
where $u_1 \in X^\star$ is the physical solution of \eqref{pb:IVP} with potential $V_1$ and initial 
data $f$, and $v_2 \in X^\star$ is the physical solution of the following final-value problem 
\begin{equation}
\label{pb:final_overV2}
	\left\{
		\begin{aligned}
		& i\partial_\tm v_2 = - \Delta v_2 + \overline{V_2} v_2 & & \textnormal{in} \, \Sigma, \\
		& v_2(T, \centerdot) = g &  & \textnormal{in} \, \R^n.
		\end{aligned}
	\right.
\end{equation}
\end{proposition}

In \cite{10.1063/5.0152372}, we showed that the previous identity for potentials 
in $L^1((0,T); L^\infty(\R^n))$ 
follows from a suitable integration-by-parts formula. 
We now state the corresponding integration by parts in our context:

\begin{proposition}\label{lem:integration_by_parts} \sl Consider 
$u , v \in X^\star$ such that $(i\partial_\tm + \Delta) u $ and $(i\partial_\tm + \Delta) v $
belong to $X$.
Then, we have
\[\int_\Sigma \big[ (i\partial_\tm + \Delta) u \overline{v} - u \overline{(i\partial_\tm + \Delta)v} \big] \, = i \int_{\R^n} \big[ 
u(T, \centerdot) \overline{v (T,\centerdot)} - u(0, \centerdot) \overline{v (0,\centerdot)} \big] \,. \]
\end{proposition}

The proof of the \cref{prop:orthogonality} is straightforward after one has the right 
integra-tion-by-parts formula. In fact, the proof follows exactly the same lines as the
\cite[Proposition 4.1]{10.1063/5.0152372}. Since it is rather short, we include it here for completeness.

\begin{proof}[Proof of the \cref{prop:orthogonality}]
Beside the solutions $u_1$ and $v_2$, consider a third solution $u_2 \in X^\star$ of the initial-value problem \eqref{pb:IVP} with potential $V_2$ and initial state $f$.

Note that for $j \in \{ 1, 2 \}$ we have that
\[ \int_\Sigma (V_j - V_2)u_j \overline{v_2} \, = \int_\Sigma \big[ (i\partial_\tm + \Delta) u_j \overline{v_2} - u_j  \overline{(i\partial_\tm + \Delta) v_2} \big] \,. \]

Since $u_j, v_2 \in X^\star $ and $V_1, \overline{V_2} \in Y$, we have that
$(i\partial_\tm + \Delta) u_j $ and $(i\partial_\tm + \Delta) v_2$ belong to $X$. Then,
the \cref{lem:integration_by_parts} yields
\[ \int_\Sigma (V_j - V_2)u_j \overline{v_2} \, = i \int_{\R^n} \big[ 
u_j(T, \centerdot) \overline{v_2 (T,\centerdot)} - u_j(0, \centerdot) \overline{v_2 (0,\centerdot)} \big] \,. \]

From the previous identity we obtain, for $j = 1$, that
\begin{equation}
\int_\Sigma (V_1 - V_2)u_1 \overline{v_2} \, = i \int_{\R^n} \big[ 
u_1(T, \centerdot) \overline{v_2 (T,\centerdot)} - u_1(0, \centerdot) \overline{v_2 (0,\centerdot)} \big] \, ,
\label{id:integration_by_parts}
\end{equation}
while for $j = 2$ we have that
\begin{equation}
i \int_{\R^n} \big[ 
u_2(T, \centerdot) \overline{v_2 (T,\centerdot)} - u_2(0, \centerdot) \overline{v_2 (0,\centerdot)} \big] = 0.
\label{id:trivial_orthogonality}
\end{equation}

Subtracting the left-hand side of \eqref{id:trivial_orthogonality} to the right-hand side of \eqref{id:integration_by_parts}, we get
\[\int_\Sigma (V_1 - V_2)u_1 \overline{v_2} \, = i \int_{\R^n} \big[ 
u_1(T, \centerdot) - u_2 (T, \centerdot)\big] \overline{v_2 (T,\centerdot)} \]
since $u_1(0, \centerdot) = u_2(0, \centerdot) = f$.

Additionally, since 
$ u_j (T, \centerdot) = \mathcal{U}_T^j f$ and $v_2(0, \centerdot) = g$, we obtain the 
identity \eqref{id:integral_identity}.
\end{proof}

The \cref{lem:integration_by_parts} is proved using the same ideas we used in the proof
of \cite[Proposition 4.2]{10.1063/5.0152372}. For convenience for the reader, we have include
\cref{app:integration_by_parts} a sketch of the proof of \cref{lem:integration_by_parts}.
%

\subsection{Dual form of the orthogonality relation}
When $\mathcal{U}_T^1 = \mathcal{U}_T^2$ the \cref{prop:orthogonality} yields an orthogonality relation
for physical solutions. Here we derive a pair of results that provide a dual description of this 
orthogonality relation. 

It is convenient to recall the
well-posedness of the following problem.
\begin{equation}
\label{pb:non-homogenous}
	\left\{
		\begin{aligned}
		& (i\partial_\tm + \Delta - V) u = F & & \textnormal{in} \, \Sigma, \\
		& u(0, \centerdot) = 0 &  & \textnormal{in} \, \R^n.
		\end{aligned}
	\right.
\end{equation}

Consider $V \in Y$. Then, for every $ F \in X $, there exists a unique $ u \in X^\star $
solving the problem \eqref{pb:non-homogenous}.
Moreover, the linear map 
$ F \in X \mapsto u \in X^\star $ is bounded.
This fact has been proved for example in 
\cite{zbMATH02222227,zbMATH05013664}\footnote{This solution
is again obtained by solving an integral equation that is a consequence of Duhamel's 
principle. One can check that this solution $u$ satisfies the equation 
$(i\partial_\tm  + \Delta - V) u = F$ in $\Sigma$ in the sense of 
distributions.}.

\begin{lemma} \label{lem:orthogonal_rewardPB} \sl Given $V \in Y$, let $u \in X^\star$
be the solution of the problem \eqref{pb:non-homogenous} with $F \in X$.
If $F$ satisfies that
\[ \int_\Sigma F \overline{v} \, = 0 \]
for every $v \in X^\star $ solving the final-value problem
\[
\left\{
		\begin{aligned}
		& i\partial_\tm v = - \Delta v + \overline{V} v & & \textnormal{in} \, \Sigma, \\
		& v(T, \centerdot) = g &  & \textnormal{in} \, \R^n,
		\end{aligned}
	\right.
\]
with final value $g \in L^2 (\R^n)$, then
\[ u(T, \centerdot) = 0. \]
\end{lemma}

The proof of the \cref{lem:orthogonal_rewardPB} follows exactly the same lines as the
\cite[Lemma 4.4]{10.1063/5.0152372}, but we include it here for completeness, and
since it is rather short.

\begin{proof}[Proof of the \cref{lem:orthogonal_rewardPB}]
Consider $g \in L^2(\R^n)$, and let $v \in X^\star$ denote the solution
of the corresponding final-value problem. Then, by the 
\cref{lem:integration_by_parts}, and the fact that $u(0, \centerdot) = 0$, we have
\[ i \int_{\R^n} u(T, \centerdot) \overline{g} \, = \int_\Sigma \big[ (i\partial_\tm + \Delta) u \overline{v} - u \overline{(i\partial_\tm + \Delta)v} \big] \, .  \]

Note that the \cref{lem:integration_by_parts} can be applied because 
$ F \in X $, $V \in Y$, and $u$ and $v$ belong to $X^\star$.

Adding and subtracting $V u \overline{v}$ we have, by 
\eqref{pb:non-homogenous}, that
\[ i \int_{\R^n} u(T, \centerdot) \overline{g} \, = \int_\Sigma F \overline{v} \, - \int_\Sigma u \overline{(i\partial_\tm + \Delta - \overline{V})v} \, .  \]

The first term on the right-hand side is assumed to vanish, 
while the second term vanishes because $v$ is
the solution of the final-value problem.

Hence,
\[\int_{\R^n} u(T, \centerdot) \overline{g} \, = 0\]
Since $g$ is arbitrary, we can conclude that $u(T, \centerdot) = 0$.
\end{proof}

We will need a symmetric version of this lemma that we state without proof.
\begin{lemma}\label{lem:orthogonal_forwardPB} \sl Given $V \in Y$, let $ v\in X^\star $ be the
solution of the problem 
\begin{equation}
\label{pb:zero-final_non-homo}
\left\{
		\begin{aligned}
		& (i\partial_\tm + \Delta - \overline{V}) v = G & & \textnormal{in} \, \Sigma, \\
		& v(T, \centerdot) = 0 &  & \textnormal{in} \, \R^n,
		\end{aligned}
	\right.
\end{equation}
with $G \in X$. If $G$ satisfies that
\[ \int_\Sigma \overline{G} u \, = 0 \]
for every $u \in X^\star $ solving the initial-value problem 
\eqref{pb:IVP} with $f \in L^2 (\R^n)$, then
\[ v(0, \centerdot) = 0. \]
\end{lemma}

\subsection{An integration by parts for exponential growing functions}

The dual counterparts of the orthogonality relation for physical solutions can be used to extend
the integration-by-parts formula of the \cref{lem:integration_by_parts} for
a suitable $u$ with vanishing initial and
final states and $v$ having an exponential growth.
To do so, we need to first show that such $u$ has an 
exponential decay whenever $(i\partial_\tm + \Delta) u$ has it too.

\begin{lemma}\label{lem:exp_decay} \sl Consider $F \in L^q((0,T); L^r(\R^n))$,
with $(q, r) \in [1, 2] \times [1, 2]$ satisfying
\eqref{id:Strichartz_indeces}.
Let $u \in X^\star$ satisfy the conditions
\[
\left\{
		\begin{aligned}
		& (i\partial_\tm + \Delta) u = F & & \textnormal{in} \enspace \Sigma, \\
		& u(0, \centerdot) = u(T, \centerdot) = 0 &  & \textnormal{in} \enspace \R^n.
		\end{aligned}
	\right.
\]

If we additionally assume that $e^{c|\x|} F \in L^q((0,T); L^r(\R^n))$ for some $c>0$,
then, there exists a constant $C > 0$ that only depends on $n$, $q$ and $r$ such that
\[ \| e^{\nu \cdot \x} u \|_{L^{q^\prime} ((0, T); L^{r\prime}(\R^n))} \leq C \| e^{\nu \cdot \x} F \|_{L^q ((0, T); L^r(\R^n))} \]
for all $\nu \in \R^n \setminus \{ 0 \}$ with $|\nu| < c$.
\end{lemma}

This lemma is pretty similar to \cite[Lemma 4.6]{10.1063/5.0152372}. Actually, it is an adaptation
of this to our case in the light of the \cref{th:non-gain_Snu}. For this reason, we omit the proof at
this point and include it in the \cref{app:transference}.

\begin{remark}\label{rem:exp_decay}
For $F$ and $u$ in the conditions of the \cref{lem:exp_decay} we have that
$e^{c^\prime |\x|} u \in L^{q^\prime} ((0, T); L^{r\prime}(\R^n))$ for all $c^\prime < c$.
\end{remark}

A useful consequence of the \cref{lem:exp_decay,rem:exp_decay} is the following 
integration-by-parts formula that admits exponentially-growing functions.

\begin{proposition}\label{prop:CGO-PhS_int-by-parts}\sl 
Consider $F \in L^q((0,T); L^r(\R^n))$,
with $(q, r) \in [1, 2] \times [1, 2]$ satisfying
\eqref{id:Strichartz_indeces}.
Let $u \in X^\star$ satisfy the conditions
\[
\left\{
		\begin{aligned}
		& (i\partial_\tm + \Delta) u = F & & \textnormal{in} \enspace \Sigma, \\
		& u(0, \centerdot) = u(T, \centerdot) = 0 &  & \textnormal{in} \enspace \R^n.
		\end{aligned}
	\right.
\]

For every $\nu \in \R^n \setminus \{ 0 \}$,
consider a locally integrable measurable function 
$v^\nu : \R \times \R^n \rightarrow \C$ such that 
$e^{-\nu \cdot \x} (i\partial_\tm + \Delta) v^\nu \in L^q(\R; L^r(\R^n))$,
and for which there is $w^\nu \in L^{q^\prime}(\R; L^{r^\prime}(\R^n))$ so that
$\nu \cdot \nabla (e^{-\nu \cdot \x} v^\nu - w^\nu) = 0$, and
\[ \| e^{-\nu \cdot \x} v^\nu - w^\nu \|_{L^{q^\prime}(\R; L^{r^\prime}(H_{\hat{\nu}}))} < \infty. \]

If we additionally assume that $e^{c|\x|} F \in L^q((0,T); L^r(\R^n))$ for some $c>0$, then,
\[\int_\Sigma (i\partial_\tm + \Delta) u \overline{v^\nu} = \int_\Sigma u \overline{(i\partial_\tm + \Delta)v^\nu} \, \]
for all $\nu \in \R^n \setminus \{ 0 \}$ with $|\nu| < c$.
\end{proposition}

\begin{remark} When reading this statement, it could be helpful for 
the reader to keep in mind a CGO solution 
$u = e^{i|\nu|\tm + \nu \cdot \x} (u^\sharp + u^\flat)$
as the one in the \cref{th:CGO}. In the setting of 
the \cref{prop:CGO-PhS_int-by-parts}, $v^\nu$ and $w^\nu$ are playing the roles of 
$\chi u$ and $\chi e^{i|\nu|\tm} u^\flat$ respectively,
with $\chi $ a bump function 
of the time variable taking value $1$ in the interval $(0,T)$.
\end{remark}

\begin{remark} \label{rem:integration_by_parts}
Note that $(i\partial_\tm + \Delta) u \overline{v^\nu}$ and 
$ u \overline{(i\partial_\tm + \Delta)v^\nu}$ are in $L^1(\Sigma)$, and therefore,
each side of the integration-by-parts formula of the \cref{prop:CGO-PhS_int-by-parts} makes
sense.

\begin{enumerate}[label=\textnormal{(\alph*)}, ref=\alph*]
\item In order to see that $ u \overline{(i\partial_\tm + \Delta)v^\nu} \in L^1(\Sigma)$,
we note that by the \cref{lem:exp_decay} we have
$e^{\nu \cdot \x} u \in L^{q^\prime} ((0, T); L^{r\prime}(\R^n))$,
while $e^{-\nu \cdot \x} (i\partial_\tm + \Delta) v^\nu \in L^q(\R; L^r(\R^n))$ by assumption.

\item \label{item:integration_by_parts}
The fact $(i\partial_\tm + \Delta) u \overline{v^\nu} \in L^1(\Sigma)$, or equivalently 
$F \overline{v^\nu} \in L^1(\Sigma)$, can be derived as follows: If we write
\[F \overline{v^\nu} = e^{\nu \cdot \x} F [\overline{e^{-\nu \cdot \x} v^\nu - w^\nu}] + e^{\nu \cdot \x} F \overline{w^\nu}, \]
we can apply Hölder's inequality to obtain
\begin{align*}
\int_\Sigma \big| F \overline{v^\nu} \big| \leq \enspace & \int_0^T
\| e^{-\nu \cdot \x} v^\nu(t, \centerdot) - w^\nu(t, \centerdot) \|_{L^{r^\prime}(H_{\hat{\nu}})}
\Big( \int_\R \| e^{\nu \cdot \x} F(t, \centerdot) \|_{L^r(H_{\hat{\nu}, s})} \, \dd s \Big) \, \dd t \\
& + \| e^{\nu \cdot \x} F \|_{L^q((0,T); L^r(\R^n))} \| w^\nu \|_{L^{q^\prime}(\R; L^{r^\prime}(\R^n))}.
\end{align*}
Here we have used that $\nu \cdot \nabla (e^{-\nu \cdot \x} v^\nu - w^\nu) = 0$.

Since
\[ \int_\R \| e^{\nu \cdot \x} F(t, \centerdot) \|_{L^r(H_{\hat{\nu}, s})} \, \dd s
\leq \Big( \int_\R e^{-r^\prime(c|s| -|\nu|s)} \, \dd s \Big)^{1/r^\prime} \| e^{c |\hat{\nu} \cdot \x|} F(t, \centerdot) \|_{L^r(\R^n)}, \]
we have that
\begin{align*}
\int_\Sigma \big| F \overline{v^\nu} \big| \lesssim \enspace & \int_0^T
\| e^{-\nu \cdot \x} v^\nu(t, \centerdot) - w^\nu(t, \centerdot) \|_{L^{r^\prime}(H_{\hat{\nu}})}
\| e^{c |\x|} F(t, \centerdot) \|_{L^r(\R^n)} \, \dd t \\
& + \| e^{c |\x|} F \|_{L^q((0,T); L^r(\R^n))} \| w^\nu \|_{L^{q^\prime}(\R; L^{r^\prime}(\R^n))}.
\end{align*}

Therefore, by Hölder's inequality we conclude that
\begin{equation}
\label{in:pre_approx}
\begin{aligned}
\int_\Sigma \big| F \overline{v^\nu} \big| \lesssim \| e^{c |\x|} F & \|_{L^q((0,T); L^r(\R^n))} \\
& \big( \| e^{-\nu \cdot \x} v^\nu - w^\nu \|_{L^{q^\prime}(\R; L^{r^\prime}(H_{\hat{\nu}}))} + \| w^\nu \|_{L^{q^\prime}(\R; L^{r^\prime}(\R^n))} \big),
\end{aligned}
\end{equation}
and consequently $F \overline{v^\nu} \in L^1(\Sigma)$.
\end{enumerate}
\end{remark}

\begin{proof}
The main idea of the proof is to approximate 
$v^\nu$ by $v^\nu_\varepsilon$ so that we can integrate by parts
following the \cref{lem:integration_by_parts}. The argument goes as follows.

After proving
\begin{equation}
\label{lim:vnu}
\int_\Sigma (i\partial_\tm + \Delta) u \overline{v^\nu} = \lim_{\varepsilon \to 0} \int_\Sigma (i\partial_\tm + \Delta) u \overline{v^\nu_\varepsilon},
\end{equation}
we could apply the \cref{lem:integration_by_parts} to obtain
\[ \int_\Sigma (i\partial_\tm + \Delta) u \overline{v^\nu_\varepsilon} = \int_\Sigma u \overline{(i\partial_\tm + \Delta) v^\nu_\varepsilon}, \]
since $u(0, \centerdot) = u(T, \centerdot) = 0$.
Then, to conclude the identity of the statement, we would only need to check that
\begin{equation}
\label{lim:Sch_vnu}
\lim_{\varepsilon \to 0} \int_\Sigma u \overline{(i\partial_\tm + \Delta) v^\nu_\varepsilon} = \int_\Sigma u \overline{(i\partial_\tm + \Delta) v^\nu} .
\end{equation}

The rest of the proof consists of defining the $v^\nu_\varepsilon$ that will approximate $v^\nu$,
and showing that \eqref{lim:vnu} and \eqref{lim:Sch_vnu} hold.

Consider a smooth cut-off $\chi \in \mathcal{S}(\R^n)$ 
such that $ 0 \leq \chi(x) \leq 1 $ for 
all $x \in \R^n$, $\supp \chi \subset \{ x \in \R^n: |x| \leq 2 \}$ and
$ \chi(x) = 1 $ whenever $|x| \leq 1$. Consider $\varphi \in \mathcal{S}(\R \times \R^n)$
such that $\supp \varphi \subset [-1, 1] \times K$ with $K$ a compact subset of $\R^n$, 
$\int_{\R \times \R^n} \varphi = 1$ and $\varphi(t, x) \in [0, \infty)$
for all $(t, x) \in \R \times \R^n$. Whenever $R, \varepsilon > 0$, define 
$\chi^R (x) = \chi (x/R)$ for $x \in \R^n$,
$\varphi_\varepsilon (t, x) = \varepsilon^{-(n+1)} \varphi (t/\varepsilon, x/\varepsilon) $ 
for $(t, x) \in \R \times \R^n$. Define
\[ v^\nu_\varepsilon = e^{\nu \cdot \x} \chi^{R(\varepsilon)} \varphi_\varepsilon \ast (e^{-\nu \cdot \x} v^\nu), \]
where $\varepsilon \in (0, 1) \mapsto R(\varepsilon) \in (0, \infty)$ is continuous and 
$\lim_{\varepsilon \to 0} R(\varepsilon) = \infty$.

The identity \eqref{lim:vnu} is a consequence of 
\begin{equation}
\label{lim:Fvn-vneps}
\lim_{\varepsilon \to 0} \Big| \int_\Sigma F (\overline{v^\nu - v^\nu_\varepsilon}) \Big| = 0.
\end{equation}

In order to check that this limit holds, we re-write $F (\overline{v^\nu - v^\nu_\varepsilon})$
in a similar form as we wrote $F \overline{v^\nu}$ in the 
item \eqref{item:integration_by_parts} of
the \cref{rem:integration_by_parts} and we derive an inequality similar to
\eqref{in:pre_approx}. Indeed, from the identity
\[ e^{-\nu \cdot \x} (v^\nu - v^\nu_\varepsilon) = (e^{-\nu \cdot \x} v^\nu - w^\nu) - \chi^{R(\varepsilon)} \varphi_\varepsilon \ast (e^{-\nu \cdot \x} v^\nu - w^\nu) + w^\nu - \chi^{R(\varepsilon)} \varphi_\varepsilon \ast w^\nu, \]
we have that
\begin{align*}
F (\overline{v^\nu - v^\nu_\varepsilon}) &= e^{\nu \cdot \x} F [\overline{(e^{-\nu \cdot \x} v^\nu - w^\nu) - \chi^{R(\varepsilon)} \varphi_\varepsilon \ast (e^{-\nu \cdot \x} v^\nu - w^\nu)}] \\
&\quad + e^{\nu \cdot \x} F [\overline{w^\nu - \chi^{R(\varepsilon)} \varphi_\varepsilon \ast w^\nu}].
\end{align*}

Arguing as we did to obtain the inequality \eqref{in:pre_approx} we have
\begin{align*}
& \int_\Sigma \big| F (\overline{v^\nu - v^\nu_\varepsilon}) \big|  \lesssim \| e^{c |\x|} F \|_{L^q((0,T); L^r(\R^n))} \big( \| w^\nu - \chi^{R(\varepsilon)} \varphi_\varepsilon \ast w^\nu \|_{L^{q^\prime}(\R; L^{r^\prime}(\R^n))}\\
& \quad +
  \| (e^{-\nu \cdot \x} v^\nu - w^\nu) - \chi^{R(\varepsilon)} \varphi_\varepsilon \ast (e^{-\nu \cdot \x} v^\nu - w^\nu) \|_{L^{q^\prime}(\R; L^{r^\prime}(H_{\hat{\nu}}))} \big).
\end{align*}

After this inequality, \eqref{lim:Fvn-vneps} holds if
\begin{equation}
\label{lim:u^nu_eps}
\lim_{\varepsilon \to 0} \Big[ \| (e^{-\nu \cdot \x} v^\nu - w^\nu) - \chi^{R(\varepsilon)} \varphi_\varepsilon \ast (e^{-\nu \cdot \x} v^\nu - w^\nu) \|_{L^{q^\prime}(\R; L^{r^\prime}(H_{\hat{\nu}}))} \Big] = 0
\end{equation}
and
\begin{equation}
\label{lim:w^nu_eps}
\lim_{\varepsilon \to 0} \| w^\nu - \chi^{R(\varepsilon)} \varphi_\varepsilon \ast w^\nu \|_{L^{q^\prime}(\R; L^{r^\prime}(\R^n))} = 0.
\end{equation}
At this point, it is a standard procedure to check that \eqref{lim:u^nu_eps} and \eqref{lim:w^nu_eps} hold.

This ends the proof of \eqref{lim:vnu}. We now turn our attention to \eqref{lim:Sch_vnu}.

Applying Leibniz's rule we obtain
\begin{align*}
(i\partial_\tm + \Delta) v^\nu_\varepsilon &= e^{\nu \cdot \x} (i\partial_\tm + \Delta + 2 \nu \cdot \nabla + |\nu|^2) [\chi^{R(\varepsilon)} \varphi_\varepsilon \ast (e^{-\nu \cdot \x} v^\nu)] \\
&=  e^{\nu \cdot \x} \chi^{R(\varepsilon)} (i\partial_\tm + \Delta + 2 \nu \cdot \nabla + |\nu|^2) [\varphi_\varepsilon \ast (e^{-\nu \cdot \x} v^\nu)] \\
& \quad + e^{\nu \cdot \x}\frac{2}{R(\varepsilon)} \nabla \chi (\centerdot/R(\varepsilon)) \cdot \big( \nabla [\varphi_\varepsilon \ast (e^{-\nu \cdot \x} v^\nu)] + \varphi_\varepsilon \ast (e^{-\nu \cdot \x} v^\nu) \nu\big) \\
& \quad + e^{\nu \cdot \x}\frac{1}{R(\varepsilon)^2} \Delta \chi (\centerdot/R(\varepsilon)) \, \varphi_\varepsilon \ast (e^{-\nu \cdot \x} v^\nu).
\end{align*}

Observe that part of the first summand on the right-hand side of the second equality can be re-written as 
follows:
\begin{align*}
(i\partial_\tm + \Delta + 2 \nu \cdot \nabla + |\nu|^2) [\varphi_\varepsilon \ast (e^{-\nu \cdot \x} v^\nu)] & = \varphi_\varepsilon \ast [(i\partial_\tm + \Delta + 2 \nu \cdot \nabla + |\nu|^2)(e^{-\nu \cdot \x} v^\nu)] \\
& = \varphi_\varepsilon \ast [e^{-\nu \cdot \x} (i\partial_\tm + \Delta)v^\nu].
\end{align*}

Now, the whole first summand that we were referring to, can be written as
\begin{align*}
e^{\nu \cdot \x} \chi^{R(\varepsilon)} &(i\partial_\tm + \Delta + 2 \nu \cdot \nabla + |\nu|^2) [\varphi_\varepsilon \ast (e^{-\nu \cdot \x} v^\nu)] \\
&= (i\partial_\tm + \Delta) v^\nu + e^{\nu \cdot \x} \big( \chi^{R(\varepsilon)} \varphi_\varepsilon \ast [e^{-\nu \cdot \x} (i\partial_\tm + \Delta)v^\nu] - e^{-\nu \cdot \x} (i\partial_\tm + \Delta) v^\nu \big).
\end{align*}

Therefore,
\begin{align*}
(i\partial_\tm + \Delta) v^\nu_\varepsilon &= (i\partial_\tm + \Delta) v^\nu + e^{\nu \cdot \x} \big( \chi^{R(\varepsilon)} \varphi_\varepsilon \ast [e^{-\nu \cdot \x} (i\partial_\tm + \Delta)v^\nu] - e^{-\nu \cdot \x} (i\partial_\tm + \Delta) v^\nu \big) \\
& \quad + e^{\nu \cdot \x}\frac{2}{R(\varepsilon)} \nabla \chi (\centerdot/R(\varepsilon)) \cdot \big( \nabla [\varphi_\varepsilon \ast (e^{-\nu \cdot \x} v^\nu)] + \varphi_\varepsilon \ast (e^{-\nu \cdot \x} v^\nu) \nu\big) \\
& \quad + e^{\nu \cdot \x}\frac{1}{R(\varepsilon)^2} \Delta \chi (\centerdot/R(\varepsilon)) \, \varphi_\varepsilon \ast (e^{-\nu \cdot \x} v^\nu).
\end{align*}

In the view of last identity, \eqref{lim:Sch_vnu} holds if the following limits vanish:
\begin{gather}
\label{lim:exp_Sch_vnu}
\lim_{\varepsilon \to 0} \int_\Sigma e^{\nu \cdot \x} u \big( \chi^{R(\varepsilon)} \varphi_\varepsilon \ast [e^{-\nu \cdot \x} \overline{(i\partial_\tm + \Delta) v^\nu}] - e^{-\nu \cdot \x} \overline{(i\partial_\tm + \Delta) v^\nu} \big) \, ,\\
\label{lim:nabla_chi}
\lim_{\varepsilon \to 0} \frac{2}{R(\varepsilon)} \int_\Sigma e^{\nu \cdot \x} u \nabla \chi (\centerdot/R(\varepsilon)) \cdot \big( \nabla [\varphi_\varepsilon \ast (e^{-\nu \cdot \x} \overline{v^\nu})] + \varphi_\varepsilon \ast (e^{-\nu \cdot \x} \overline{v^\nu}) \nu\big) \, ,\\
\label{lim:Delta_chi}
\lim_{\varepsilon \to 0} \frac{1}{R(\varepsilon)^2} \int_\Sigma e^{\nu \cdot \x} u \Delta \chi (\centerdot/R(\varepsilon)) \, \varphi_\varepsilon \ast (e^{-\nu \cdot \x} \overline{v^\nu}) \, .
\end{gather}

The limit \eqref{lim:exp_Sch_vnu} vanishes because 
$e^{\nu \cdot \x} u $ belongs to $L^{q^\prime} ((0, T); L^{r\prime}(\R^n))$,
$e^{-\nu \cdot \x} (i\partial_\tm + \Delta) v^\nu \in L^q(\R; L^r(\R^n))$ by assumption, and
\[ \lim_{\varepsilon \to 0} \| \chi^{R(\varepsilon)} \varphi_\varepsilon \ast [e^{-\nu \cdot \x} \overline{(i\partial_\tm + \Delta) v^\nu}] - e^{-\nu \cdot \x} \overline{(i\partial_\tm + \Delta) v^\nu} \|_{L^q(\R; L^r(\R^n))} = 0. \]

Note that $e^{\nu \cdot \x} u \in L^{q^\prime} ((0, T); L^{r\prime}(\R^n))$ after
the \cref{lem:exp_decay}, while the previous limit follows from the same type of arguments 
that guarantees that \eqref{lim:u^nu_eps} and \eqref{lim:w^nu_eps} hold.

Next we will show that
\begin{equation}
\label{lim:worst}
\lim_{\varepsilon \to 0} \frac{2}{R(\varepsilon)} \int_\Sigma e^{\nu \cdot \x} u \nabla \chi (\centerdot/R(\varepsilon)) \cdot \nabla [\varphi_\varepsilon \ast (e^{-\nu \cdot \x} \overline{v^\nu})] \, = 0.
\end{equation}

Start by noticing that
\begin{align*}
& \Big| \int_\Sigma e^{\nu \cdot \x} u \nabla \chi (\centerdot/R(\varepsilon)) \cdot \nabla [\varphi_\varepsilon \ast (e^{-\nu \cdot \x} \overline{v^\nu})] \Big| \\
& \lesssim \frac{1}{\varepsilon} \| e^{c^\prime |\x|} u \|_{L^{q^\prime} ((0, T); L^{r\prime}(\R^n))} \Big[ \| e^{-c^\prime|\x| + \nu \cdot \x} \|_{L^b(\R^n)} \| w^\nu \|_{L^{q^\prime}(\R; L^{r^\prime}(\R^n))}\\
& + \| e^{-\nu \cdot \x} v^\nu - w^\nu \|_{L^{q^\prime}(\R; L^{r^\prime}(H_{\hat{\nu}}))}
\Big(\int_\R \| e^{-c^\prime|\x| + \nu \cdot \x} \|_{L^b(H_{\hat{\nu}, s})}^r \, \dd s \Big)^{1/r} \Big].
\end{align*}
Here we have used that 
$\| \nabla \chi (\centerdot/R(\varepsilon)) \|_{L^\infty(\R^n)} \lesssim 1$ and 
$ \| \nabla \varphi_\varepsilon \|_{L^1(\R \times \R^n)} \lesssim 1/\varepsilon$.

After the \cref{rem:exp_decay}, we have that 
$e^{c^\prime |\x|} u \in L^{q^\prime} ((0, T); L^{r\prime}(\R^n))$ whenever $c^\prime < c$.
Furthermore, if $|\nu| < c^\prime$ we have that
\[ \| e^{-c^\prime|\x| + \nu \cdot \x} \|_{L^b(\R^n)} + \Big(\int_\R \| e^{-c^\prime|\x| + \nu \cdot \x} \|_{L^b(H_{\hat{\nu}, s})}^r \, \dd s \Big)^{1/r} < \infty. \]

Then, for $c^\prime$ so that $|\nu| < c^\prime < c$ we have that
\begin{align*}
\Big| \frac{1}{R(\varepsilon)} \int_\Sigma e^{\nu \cdot \x} u \nabla \chi (\centerdot/R(\varepsilon)) & \cdot \nabla [\varphi_\varepsilon \ast (e^{-\nu \cdot \x} \overline{v^\nu})] \Big|
\lesssim \frac{1}{\varepsilon R(\varepsilon)} \| e^{c^\prime |\x|} u \|_{L^{q^\prime} ((0, T); L^{r\prime}(\R^n))} \\
& \Big[ \| w^\nu \|_{L^{q^\prime}(\R; L^{r^\prime}(\R^n))} + \| e^{-\nu \cdot \x} v^\nu - w^\nu \|_{L^{q^\prime}(\R; L^{r^\prime}(H_{\hat{\nu}}))} \Big].
\end{align*}
If the function $\varepsilon \in (0, 1) \mapsto R(\varepsilon) \in (0, \infty)$ is chosen so that
$\lim_{\varepsilon \to 0} \varepsilon R(\varepsilon) = \infty$, then \eqref{lim:worst}
holds.\footnote{Using the dominate convergence theorem it would have been enough to ensure that 
$\varepsilon R(\varepsilon) \lesssim 1$.}

With the same type of arguments one can show that \eqref{lim:Delta_chi} and
\[\lim_{\varepsilon \to 0} \frac{2}{R(\varepsilon)} \int_\Sigma e^{\nu \cdot \x} u \, \nu \cdot \nabla \chi (\centerdot/R(\varepsilon)) \, \varphi_\varepsilon \ast (e^{-\nu \cdot \x} \overline{v^\nu}) \]
vanish.

These facts allow us to conclude that \eqref{lim:Sch_vnu} holds.
\end{proof}

\subsection{An orthogonality relation for exponentially-growing solutions}
Whenever $\mathcal{U}_T^1 = \mathcal{U}_T^2$, the integral identity \eqref{id:integral_identity}
becomes an orthogonality relation between $V_1 - V_2$ and the product of the physical solutions $u_1 $ and
$\overline{v_2}$. Our goal here is to prove that the same orthogonality relation
holds for the product of exponentially-growing solutions. A key ingredient in our argument is 
the integration-by-parts formula of the \cref{prop:CGO-PhS_int-by-parts}.
For that reason, it is convenient to define precisely the concept of admissible 
exponentially-growing solutions.

\begin{definition}\label{def:exp-gro}\sl
Consider $V \in L^a((0, T); L^b(\R^n))$ for 
$(a, b) \in [1, \infty] \times [1, \infty]$ satisfying \eqref{cond:V_ab}
such that $e^{c|\x|} V \in L^a((0, T); L^b(\R^n))$.

Given $\nu \in \R^n \setminus \{ 0 \}$ with $|\nu| < c$, 
we say that $u: \R \times \R^n \rightarrow \C$ is a 
\emph{$\nu$-admissible exponentially-growing solution} if 
\begin{enumerate}[label=\textnormal{(\alph*)}, ref=\textnormal{\alph*}]
\item it is a measurable and locally 
integrable solution of 
\[ (i\partial_\tm + \Delta - V) u = 0 \enspace \textnormal{in} \enspace (0,T) \times \R^n, \]

\item there is $v \in L^{q^\prime}(\R; L^{r^\prime}(\R^n))$ with $q$
and $r$ satisfying \eqref{id:indices4Vab} 
such that $\nu \cdot \nabla (e^{-\nu \cdot \x} u - v) = 0$ and
\[ \| e^{-\nu \cdot \x} u - v \|_{L^{q^\prime}(\R; L^{r^\prime}(H_{\hat{\nu}}))} < \infty. \]
\end{enumerate}
\end{definition}

\begin{remark} Again, it could be helpful for 
the reader to keep in mind a CGO solution 
$w = e^{i|\nu|\tm + \nu \cdot \x} (w^\sharp + w^\flat)$
as the one in the \cref{th:CGO}. Then, for 
the \cref{def:exp-gro}, $u$ and $v$ are playing the roles of 
$\chi w$ and $\chi e^{i|\nu|\tm} w^\flat$ respectively,
with $\chi $ a bump function 
of the time variable taking value $1$ in the interval $(0,T)$.
\end{remark}

\begin{remark}\label{rem:Sch_vnu}
If $u$ is a $\nu$-admissible exponentially-growing solution, then
$e^{-\nu \cdot \x} (i\partial_\tm + \Delta) u \in L^q(\R; L^r(\R^n))$.

Indeed, we
only have to check that $e^{-\nu \cdot \x} V u $ belongs to $ L^q(\R; L^r(\R^n))$:
\[ \| e^{-\nu \cdot \x} V u \|_{L^q(\R; L^r(\R^n))} \leq \|  V (e^{-\nu \cdot \x} u - v) \|_{L^q(\R; L^r(\R^n))} + \| V v \|_{L^q(\R; L^r(\R^n))} \]

By Hölder's inequality we have that
\[ \| V v \|_{L^q(\R; L^r(\R^n))} \leq \| V \|_{L^a(\R; L^b(\R^n))} \| v \|_{L^{q^\prime}(\R; L^{r^\prime}(\R^n))} \]
and
\begin{equation*}
\|  V (e^{-\nu \cdot \x} u - v) \|_{L^q(\R; L^r(\R^n))} \lesssim 
\| e^{-\nu \cdot \x} u - v \|_{L^{q^\prime}(\R; L^{r^\prime}(H_{\hat{\nu}}))}
\| e^{\varepsilon |\hat{\nu} \cdot \x|} V \|_{L^a(\R; L^b(\R^n))}
\end{equation*}
for $\varepsilon < |\nu|$, where the implicit constant is
\[ \Big( \int_\R e^{-r^\prime \varepsilon |s|} \, \dd s \Big)^{1/r^\prime}. \]
\end{remark}

\begin{remark}\label{rem:erho_Vtilde_vnu}
Let $u$ be a $\nu$-admissible exponentially-growing solution for a potential $V$
as in the \cref{def:exp-gro}, and 
consider $\tilde V $ in the space $ L^a((0, T); L^b(\R^n))$, with 
$(a, b) \in [1, \infty] \times [1, \infty]$ satisfying \eqref{cond:V_ab}.

If we additionally assume that $e^{c|\x|} \tilde V \in L^a((0, T); L^b(\R^n))$, then
$e^{\rho |\x|} \tilde V u \in L^q(\R; L^r(\R^n))$ for all $\rho \geq 0$ such that
$|\nu| + \rho < c$.

Indeed, 
\begin{align*}
 \|e^{\rho |\x|}  \tilde V u & \|_{L^q(\R; L^r(\R^n))} \\
& \leq \| e^{\rho |\x| + \nu \cdot \x} \tilde V (e^{-\nu \cdot \x} u - v) \|_{L^q(\R; L^r(\R^n))} + \| e^{\rho |\x| + \nu \cdot \x} \tilde V v \|_{L^q(\R; L^r(\R^n))}.
\end{align*}

By Hölder's inequality we have that
\[ \| e^{\rho |\x| + \nu \cdot \x} \tilde V v \|_{L^q(\R; L^r(\R^n))} \leq \| e^{c|\x|} \tilde V \|_{L^a(\R; L^b(\R^n))} \| v \|_{L^{q^\prime}(\R; L^{r^\prime}(\R^n))} \]
and
\begin{align*}
\| & e^{\rho |\x| + \nu \cdot \x} \tilde V (e^{-\nu \cdot \x} u - v) \|_{L^q(\R; L^r(\R^n))} \\
& \leq \| e^{-\nu \cdot \x} u - v \|_{L^{q^\prime}(\R; L^{r^\prime}(H_{\hat{\nu}}))}
\Big( \int_\R e^{-r^\prime ((c -\rho) |s| - |\nu| s)} \, \dd s \Big)^{1/r^\prime}
\| e^{c |\x|} \tilde V \|_{L^a(\R; L^b(\R^n))}
\end{align*}
for $|\nu|< c - \rho$.
\end{remark}

\begin{remark}\label{rem:Vtilde_u1_u2}
Let $u_1^{\nu_1}$ and $u_2^{\nu_2}$ be ${\nu_1}$- and ${\nu_2}$-admissible exponentially-growing solutions for potentials $V_1$ and $V_2$ as in the \cref{def:exp-gro}.

If $\tilde V $ is as in the \cref{rem:erho_Vtilde_vnu} with 
$e^{c|\x|} \tilde V \in L^a((0, T); L^b(\R^n))$, then
$\tilde V u_1^{\nu_1} u_2^{\nu_2} \in L^1(\Sigma)$ whenever $|{\nu_1}| + |{\nu_2}|< c$.

Indeed, we have that
\begin{align*}
\tilde V u_1^{\nu_1} u_2^{\nu_2} 
&= e^{({\nu_1} + {\nu_2}) \cdot \x} \tilde V (e^{-{\nu_1} \cdot x} u_1^{\nu_1} - v_1^{\nu_1}) (e^{-{\nu_2} \cdot x} u_2^{\nu_2} - v_2^{\nu_2})\\
&\enspace + e^{({\nu_1} + {\nu_2}) \cdot \x} \tilde V (e^{-{\nu_1} \cdot x} u_1^{\nu_1} - v_1^{\nu_1}) v_2^{\nu_2} + e^{({\nu_1} + {\nu_2}) \cdot \x} \tilde V  v_1^{\nu_1} (e^{-{\nu_2} \cdot x} u_2^{\nu_2} - v_2^{\nu_2})\\
&\enspace + e^{({\nu_1} + {\nu_2}) \cdot \x} \tilde V  v_1^{\nu_1} v_2^{\nu_2}
\end{align*}

Writing $w_j^{\nu_j} = (e^{-{\nu_j} \cdot x} u_j^{\nu_j} - v_j^{\nu_j})$ we have that
\begin{align*}
\tilde V u_1^{\nu_1} u_2^{\nu_2} 
&= e^{({\nu_1} + {\nu_2}) \cdot \x + \varepsilon (|{{\hat \nu_1}} \cdot \x| + |{{\hat \nu_2}} \cdot \x|)} \tilde V e^{-\varepsilon |{{\hat \nu_1}} \cdot \x|} w_1^{\nu_1} e^{-\varepsilon |{{\hat \nu_2}} \cdot \x|}w_2^{\nu_2} \\
& \enspace +e^{({\nu_1} + {\nu_2}) \cdot \x+ \varepsilon |{{\hat \nu_1}} \cdot \x|} \tilde V e^{-\varepsilon |{{\hat \nu_1}} \cdot \x|}w_1^{\nu_1} v_2^{\nu_2} + e^{({\nu_1} + {\nu_2}) \cdot \x + \varepsilon |{{\hat \nu_2}} \cdot \x|} \tilde V  v_1^{\nu_1} e^{-\varepsilon |{{\hat \nu_2}} \cdot \x|} w_2^{\nu_2}\\
& \enspace + e^{({\nu_1} + {\nu_2}) \cdot \x} \tilde V  v_1^{\nu_1} v_2^{\nu_2}.
\end{align*}

Noticing that
\[\| e^{-\varepsilon |{{\hat \nu_j}} \cdot \x|} w_j^{\nu_j} \|_{L^{q^\prime}(\R; L^{r^\prime}(\R^n))} \lesssim \| e^{-{\nu_j} \cdot x} u_j^{\nu_j} - v_j^{\nu_j} \|_{L^{q^\prime}(\R; L^{r^\prime}(H_{\hat{\nu}_j}))}, \]
and choosing $\varepsilon > 0$ so that $|{\nu_1}| + |\nu_2| + 2\varepsilon < c$, we can apply
Hölder inequality to conclude that $\tilde V u_1^{\nu_1} u_2^{\nu_2} \in L^1(\Sigma)$ whenever 
$|{\nu_1}| + |{\nu_2}|< c$.
\end{remark}

\begin{theorem}\label{th:orthogolaity_CGO} \sl 
Consider $V_1, V_2 \in L^a((0, T); L^b(\R^n))$ with 
$(a, b) \in [1, \infty] \times [1, \infty]$ satisfying \eqref{cond:V_ab}.
If $(a, b) = (\infty, n/2)$ with $n \geq 3$, we also suppose 
$V_1 $ and $ V_2  $ belong to $ C ([0, T]; L^{n/2}(\R^n))$.

Let $\mathcal{U}_T^1$ and $\mathcal{U}_T^2$ denote the initial-to-final-state
maps associated to $V_1$ and $V_2$ respectively.

Then, if we additionally assume that there 
exists $c > 0$ so that $e^{c |\x|} V_j \in L^a((0, T); L^b(\R^n))$ for $j \in \{ 1, 2 \}$,
the equality 
$\mathcal{U}_T^1 = \mathcal{U}_T^2$ implies that
\[ \int_\Sigma (V_1 - V_2)u_1^\eta \overline{v_2^\nu}\, = 0 \]
whenever $|\eta| + |\nu| < c$ and
for all $\eta$- and $\nu$-admissible exponentially-growing solutions $u_1^\eta$ and $v_2^\nu$ 
of the equations 
\[ (i\partial_\tm + \Delta - V_1) u_1^\eta = 0 \enspace \textnormal{and} \enspace  (i\partial_\tm + \Delta - \overline{V_2}) v_2^\nu = 0 \enspace \textnormal{in} \enspace (0,T) \times \R^n. \]
\end{theorem}

\begin{proof}
We start by proving that if $\mathcal{U}_T^1 = \mathcal{U}_T^2$ then
\[ \int_\Sigma (V_1 - V_2)u_1 \overline{v_2^\nu}\, = 0 \]
for every $u_1$ physical solution of \eqref{pb:IVP} with potential $V_1$,
and every $\nu$-admissible exponentially-growing solution $v_2^\nu $,
with $\nu \in \R^n \setminus \{ 0 \}$ and $|\nu| < c$.

For every $u_1 \in X^\star$ solution of \eqref{pb:IVP} with potential 
$V_1 \in Y$, we have that $(V_1 - V_2)u_1 \in X$ since 
$V_1 - V_2 \in Y$. Let $w_2 \in X^\star$ be the solution
of the problem 
\[\left\{
		\begin{aligned}
		& (i\partial_\tm + \Delta -  V_2)w_2 = (V_1 - V_2)u_1 & & \textnormal{in} \, \Sigma, \\
		& w_2(0, \centerdot) = 0 &  & \textnormal{in} \, \R^n.
		\end{aligned}
\right.\]

After the identity $\mathcal{U}_T^1 = \mathcal{U}_T^2$,
we can apply the \cref{lem:orthogonal_rewardPB} with $F = (V_1 - V_2)u_1$
and using \eqref{id:integral_identity} to deduce that $w_2(T, \centerdot) = 0$.

Furthermore, since
$e^{c |\x|} V_j $ belong to $L^a((0, T); L^b(\R^n))$ for $j \in \{ 1, 2 \}$,
we have that $e^{c |\x|} (i\partial_\tm + \Delta)w_2 \in L^q((0, T); L^r(\R^n)) $ with $q$
and $r$ satisfying \eqref{id:indices4Vab}. One can check that, in fact,
$(q, r) \in [1, 2] \times [1, 2]$ satisfies \eqref{id:Strichartz_indeces}.
This means that $w_2$ shares
the same properties as $u$ in the \cref{prop:CGO-PhS_int-by-parts} for all
$\nu \in \R^n \setminus \{ 0\}$ with $|\nu| < c$.

Furthermore, every $\nu$-admissible exponentially-growing solution $v_2^\nu$ satisfies the 
same conditions as $v^\nu$ in the \cref{prop:CGO-PhS_int-by-parts} since 
$e^{-\nu \cdot \x} (i\partial_\tm + \Delta) v_2^\nu \in L^q(\R; L^r(\R^n))$, as we observed in 
the \cref{rem:Sch_vnu}.

Then, $w_2$ and $v_2^\nu$ satisfy the hypothesis of the 
\cref{prop:CGO-PhS_int-by-parts}, for all $\nu \in \R^n \setminus \{ 0\}$ with 
$|\nu| < c$, and we have the 
following integration-by-parts formula
\[\int_\Sigma (i\partial_\tm + \Delta) w_2 \overline{v_2^\nu} = \int_\Sigma w_2 \overline{(i\partial_\tm + \Delta)v_2^\nu} \,. \]

Observe that $e^{c|\x|}(V_1 - V_2)u_1 \in L^q((0, T); L^r(\R^n)) $ since 
$e^{c|\x|} (V_1 - V_2)$ belongs to $ L^a((0, T); L^b(\R^n))$ and $u_1 \in X^\star$.
By the exact same argument in \eqref{item:integration_by_parts} of the 
\cref{rem:integration_by_parts}, we can ensure that
$(V_1 - V_2)u_1 \overline{v_2^\nu} \in L^1(\Sigma)$ whenever $|\nu| < c$.

As a consequence of this fact, we can write
\[ \int_\Sigma (V_1 - V_2)u_1 \overline{v_2^\nu}\, = \int_\Sigma (i\partial_\tm + \Delta - V_2)w_2 \overline{v_2^\nu}\, = \int_\Sigma w_2 \overline{(i\partial_\tm + \Delta - \overline{V_2}) v_2^\nu}\, = 0. \]
In the first identity we have used that
$(i\partial_\tm + \Delta - V_2)w_2 = (V_1 - V_2)u_1$, in the second one we have
used the previous integration-by-parts formula for $w_2$ and $v_2^\nu$,
and in the last equality we have used that $(i\partial_\tm + \Delta - \overline{V_2}) v_2^\nu = 0$ in $\Sigma$.

This proves that
if $\mathcal{U}_T^1 = \mathcal{U}_T^2$ then
\begin{equation}
\label{id:orthogonal_phy-exp}
\int_\Sigma (V_1 - V_2)u_1 \overline{v_2^\nu}\, = 0
\end{equation}
for every $u_1$ physical solution with potential $V_1$,
and every $\nu$-admissible exponentia-lly-growing solution $v_2^\nu$,
with $\nu \in \R^n \setminus \{ 0 \}$ and $|\nu| < c$.

Finally, we will prove that the orthogonality relation 
\eqref{id:orthogonal_phy-exp} yields 
the one in the statement, that is,
\[ \int_\Sigma (V_1 - V_2)u_1^\eta \overline{v_2^\nu}\, = 0 \]
for every $\eta$- and $\nu$-admissible exponentially-growing solutions $u_1^\eta$ and $v_2^\nu$ with $|\eta| + |\nu| < c$.

Start by noticing that the \cref{rem:erho_Vtilde_vnu} with $\rho=0$ ensures that
$(\overline{V_1} - \overline{V_2}) v_2^\nu $ belongs to 
$ L^q((0,T); L^r(\R^n)) \subset X$
for all $\nu \in \R^n \setminus \{ 0\}$ with $|\nu| < c$.
Then, let $w_1^\nu \in X^\star$ be the solution
of the problem
\[\left\{
		\begin{aligned}
		& (i\partial_\tm + \Delta - \overline{V_1})w_1^\nu = (\overline{V_1} - \overline{V_2}) v_2^\nu & & \textnormal{in} \, \Sigma, \\
		& w_1(T, \centerdot) = 0 &  & \textnormal{in} \, \R^n.
		\end{aligned}
\right.\]

Since $ V_1 \in Y$ and $(\overline{V_1} - \overline{V_2}) v_2^\nu \in X$,
we can use \eqref{id:orthogonal_phy-exp} to  apply the \cref{lem:orthogonal_forwardPB} with 
$G = (\overline{V_1} - \overline{V_2}) v_2^\nu$ and deduce that $w_1^\nu(0, \centerdot) = 0$.
As we argued earlier to prove \eqref{id:orthogonal_phy-exp}, we want to show that,
given $\nu \in \R^n \setminus \{ 0 \}$ with  $|\nu| < c$, we have that
$ e^{c^\prime |\x|} (i\partial_\tm + \Delta)w_1^\nu \in L^q((0, T); L^r(\R^n)) $
for all $c^\prime \in (0, c - |\nu|)$.
In this way, we could conclude that $w_1^\nu$,
with $\nu \in \R^n \setminus \{ 0\}$ and $|\nu| < c$, shares
the same properties as $u$ in the \cref{prop:CGO-PhS_int-by-parts} for all 
$c^\prime < c - |\nu|$.

Since $e^{c |\x|} \overline{V_1} w_1^\nu \in L^q((0, T); L^r(\R^n))$, in order to prove that
$ e^{c^\prime |\x|} (i\partial_\tm + \Delta)w_1^\nu \in L^q((0, T); L^r(\R^n)) $
for all $c^\prime \in (0, c - |\nu|)$, we only have to check that
$ e^{c^\prime |\x|} (\overline{V_1} - \overline{V_2}) v_2^\nu \in L^q((0, T); L^r(\R^n)) $
for all $c^\prime < c - |\nu|$. However, this follows from the \cref{rem:erho_Vtilde_vnu}
for $\rho = c^\prime$.

Additionally, every $\eta$-admissible exponentially-growing solution
$u_1^\eta$ satisfies the same 
conditions as $v^\eta$ in the \cref{prop:CGO-PhS_int-by-parts} since 
$e^{-\eta \cdot \x} (i\partial_\tm + \Delta) u_1^\eta \in L^q(\R; L^r(\R^n))$, as we observed 
in the \cref{rem:Sch_vnu}.

Then, given $\nu \in \R^n \setminus \{ 0 \}$ with $|\nu| < c$, we have $w_1^\nu$ 
and $u_1^\eta$, with $\eta \in \R^n \setminus \{ 0 \}$ and $|\eta| < c^\prime$, 
satisfy the hypothesis of the \cref{prop:CGO-PhS_int-by-parts} for all
$c^\prime \in (0, c - |\nu|)$. Thus, we have that the integration-by-parts formula
\[\int_\Sigma (i\partial_\tm + \Delta) w_1^\nu \overline{u_1^\eta} = \int_\Sigma w_1^\nu \overline{(i\partial_\tm + \Delta)u_1^\eta} \, \]
holds for all $ \eta, \nu \in \R^n \setminus \{ 0 \}$ with $|\eta| + |\nu| < c$.

Finally, the \cref{rem:Vtilde_u1_u2} guarantees that
$(V_1 - V_2)u_1^\eta \overline{v_2^\nu} \in L^1(\Sigma)$, so we can write
\[ \int_\Sigma (V_1 - V_2)u_1^\eta \overline{v_2^\nu}\, = \int_\Sigma u_1^\eta \overline{(i\partial_\tm + \Delta - \overline{V_1}) w_1^\nu} \, = \int_\Sigma (i\partial_\tm + \Delta - V_1) u_1^\eta \overline{w_1^\nu} \, = 0. \]
In the first identity we have used that
$(i\partial_\tm + \Delta - \overline{V_1})w_1^\nu = (\overline{V_1} - \overline{V_2}) v_2^\nu$, in the second one we have
used the previous integration-by-parts formula for $w_1^\nu$ and $u_1^\eta$,
and in the last equality we have used that $(i\partial_\tm + \Delta - V_1) u_1^\eta = 0$ in $\Sigma$.

This proves that
the orthogonality relation
\[ \int_\Sigma (V_1 - V_2)u_1^\eta \overline{v_2^\nu}\, = 0 \]
holds for every every $\eta$- and $\nu$-admissible exponentially-growing solutions 
$u_1^\eta$ and $v_2^\nu$,
with $\eta, \nu \in \R^n \setminus \{ 0 \}$ and $|\eta| + |\nu| < c$.
\end{proof}

\section{Proof of the \cref{th:uniqueness}}
\label{sec:uniqueness}
Consider $V_1, V_2 \in L^a((0, T); L^b (\R^n))$ with 
$(a, b) \in [1, \infty] \times [1, \infty]$ 
satisfying \eqref{cond:V_ab}.
Recall that if $(a, b) = (\infty, n/2)$ with $n \geq 3$, we also have that
$V_1 $ and $ V_2  $ belong to $ C ([0, T]; L^{n/2}(\R^n))$.
Furthermore, for $j \in \{ 1, 2 \}$,
$e^{\rho |\x|} V_j \in L^a((0, T); L^b(\R^n))$ for all $\rho > 0$, and
\[ \sup_{\omega \in \Sph^{n-1}} \int_\R \| \mathbf{1}_{>R} V_j \|_{L^\infty ((0,T) \times H_{\omega, s})} \, \dd s < \infty\]
for $R = \max(R_1, R_2)$.

Let $V_1^{\rm ext}$ and $V_2^{\rm ext}$ denote 
denote either their extensions by zero outside $\Sigma$ if 
$(a, b) \neq (\infty, n/2)$, or two suitable continuous extensions with support in 
$[T^\prime ,T^{\prime \prime}] \times \R^n$
with $T^\prime < 0 < T < T^{\prime \prime}$ if $(a, b) = (\infty, n/2)$ with $n\geq 3$.

For $\nu \in \R^n \setminus \{ 0 \}$, choose
\[\varphi_1(t,x) = i |\nu|^2 t - \nu \cdot x , \quad \varphi_2(t,x) = i |\nu|^2 t + \nu \cdot x \qquad \forall (t, x) \in \R \times \R^n.\]

Let $u_1 = e^{\varphi_1} (u_1^\sharp + u_1^\flat)$ and 
$v_2 = e^{\varphi_2} (v_2^\sharp + v_2^\flat)$ denote the CGO solutions of
the \cref{th:CGO} for the equations
\[ (i\partial_\tm + \Delta - V_1^{\rm ext}) u_1 = 0 \enspace \textnormal{and} \enspace (i\partial_\tm + \Delta - \overline{V_2^{\rm ext}}) v_2 = 0 \enspace \textnormal{in} \enspace \R \times \R^n, \]
respectively.

For convenience, recall that for $\psi_1, \psi_2 \in \mathcal{S}(H_{\hat{\nu}})$ we chose
\begin{align*}
& u_1^\sharp (t, x) = \frac{1}{(2\pi)^\frac{n-1}{2}} \int_{H_{\hat{\nu}}} e^{ix \cdot \xi} e^{-i t |\xi|^2} {\psi_1}(\xi) \, \dd \sigma_{\hat{\nu}} (\xi) \quad \forall (t, x) \in \R \times \R^n, \\
& v_2^\sharp (t, x) = \frac{1}{(2\pi)^\frac{n-1}{2}} \int_{H_{\hat{\nu}}} e^{ix \cdot \xi} e^{-i t |\xi|^2} {\psi_2}(\xi) \, \dd \sigma_{\hat{\nu}} (\xi) \quad \forall (t, x) \in \R \times \R^n.
\end{align*}

For coherence with the choice $-\nu$ for $u_1$, we should have 
written, in the definition of $u_1^\sharp$, $H_{-\hat \nu}$ and $\sigma_{-\hat \nu}$
instead of $H_{\hat \nu}$ and $\sigma_{\hat \nu}$. But observe that 
$H_{\hat \nu} = H_{-\hat \nu}$ and $\sigma_{\hat \nu} = \sigma_{-\hat \nu}$.

We want to apply the \cref{th:orthogolaity_CGO} to conclude that
\[ \int_\Sigma (V_1 - V_2)u_1^\sharp \overline{v_2^\sharp}\, = -\int_\Sigma (V_1 - V_2)u_1^\sharp \overline{v_2^\flat}\, + \int_\Sigma (V_1 - V_2)u_1^\flat \overline{v_2^\sharp}\, + \int_\Sigma (V_1 - V_2)u_1^\flat \overline{v_2^\flat}\, . \]
In order to apply the \cref{th:orthogolaity_CGO}, we only need to take
$u^\eta = \chi u_1$ and $v^\nu = \chi v_2$ with $\chi$ a bump function of time variable 
that is identically $1$ in $(0,T)$. Note such functions are solutions in $(0,T) \times \R^n $.

By the Cauchy--Schwarz inequality, we have
\begin{align*}
\Big| \int_\Sigma (V_1 - V_2)u_1^\sharp \overline{v_2^\sharp}\, \Big| & \leq
\| |V_1 - V_2|^{1/2} u_1^\flat \|_{L^2(\R \times \R^n)} \| |V_1 - V_2|^{1/2} v_2^\sharp \|_{L^2(\R \times \R^n)} \\
& \enspace + \| |V_1 - V_2|^{1/2} u_1^\sharp \|_{L^2(\R \times \R^n)} \| |V_1 - V_2|^{1/2} v_2^\flat \|_{L^2(\R \times \R^n)} \\
& \enspace + \| |V_1 - V_2|^{1/2} u_1^\flat \|_{L^2(\R \times \R^n)} \| |V_1 - V_2|^{1/2} v_2^\flat \|_{L^2(\R \times \R^n)}.
\end{align*}

Letting $|\nu|$ tend to infinity, we deduce by the \cref{th:CGO}, more particularly by 
\eqref{boun:CGO-leading} and \eqref{lim:CGO-remainder}, that
\[ \int_\Sigma (V_1 - V_2) u_1^\sharp \overline{v_2^\sharp} \, = 0. \]

Write $F = V_1^{\rm ext} - V_2^{\rm ext}$, and note that it belongs to $L^1(\R \times \R^n)$.
By Fubini
\begin{align*}
& 0 = \int_{\R \times \R^n} F u_1^\sharp \overline{v_2^\sharp} \\
& = \frac{1}{(2\pi)^\frac{n-2}{2}} \int_{H_{\hat{\nu}} \times H_{\hat{\nu}}} \widehat{F} (|\eta|^2 - |\kappa|^2, \kappa - \eta) \psi_1(\eta) \overline{\psi_2(\kappa)} \, \dd \sigma_{\hat{\nu}} \otimes \sigma_{\hat{\nu}} (\eta, \kappa).
\end{align*}
Here $\widehat{F}$ denotes the Fourier transform of $F$:
\[ \widehat{F} (|\eta|^2 - |\kappa|^2, \kappa - \eta) = \frac{1}{(2\pi)^\frac{n+1}{2}} \int_{\R \times \R^n} F(t, x) e^{-i t |\eta|^2 + ix \cdot \eta} \overline{e^{-i t |\kappa|^2 + ix \cdot \kappa}} \, \dd (t,x).\]

Since the choices of $\psi_1$ and $\psi_2$ in $\mathcal{S}(H_{\hat{\nu}})$ are arbitrary, we can 
conclude by density that
\[ \widehat{F} (|\eta|^2 - |\kappa|^2, \kappa - \eta) = 0 \quad \forall \eta, \kappa \in H_{\hat \nu}. \]

Given $(\tau, \xi) \in \R \times \R^n$ such that $\xi \neq 0$,
we choose $\nu \in \R^n \setminus \{ 0 \}$ so that 
$\xi \cdot \nu = 0$ and, $\eta, \kappa \in \R^n$ as
\[\eta = -\frac{1}{2} \Big( 1 + \frac{\tau}{|\xi|^2} \Big) \xi, \qquad \kappa = \frac{1}{2} \Big( 1 - \frac{\tau}{|\xi|^2} \Big) \xi.\]
It is clear that $\eta, \kappa \in H_{\hat \nu}$, and after a simple computation we see that
$ |\eta|^2 - |\kappa|^2 = \tau $ and $ \kappa - \eta = \xi $. Hence
$ \widehat{F} (\tau, \xi) = 0 $ for all $(\tau, \xi) \in \R \times (\R^n \setminus \{ 0 \}) $. 

Since $F \in L^1(\R \times \R^n)$, we know that $\widehat{F}$ is 
continuous in $\R \times \R^n$, and consequently $\widehat{F}(\tau, \xi) = 0$ 
for all $(\tau, \xi) \in \R \times \R^n$.

By the injectivity of the 
Fourier transform we have that $F(t, x) = 0$ for almost every 
$(t, x) \in \R \times \R^n$, which implies that $V_1 (t, x) = V_2(t, x)$ for
almost every $(t, x) \in \Sigma$. This concludes the proof of the \cref{th:uniqueness}.

\section{Differences with the Calderón problem}\label{sec:bourgain}
In the context of the Calderón problem, there is an alternative approach that avoids the use of 
the Lavine--Nachman trick. This approach involves Bourgain spaces. In this section we explain why 
this approach fails for the initial-to-final state inverse problem.

\subsection{The Bourgain spaces in the Calderón problem}
Calder\'on posed in \cite{zbMATH05684831} the following inverse 
boundary value problem: Let $D$ be a bounded domain in $\R^n$ ($n \geq 2$) 
with Lipschitz boundary $\partial D$, and let $\gamma$ be a real bounded 
measurable function in $D$ with a positive lower bound. Consider 
the Dirichlet-to-Neumann map $\Lambda_\gamma : H^{1/2} (\partial D) \to H^{-1/2} 
(\partial D)$ 
defined by
\[ \Lambda_\gamma f = \gamma \partial_\nu u|_{\partial D} \]
where $\partial_\nu = \nu \cdot \nabla$, with $\nabla$ denoting the gradient 
and $\nu$ denoting the outward 
unit normal vector to $\partial D$,
and $ u \in H^1 (D) $ is the solution of the boundary value problem
\begin{equation*}
\left\{
\begin{aligned}
\nabla \cdot (\gamma \nabla u)  &
= 0 \,  \textnormal{in} \, D,\\
u|_{\partial D} &= f.
\end{aligned}
\right.
\end{equation*}
The inverse Calder\'on problem
is to decide whether the conductivity $\gamma$ is uniquely determined by 
$\Lambda_\gamma$, and to calculate $\gamma$ in terms of $\Lambda_\gamma$ 
whenever the unique determination is possible.

Sylvester and Uhlmann proved uniqueness for smooth conductivities in \cite{zbMATH04015323}
for dimension $n \geq 3$. However, the key points of their method 
only requires the conductivity to have bounded second-order partial 
derivatives \cite{zbMATH04050176}. After this piece of work, and some others as 
\cite{zbMATH03998383} by Alssandrini and \cite{zbMATH04105476} by Nachman, a standard 
strategy to address the Calderón problem has been established: One starts by showing
\[ \Lambda_{\gamma_1} = \Lambda_{\gamma_2} \Rightarrow \int_{\R^n} (q_1 - q_2) v_1 v_2 \, = 0,\]
where $q_j = \gamma_j^{-1/2} \Delta \gamma_j^{1/2}$ and $v_j$ is any 
solution of $- \Delta v_j + q_j v_j = 0$ in $\R^n$. Then, one constructs CGO solutions of the 
form:
\[v_j = e^{\zeta_j \cdot {\rm x}} (1 + w_j) \]
with $\zeta_j \in \C^n$ and $\zeta_j \cdot \zeta_j = 0$ for $ j \in \{1, 2 \}$, and so that
$w_j$ tends to vanish in some sense as $|\zeta_j|$ goes to $\infty$. Finally, one makes
appropriate choices of $\zeta_1$ and $\zeta_2$ so that one concludes the uniqueness.
The work of Sylvester and Uhlmann has been followed many attempts to lower the regularity
assumed, see for example \cite{zbMATH00912089, zbMATH02102106, zbMATH06145493, zbMATH06490961, 
zbMATH06534426, zbMATH07373390, zbMATH07395052}.

In \cite{zbMATH06145493}, Haberman and Tataru 
introduced the Bourgain spaces $\dot{Y}_\zeta^s$ with 
$s \in \{ 1/2, -1/2 \}$. These 
spaces have the norms
\[ \| f \|_{\dot{Y}_\zeta^s} = \| |q_\zeta|^s \widehat{f} \|_{L^2(\R^n)}, \]
where $q_\zeta$ stands for the symbol of the differential operator 
$\Delta + 2\zeta \cdot \nabla$, that is,
\[ q_\zeta (\xi) = - |\xi|^2 + 2i \zeta \cdot \xi \quad \forall \xi \in \R^n. \]
These spaces where introduced with the idea of making possible an average with respect to the
free parameters $|\Re \zeta|$ and $\Re \zeta/|\Re \zeta|$,
so that the correction terms of the CGO solutions
tend to vanish when the conductivity is assumed to be continuously differentiable. However,
the space $\dot{Y}_\zeta^{1/2}$ has other properties such as the embeddings
\begin{equation}
\label{in:embeddings_HT-H}
|\zeta|^{1/2} \| f \|_{B^\ast(\R^n)} + \| f \|_{L^{p_n}(\R^n)} 
\lesssim \| f \|_{\dot{Y}^{1/2}_\zeta}
\end{equation}
for every $f \in \mathcal{S}(\R^n)$---see \cite{zbMATH06145493, zbMATH06490961}.
Here $p_n$ is defined by
the relation $1/p_n = 1/2 - 1/n$ for $n \geq 3$, and
\[ \| f \|_{B^\ast(\R^n)} = \sup_{j\in\N_0}\big(2^{-j/2}\| f \|_{L^2(D_j)}\big) \]
with $D_j=\{x\in \R^n : 2^{j-1}<|x|\leq 2^j\}$ for $j \in \N$ and 
$D_0 = \{ x \in \R^n : |x| \leq 1 \}$.

The fact that the Bourgain space $\dot{Y}_\zeta^{1/2}$ enjoys the embeddings in
\eqref{in:embeddings_HT-H} makes it possible to prove uniqueness for the Calderón problem
for conductivities satisfying $\nabla \gamma \in L^n(D)^n$ and 
$\Delta \gamma \in L^{n/2} (D)$, without using the Lavine--Nachman trick.
For this reason, one could think that the use of Bourgain spaces 
in the context of the initial-to-final 
state inverse problem would avoid the use of the Lavine--Nachman trick. 
Unfortunately, this is not the case. In fact, we will see in this 
section the limitations of Bourgain spaces to deal with our problem.

The Calder\'on problem has been extensively studied: 
uniqueness, reconstruction and stability. A non-comprehensive list of
references is the following. Uniqueness in dimension 
$n = 2$: \cite{zbMATH00854849, zbMATH01044195, zbMATH05050053}.
For reconstruction: 
\cite{zbMATH04105476, zbMATH06659335, zbMATH08080725}.
For stability: \cite{zbMATH03998383, zbMATH04028037, zbMATH00004861,
zbMATH01649244, zbMATH05223898, zbMATH05704414, zbMATH06117512}.

\subsection{The Bourgain spaces for the initial-to-final state inverse problem}
In the context of the initial-to-final state inverse problem one can introduce the Bourgain 
spaces $\dot{X}_\nu^s$ with 
$s \in \{ 1/2, -1/2 \}$. These 
spaces have the norms
\[ \| u \|_{\dot{X}_\nu^s} = \| |p_\nu|^s \widehat{u} \|_{L^2(\R \times \R^n)}, \]
for $u \in \mathcal{S}(\R \times \R^n)$, with $p_\nu$  as in \eqref{id:symbol}.
One can prove that
there exists an absolute constant $C > 0$ such that
\begin{equation}
\label{in:embedding_half-in}
|\nu|^{1/4} \| u \|_{L^2 ((0,T) \times B_R)} \leq C T^{1/4} R^{1/4} \| u \|_{\dot{X}^{1/2}_{\nu}}
\end{equation}
for all $u \in \mathcal{S}(\R \times \R^n)$. Here $B_R = \{ x \in \R^n : |x| < R \}$. In fact,
this inequality can be though as half of the inequality for $S_\nu$ that we used in our 
previous work \cite[inequality (27)]{10.1063/5.0152372}---for that it is enough to chain 
\eqref{in:embedding_half-in} with its dual. 
However, despite what happens in the context of the Calderón problem,
the Bourgain space for $s=1/2$ is not embedded in the mixed-norm Lebesgue spaces appearing
in the \cref{th:non-gain_Snu}. In fact, we will show that the inequality
\begin{equation}
\label{in:parameter_embedding}
\| u \|_{L^{q^\prime}(\R; L^{r^\prime}(\R^n))} \lesssim \| u \|_{\dot{X}^{1/2}_\nu} \quad \forall u \in \mathcal{S}(\R \times \R^n)
\end{equation}
does not hold for $(q^\prime, r^\prime) \in [2, \infty] \times [2, \infty]$ satisfying 
\eqref{id:Strichartz_dual-indeces}.
Using \eqref{id:rescaling}, one can see that \eqref{in:parameter_embedding} is equivalent
to
\begin{equation}
\label{in:free_parameter_embedding}
\| u \|_{L^{q^\prime}(\R; L^{r^\prime}(\R^n))} \lesssim \| |p|^{1/2} \widehat{u} \|_{L^2(\R \times \R^n)} \quad \forall u \in \mathcal{S}(\R \times \R^n),
\end{equation}
with $p$ as in \eqref{id:normalized_symbol}.
Our goal here is to give a counterexample for the inequality 
\eqref{in:free_parameter_embedding}. Before that, we make some general comments.

In \cite{zbMATH05035890}, Tao states that the inequality
\begin{equation}
\label{in:Tao_embedding}
\| u \|_{L^{q^\prime}(\R; L^{r^\prime}(\R^n))} \lesssim \| \langle \Re p \rangle^b \widehat{u} \|_{L^2(\R \times \R^n)} \qquad \forall u \in \mathcal{S}(\R \times \R^n)
\end{equation}
holds for $b > 1/2$, where 
$\langle \Re p \rangle = (1 + |\Re p|^2)^{1/2}$ with $ \Re p(\tau, \xi) = \tau - |\xi|^2 $ for 
all $(\tau, \xi) \in \R \times \R^n$. This embedding of the non-homogenous Bourgain space
follows from \cite[Corollary 2.10]{zbMATH05035890}, which states that
\begin{equation}
\label{in:corollary2.10}
\sup_{t \in \R} \| u (t, \centerdot) \|_{L^2(\R^n))} \lesssim \| \langle \Re p \rangle^b \widehat{u} \|_{L^2(\R \times \R^n)} \qquad \forall u \in \mathcal{S}(\R \times \R^n)
\end{equation}
whenever $b > 1/2$.
In \cite[Exercise 2.71]{zbMATH05035890}, the reader is asked to
prove that the inequality \eqref{in:corollary2.10} fails in the endpoint $b=1/2$.
Here we go beyond that and show a counterexample for the embedding \eqref{in:Tao_embedding} 
for $b=1/2$. This counterexample will be adapted to eventually prove that
\eqref{in:free_parameter_embedding} does not hold.

\subsubsection{Failure of \eqref{in:Tao_embedding} at the endpoint}
\label{sec:failure_end-point}
In order to point out why \eqref{in:Tao_embedding} fails at the endpoint $b = 1/2$, 
we test the norms involved in the inequality with functions 
$u \in \mathcal{S}(\R \times \R^n)$ of the form
\begin{equation}
\label{id:u_product}
\widehat{u}(\tau, \xi) = \widehat{g}(\tau - |\xi|^2) \widehat{f}(\xi) \qquad \forall (\tau, \xi) \in \R \times \R^n
\end{equation}
with $g \in \mathcal{S}(\R)$ and $f \in \mathcal{S}(\R^n)$. One can check that
\begin{equation}
\label{id:norm_of_product}
\| \langle \Re p \rangle^{1/2} \widehat{u} \|_{L^2(\R \times \R^n)} = \| g \|_{H^{1/2}(\R)} \| f \|_{L^2(\R^n)},
\end{equation}
and
\[ u(t, x) = \frac{g(t)}{(2\pi)^{n/2}} \int_{\R^n} e^{i x \cdot \xi} e^{i t |\xi|^2} \widehat{f}(\xi) \, \dd \xi. \]
Noting that $2/{q^\prime} = n/2 - n/{r^\prime}$ defines admissible Strichartz pairs for the 
Schr\"odinger equation we have that
\[ \| u \|_{L^{q^\prime}(\R; L^{r^\prime}(\R^n))} \lesssim \| g \|_{L^\infty(\R)} \| f \|_{L^2(\R^n)}. \]
Then, by \eqref{id:norm_of_product} we have that
\[ \| u \|_{L^{q^\prime}(\R; L^{r^\prime}(\R^n))} \lesssim \frac{\| g \|_{L^\infty(\R)}}{\| g \|_{H^{1/2}(\R)}} \| \langle \Re p \rangle^{1/2} \widehat{u} \|_{L^2(\R \times \R^n)}. \]
Thus, in order to derive \eqref{in:Tao_embedding} from the previous inequality for 
the family of functions defined as \eqref{id:u_product} we should ensure that the 
quotient $ \| g \|_{L^\infty(\R)} / \| g \|_{H^{1/2}(\R)} $ remains bounded for all
$g \in \mathcal{S}(\R)$. Note the relation of this quotient with the endpoint 
Sobolev embedding, which is known to fail. Thus, one should expect that choosing 
an unbounded $g \in H^{1/2}(\R) $, we could construct a counterexample for 
\eqref{in:Tao_embedding} with $b=1/2$. This will be our strategy.

Consider $\chi \in \mathcal{S}(\R^2)$ such that 
$\supp \chi \subset \{ z \in \R^2 : |z| < 1/e \}$
with $\chi(z) = 1$ whenever $|z|\leq 1/(2e)$. The function
\[ z \in \R^2 \setminus \{ 0\} \longmapsto \chi(z) \log (\log 1/|z|) \in [0, \infty) \]
can be extended to represent a $v \in H^1(\R^2)$. 
If $g$ denotes the trace of $v$ to 
$\R \times \{0\}$, we can ensure that $g \in H^{1/2}(\R)$ and 
$g(t) \geq \log (\log 1/\delta) $ for almost every $t \in (-\delta, \delta) $ with 
$ \delta \in (0, 1/(2e)]$.

Let $f_\rho \in L^2(\R^n)$ with $\rho > 0$ be so that
\[ \widehat{f_\rho}(\xi) = \rho^{-n/2} \mathbf{1}_{<\rho} (\xi) \qquad \forall \xi \in \R^n. \]
It is clear that 
$\| f_\rho \|_{L^2(\R^n)} = \| \mathbf{1}_{<1} \|_{L^2(\R^n)}$
which is independent of $\rho$.

We consider the family $ \{ u_\rho : \rho^2 \geq e \}$ with
\[ u_\rho(t, x) = \frac{g(t)}{(2\pi)^{n/2}} \int_{\R^n} e^{i x \cdot \xi} e^{i t |\xi|^2} \widehat{f_\rho}(\xi) \, \dd \xi \qquad \dot \forall (t, x) \in \R \times \R^n.\footnote{The
symbol $\dot{\forall}$ means 'for almost every'.} \]
As in \eqref{id:norm_of_product}, one can check that
\begin{equation}
\label{id:constant_norm}
\| \langle \Re p \rangle^{1/2} \widehat{u}_\rho \|_{L^2(\R \times \R^n)} = \| g \|_{H^{1/2}(\R)} \| \mathbf{1}_{<1} \|_{L^2(\R^n)}.
\end{equation}
Furthermore,
\[ |u_\rho (t, x)| \geq \rho^{-n/2} \frac{g(t)}{(2\pi)^{n/2}} \int_{\R^n} \cos ( x \cdot \xi + t |\xi|^2) \mathbf{1}_{<\rho} (\xi) \, \dd \xi. \]
Note that if $| x \cdot \xi + t |\xi|^2| \leq 1 < \pi/3$, then $\cos ( x \cdot \xi + t |\xi|^2) > 1/2$.
Thus, for $|x|\leq 1/(2 \rho)$ and $|t|\leq 1/(2 \rho^2)$ we have that
\[\cos ( x \cdot \xi + t |\xi|^2) \mathbf{1}_{<\rho} (\xi) > 1/2, \]
and consequently,
\[ |u_\rho (t, x)| \gtrsim \rho^{n/2} \log(\log 2 \rho^2) \]
whenever $|x|\leq 1/(2 \rho)$ and $|t|\leq 1/(2 \rho^2)$. This means that
\[ \| u_\rho \|_{L^{q^\prime}(\R; L^{r^\prime}(\R^n))} \gtrsim \rho^{n/2} \rho^{-2/q^\prime} \rho^{-n/r^\prime} \log(\log 2 \rho^2). \]
Since $2/q^\prime = n/2 - n/r^\prime$ we see that
\[ \| u_\rho \|_{L^{q^\prime}(\R; L^{r^\prime}(\R^n))} \gtrsim \log(\log 2 \rho^2). \]
This inequality together with the identity \eqref{id:constant_norm} shows that \eqref{in:Tao_embedding} with $b=1/2$ fails.

\subsubsection{Counterexample to \eqref{in:free_parameter_embedding}}
From the failure of \eqref{in:Tao_embedding} for $b=1/2$, we can not immediately 
deduce that \eqref{in:free_parameter_embedding} fails
since the counterexample we presented use a family of functions 
whose Fourier support is not contained in $\{ \xi \in \R^n : |\xi_n| \leq 1 \}$.
However, the main points of the counterexample are robust enough to be adapted 
to the current situation.

As we did earlier in the \cref{sec:failure_end-point}, 
consider $\chi \in \mathcal{S}(\R^2)$ such that 
$\supp \chi \subset \{ z \in \R^2 : |z| < 1/e \}$
with $\chi(z) = 1$ whenever $|z|\leq 1/(2e)$. The function
\[ z \in \R^2 \setminus \{ 0\} \longmapsto \chi(z) \log (\log 1/|z|) \in [0, \infty) \]
can be extended to represent a $v \in H^1(\R^2)$. 
If $g$ denotes the trace of $v$ to 
$\R \times \{0\}$, we can ensure that $g \in H^{1/2}(\R)$ and 
$g(t) \geq \log (\log 1/\delta) $ for almost every $t \in (-\delta, \delta) $ with 
$ \delta \in (0, 1/(2e)]$.
Define
\[ g_\rho (t) = g(\rho t) \qquad \dot \forall t \in \R. \]

Let $f_\rho \in L^2(\R^n)$ with $\rho > 0$ be so that
\[ \widehat{f_\rho}(\xi) = \rho^{-n/2} \mathbf{1}_{<\rho} (\xi - 2 \rho e_n) \qquad \forall \xi \in \R^n, \]
It is clear again that 
$\| f_\rho \|_{L^2(\R^n)} = \| \mathbf{1}_{<1} \|_{L^2(\R^n)}$
which is independent of $\rho$.

We consider the family $ \{ u_\rho : \rho \geq e/9 \}$ satisfying
\[ \widehat{u_\rho} (\tau, \xi) = \widehat{g_\rho} (\tau - |\xi|^2) \widehat{f_\rho} (\xi) \qquad \forall (\tau, \xi) \in \R \times \R^n. \]
One can check that
\[ u_\rho(t, x) = \frac{g_\rho(t)}{(2\pi)^{n/2}} \int_{\R^n} e^{i x \cdot \xi} e^{i t |\xi|^2} \widehat{f_\rho}(\xi) \, \dd \xi \qquad \dot \forall (t, x) \in \R \times \R^n. \]
Additionally, since 
$\supp \widehat{f_\rho} \subset \{ \xi \in \R^n : \rho < |\xi_n| < 3\rho \}$
we have
\[ \| u_\rho \|_{\dot{X}^{1/2}} \eqsim \| |\Re p + i \rho|^{1/2} \widehat{u_\rho} \|_{L^2(\R \times \R^n)} = \| |\centerdot + i \rho|^{1/2} \widehat{g_\rho} \|_{L^2(\R)} \| \widehat{f_\rho} \|_{L^2(\R^n)}. \]
Since $g_\rho$ is defined as an $L^\infty(\R)$ scaling of $g$, we can see that
\[ \| |\centerdot + i \rho|^{1/2} \widehat{g_\rho} \|_{L^2(\R)} = \| g \|_{H^{1/2}(\R)}, \]
and consequently, we have that
\begin{equation}
\label{eq:constant_norm}
\| u_\rho \|_{\dot{X}^{1/2}} \eqsim \| g \|_{H^{1/2}(\R)} \| \mathbf{1}_{<1} \|_{L^2(\R^n)}.
\end{equation}

Once again,
\[ |u_\rho (t, x)| \geq \rho^{-n/2} \frac{g_\rho(t)}{(2\pi)^{n/2}} \int_{\R^n} \cos ( x \cdot \xi + t |\xi|^2) \mathbf{1}_{<\rho} (\xi - 2 \rho e_n) \, \dd \xi. \]
Note that if $| x \cdot \xi + t |\xi|^2| \leq 1 < \pi/3$, then $\cos ( x \cdot \xi + t |\xi|^2) > 1/2$.
Thus, since $\supp \widehat{f_\rho} \subset \{ \xi \in \R^n: |\xi| < 3\rho  \}$
we have for $|x|\leq 1/(6 \rho)$ and $|t|\leq 1/(18 \rho^2)$ that
\[\cos ( x \cdot \xi + t |\xi|^2) \mathbf{1}_{<\rho} (\xi - 2 \rho e_n) > 1/2, \]
and consequently,
\[ |u_\rho (t, x)| \gtrsim \rho^{n/2} g(\rho t) \geq \rho^{n/2} \log(\log 18 \rho) \]
whenever $|x|\leq 1/(6 \rho)$ and $|t|\leq 1/(18 \rho^2)$.
This means that
\[ \| u_\rho \|_{L^{q^\prime}(\R; L^{r^\prime}(\R^n))} \gtrsim \rho^{n/2} \rho^{-2/q^\prime} \rho^{-n/r^\prime} \log(\log 18 \rho). \]
Since $2/q^\prime = n/2 - n/r^\prime$ we see that
\[ \| u_\rho \|_{L^{q^\prime}(\R; L^{r^\prime}(\R^n))} \gtrsim \log(\log 18 \rho). \]
This inequality together with the estimate \eqref{eq:constant_norm} shows that 
\eqref{in:free_parameter_embedding} fails.

\subsubsection{Further thoughts about the \eqref{in:free_parameter_embedding}}
A more careful look at our counterexample for \eqref{in:free_parameter_embedding} indicates that
this embedding also fails when $2/q^\prime < n/2 - n/r^\prime$. However, it does not say
anything when $2/q^\prime > n/2 - n/r^\prime$.

We have not analysed the latter case, despite that it could be interesting. Indeed, if
\eqref{in:free_parameter_embedding} held true for $2/q^\prime > n/2 - n/r^\prime$, 
by the scaling rule \eqref{id:rescaling}, the inequality
\eqref{in:parameter_embedding} would hold with a gain of $|\nu|^{2/q^\prime - n/2 + n/r^\prime}$:
\begin{equation}
\label{emb:out-scaleinvariance}
|\nu|^{2/q^\prime - n/2 + n/r^\prime} \| u \|_{L^{q^\prime}(\R; L^{r^\prime}(\R^n))} \lesssim \| u \|_{\dot{X}^{1/2}_\nu} \quad \forall u \in \mathcal{S}(\R \times \R^n).
\end{equation}
This case would allow us to deal
with potentials in $ L^a((0, T); L^b (\R^n))$ for
$n \geq 2$ and $2-2/a>n/b$, without appealing to the Lavine--Nachman trick. 
Furthermore, this embedding would immediately imply
an extension of the \cref{th:non-gain_Snu} of the following form:
\begin{equation}
\label{in:Snu_out-scaleinvariace}
|\nu|^{2/q_1^\prime - n/2 + n/r_1^\prime} \| S_\nu f \|_{L^{q_1^\prime} (\R; L^{r_1\prime}(\R^n))} \lesssim |\nu|^{n/r_2 - n/2 -2 +2/q_2 } \| f \|_{L^{q_2} (\R; L^{r_2}(\R^n))}
\end{equation}
for all $f \in \mathcal{S}(\R \times \R^n)$
with $2/q_1^\prime > n/2 - n/r_1^\prime$ and $2-2/q_2 > n/r_2 - n/2$.

It is worth mentioning that the inequality \eqref{in:Snu_out-scaleinvariace}
with $2/q_1^\prime \geq n/2 - n/r_1^\prime$ and $2-2/q_2 \geq n/r_2 - n/2$ could hold even if 
the embedding \eqref{emb:out-scaleinvariance} fails.
Fortunately, this situation would still allow to deal
with potentials in $ L^a((0, T); L^b (\R^n))$ for
$n \geq 2$ and $2-2/a>n/b$, without appealing to the Lavine--Nachman trick.

Comparing this situation with the elliptic setting of the Calderón problem,
\eqref{in:Snu_out-scaleinvariace} corresponds to the full Carlmen inequality
that Kenig, Ruiz and Sogge proved in \cite{zbMATH04050093}---not
only in the scale-invariant situation. In fact, this
inequality is enough to prove uniqueness of the Calderón problem
for conductivities satisfying $\nabla \gamma \in L^{2p}(D)^n$ and 
$\Delta \gamma \in L^p (D)$ for $p>n/2$, without using the Lavine--Nachman trick.

\appendix

\section{Sketch of the proof of the \cref{lem:integration_by_parts}}\label{app:integration_by_parts}

We follow here the same argument we used to prove \cite[Proposition 4.2]{10.1063/5.0152372}.

\begin{proof}[Sketch of the proof of the \cref{lem:integration_by_parts}]
Start by noticing that under the assumptions of the 
\cref{lem:integration_by_parts}, we can apply the dominate convergence theorem 
to ensure that
\begin{equation}
\label{lim:delta}
\int_\Sigma \big[ (i\partial_\tm + \Delta) u \overline{v} - u \overline{(i\partial_\tm + \Delta)v} \big] \, = \lim_{\delta \to 0} \int_{\Sigma_\delta} \big[ (i\partial_\tm + \Delta) u \overline{v} - u \overline{(i\partial_\tm + \Delta)v} \big]
\end{equation}
where $\Sigma_\delta = (\delta, T - \delta) \times \R^n$.
Now, we will approximate $u$ and $v$ by $u_\varepsilon$ and 
$v_\varepsilon$ respectively, which will be smooth in $\Sigma_\delta$ and 
compactly supported in space.

Consider a smooth cut-off $\chi \in \mathcal{S}(\R^n)$ 
such that $ 0 \leq \chi(x) \leq 1 $ for 
all $x \in \R^n$, $\supp \chi \subset \{ x \in \R^n: |x| \leq 2 \}$ and
$ \chi(x) = 1 $ whenever $|x| \leq 1$. Consider $\varphi \in \mathcal{S}(\R \times \R^n)$
such that $\supp \varphi \subset [-1, 1] \times K$ with $K$ a compact subset of $\R^n$, 
$\int_{\R \times \R^n} \varphi = 1$ and $\varphi(t, x) \in [0, \infty)$
for all $(t, x) \in \R \times \R^n$. For $R, \varepsilon \in (0, \infty) $, define 
$\chi^R (x) = \chi (x/R)$ for $x \in \R^n$,
$\varphi_\varepsilon (t, x) = \varepsilon^{-(n+1)} \varphi (t/\varepsilon, x/\varepsilon) $ 
for $(t, x) \in \R \times \R^n$.

Let $(X_\delta, \| \centerdot \|_{X_{\delta}})$ and
$(X_\delta^\star, \| \centerdot \|_{X_{\delta}^\star})$ denote the Banach spaces defined as 
$(X, \| \centerdot \|_X)$ and $(X^\star, \| \centerdot \|_{X^\star})$ for the time interval 
$(\delta, T-\delta)$ instead of $(0, T)$.

For $w \in X^\star$ and $\varepsilon < \delta$, 
we consider
\[w_\varepsilon: (t, x) \in \Sigma_\delta \mapsto \chi^{R(\varepsilon)} (x)  \int_\Sigma \varphi_{\varepsilon} (t- s, x - y) w (s, y) \, \dd (s, y),\]
where $\varepsilon \in (0, \delta) \mapsto R(\varepsilon) \in (0, \infty)$ is continuous and 
$\lim_{\varepsilon \to 0} R(\varepsilon) = \infty$.

One can check that
\begin{equation}
\label{lim:t-L2_approx}
\lim_{\varepsilon \to 0} \| w(t, \centerdot) - w_\varepsilon (t, \centerdot)  \|_{L^2(\R^n)} = 0  \qquad \forall t \in [\delta, T - \delta],
\end{equation}
and 
\begin{equation}
\label{lim:Lq'Lr'_approx}
\lim_{\varepsilon \to 0} \| w - w_\varepsilon \|_{L^{q^\prime}((\delta, T - \delta); L^{r^\prime}(\R^n))} = 0.
\end{equation}

Furthermore, if $(i\partial_\tm + \Delta) w \in X_\delta$, then
\begin{equation}
\label{lim:X_approx}
\lim_{\varepsilon \to 0} \| (i\partial_\tm + \Delta)w - (i\partial_\tm + \Delta)w_\varepsilon  \|_{X_{\delta}} = 0.
\end{equation}

For the sake of clarity, let us discuss why \eqref{lim:X_approx} holds.
Let $\tilde{w}$ denote the 
trivial extension of $w$
\[ \tilde w (t, \centerdot) = \left\{  
\begin{aligned} 
&w(t, \centerdot) &  & t \in [0, T], \\
&0 & & t \notin [0, T].
\end{aligned}
\right.\]
Then, we have that
$ w_\varepsilon (t, x) = \chi_\varepsilon (x) (\varphi_\varepsilon \ast \tilde{w}) (t, x) $
for all $(t,x) \in \Sigma_\delta$.

In order to show that \eqref{lim:X_approx} holds, 
let us compute
\[(i\partial_\tm + \Delta)w_\varepsilon = \chi^{R(\varepsilon)} (i\partial_\tm + \Delta) (\varphi_\varepsilon \ast \tilde{w}) + 2 \nabla \chi^{R(\varepsilon)} \cdot \nabla (\varphi_\varepsilon \ast \tilde{w}) + \Delta \chi^{R(\varepsilon)} (\varphi_\varepsilon \ast \tilde{w}) \enspace \textnormal{in} \enspace \Sigma_\delta. \]

Observe, that \eqref{lim:X_approx} would follow from the following limits
\[ \lim_{\varepsilon \to 0} \| (i\partial_\tm + \Delta)w - \chi^{R(\varepsilon)} (i\partial_\tm + \Delta) (\varphi_\varepsilon \ast \tilde{w})  \|_{X_{\delta}} = 0, \]
\begin{equation}
\label{lim:grad_chi}
\lim_{\varepsilon \to 0} \| \nabla \chi^{R(\varepsilon)} \cdot \nabla (\varphi_\varepsilon \ast \tilde{w}) \|_{X_{\delta}} = 0,
\end{equation}
and
\[ \lim_{\varepsilon \to 0} \| \Delta \chi^{R(\varepsilon)} (\varphi_\varepsilon \ast \tilde{w}) \|_{X_{\delta}} = 0. \]

The previous three limits follow from standard arguments, however, for the limit
\eqref{lim:grad_chi} we have to choose the function 
$\varepsilon \in (0, \delta) \mapsto R(\varepsilon) \in (0, \infty)$ to satisfies the 
condition $\lim_{\varepsilon \to 0} \varepsilon R(\varepsilon) = \infty$. With this extra
condition, we have that \eqref{lim:X_approx} holds.

We now apply these convergences to
$u_\varepsilon$ and $v_\varepsilon$.
By the standard integration-by-parts rules we have that
\[\int_{\Sigma_\delta} \big[ (i\partial_\tm + \Delta) u_\varepsilon \overline{v_\varepsilon} - u_\varepsilon \overline{(i\partial_\tm + \Delta)v_\varepsilon} \big] \, = i \int_{\R^n} \big[ 
u_\varepsilon(T - \delta, \centerdot) \overline{v_\varepsilon (T - \delta,\centerdot)} - u_\varepsilon(\delta, \centerdot) \overline{v_\varepsilon (\delta,\centerdot)} \big] \,. \]

After \eqref{lim:t-L2_approx}, \eqref{lim:Lq'Lr'_approx} and \eqref{lim:X_approx}
we can conclude, by the dominate convergence theorem, that
\[\int_{\Sigma_\delta} \big[ (i\partial_\tm + \Delta) u \overline{v} - u \overline{(i\partial_\tm + \Delta)v} \big] \, = i \int_{\R^n} \big[ 
u(T - \delta, \centerdot) \overline{v (T - \delta,\centerdot)} - u(\delta, \centerdot) \overline{v (\delta,\centerdot)} \big] \,. \]

The limits \eqref{lim:delta} and
\[ \int_{\R^n} \big[ 
u(T, \centerdot) \overline{v (T,\centerdot)} - u(0, \centerdot) \overline{v (0,\centerdot)} \big] = \lim_{\delta \to 0} \int_{\R^n} \big[ 
u(T - \delta, \centerdot) \overline{v (T - \delta,\centerdot)} - u(\delta, \centerdot) \overline{v (\delta,\centerdot)} \big],\]
which follows by continuity, yield the identity in the statement.
\end{proof}

\section{Transference of the exponential decay}
\label{app:transference}

Here we include the proof of the \cref{lem:exp_decay}, which is an adaptation of the proof
of \cite[Lemma 4.6]{10.1063/5.0152372}. 
The \cref{lem:exp_decay} requires some well known properties about holomorphic extensions of the Fourier 
transform. These are all derived from the statement and the proof of
\cite[Theorem 7.4.2]{zbMATH01950198}. In \cite[Lemma 4.7]{10.1063/5.0152372}, we stated and proved  
these facts for the particular cases we needed in that context. Unfortunately, the statement in there is 
different from the one needed here, but its proof is essentially the same. For the sake of 
completeness, we include its proof here.

\begin{lemma}\label{lem:Fourier_hol-extension}\sl Consider $F \in L^1(\R; L^p(\R^n))$
with $p\geq 1$ such that there exists $c > 0$ so that $e^{c|\x|} F \in L^1(\R; L^p(\R^n))$. Then:
\begin{enumerate}[label=\textnormal{(\alph*)}, ref=\textnormal{\alph*}]
\item \label{ite:L1} For every $\nu \in \R^n$ such that 
$|\nu| < c$, we have $e^{\nu \cdot \x} F \in L^1(\R \times \R^n)$.
\item \label{ite:representation} The Fourier transform of $F$ can be extended to 
$\R \times \R^{n - 1} \times \C_{|\Im|<c}$  as the continuous function
\[ \widehat{F} : (\tau, \xi^\prime, \zeta) \mapsto \frac{1}{(2 \pi)^\frac{n+1}{2}} \int_{\R \times \R^n } e^{-i(\tau t + (\xi^\prime, \zeta) \cdot x)} F(t, x) \, \dd (t, x), \]
where $\C_{|\Im|<c} = \{ \zeta \in \C : |\Im \zeta| < c \}$. Furthermore,
it satisfies the relation 
\[ \widehat{e^{\lambda \x_n} F}(\tau, \xi^\prime, \xi_n) = \widehat{F}(\tau, \xi^\prime, \xi_n + i \lambda) \]
for all $(\tau, \xi^\prime, \xi_n) \in \R \times \R^{n-1} \times \R$ and $\lambda \in \R$ such that 
$|\lambda|<c$.
\item \label{ite:holomrphic_z} For every 
$(\tau, \xi^\prime) \in \R \times \R^{n-1}$, the function 
\[ \zeta \in \C_{|\Im|<c} \mapsto \widehat{F}(\tau, \xi^\prime, \zeta) \]
is holomorphic in $\C_{|\Im|<c} $.
\end{enumerate}
\end{lemma}

\begin{proof}
Start by proving \eqref{ite:L1}. Note that
\[ \|e^{\nu \cdot \x} F \|_{L^1 (\R \times \R^n)} \leq \| e^{- ( c - |\nu|) |x|} \|_ {L^{p^\prime}(\R^n)} \| e^{c |\x|} F \|_{L^1(\R; L^p(\R^n))}, \]
where $e^{- ( c - |\nu|) |x|} \in L^{p^\prime}(\R^n) $ since $|\nu| < c$, with $p^\prime$ denoting the
conjugate of $p$.

The property \eqref{ite:representation} 
is a direct consequence of \eqref{ite:L1} since
for $(\tau, \xi^\prime, \zeta) \in \R \times \R^{n - 1} \times \C_{|\Im|<c}$,
we have that
\begin{align*}
\widehat{e^{\Im \zeta e_n \cdot \x} F} (\tau, \xi^\prime, \Re \zeta) & = \frac{1}{(2 \pi)^\frac{n+1}{2}} \int_{\R \times \R^n } e^{-i(\tau t + (\xi^\prime, \Re \zeta) \cdot x)} e^{\Im \zeta e_n \cdot x} F(t, x) \, \dd (t, x) \\
& = \frac{1}{(2 \pi)^\frac{n+1}{2}} \int_{\R \times \R^n } e^{-i(\tau t + (\xi^\prime, \zeta) \cdot x)} F(t, x) \, \dd (t, x).
\end{align*}
Hence, we can use the last expression to extend $\widehat{F}$ to
$\R \times \R^{n - 1} \times \C_{|\Im|<c}$. Furthermore, choosing 
$\zeta = \xi_n + i \lambda$ we have that
\[ \widehat{e^{\lambda \x_n} F} (\tau, \xi^\prime, \xi_n) = \widehat{F} (\tau, \xi^\prime, \xi_n + i \lambda).  \]
This proves the property \eqref{ite:representation}.

Finally, to verify the property \eqref{ite:holomrphic_z} one 
needs to check that the function satisfies the Cauchy--Riemann equations. This 
can be justified since 
$\x_n e^{\Im \zeta e_n \cdot \x_n} F \in L^1(\R \times \R^n)$. 
This completes the proof of \eqref{ite:holomrphic_z}.
\end{proof}

In order to prove the \cref{lem:exp_decay}, we still need a preparatory result, 
which is a consequence of the \cref{lem:Fourier_hol-extension}. Again, this is pretty similar
to \cite[Corollary 4.8]{10.1063/5.0152372}.
The statement there does not include the case we need here, but 
the proof is essentially the same. For convenience for the reader, we include its proof here.

\begin{corollary}\label{cor:vanish}\sl 
Consider $F \in L^q((0,T); L^r(\R^n))$, with $(q, r) \in [1, 2] \times [1, 2]$ satisfying
\eqref{id:Strichartz_indeces}.
Let $G \in L^p(\R; L^r(\R^n))$ with $p \in [1, q]$ denote the function
\begin{equation*}
G (t, \centerdot) =
	\left\{
		\begin{aligned}
		& F(t, \centerdot) & & \textnormal{if} \, t \in (0,T), \\
		& 0 &  & \textnormal{if} \, t \in \R \setminus (0,T).
		\end{aligned}
	\right.
\end{equation*}
If there exists $u \in X^\star$ satisfying
\[
\left\{
		\begin{aligned}
		& (i\partial_\tm + \Delta) u = F & & \textnormal{in} \enspace \Sigma, \\
		& u(0, \centerdot) = u(T, \centerdot) = 0 &  & \textnormal{in} \enspace \R^n;
		\end{aligned}
	\right.
\]
with $e^{c|\x|} F \in L^q((0,T); L^r(\R^n))$ for some $c>0$,
then, for every $(\tau, \xi^\prime) \in \R \times \R^{n-1}$ such that 
$\tau \neq -|\xi^\prime|^2$, the function 
\begin{equation}
\label{map:Gzparoboloid}
\zeta \in \C_{|\Im|<c} \mapsto \frac{\widehat{G}(\tau, \xi^\prime, \zeta)}{-\tau - |\xi^\prime|^2 - \zeta^2}
\end{equation}
is holomorphic in $\C_{|\Im|<c} = \{ \zeta \in \C : |\Im \zeta| < c \}$.
\end{corollary}

\begin{proof}
Since $G \in L^1(\R; L^r(\R^n))$, we know from \eqref{ite:holomrphic_z} in the 
\cref{lem:Fourier_hol-extension} that, for every 
$(\tau, \xi^\prime) \in \R \times \R^{n-1}$, the function 
\begin{equation}
\label{map:Gz}
\zeta \in \C_{|\Im|<c} \mapsto \widehat{G}(\tau, \xi^\prime, \zeta)
\end{equation}
is holomorphic in $\C_{|\Im|<c} $. While, whenever $\tau \neq - |\xi^\prime|^2$, 
the function 
\[ \zeta \in \C_{|\Im|<c} \mapsto \frac{1}{-\tau - |\xi^\prime|^2 - \zeta^2} \]
is meromorphic, with simple poles at $\zeta_1 = -|\tau + |\xi^\prime|^2 |^{1/2} $ 
and $ \zeta_2 = |\tau + |\xi^\prime|^2 |^{1/2} $ if $- \tau - |\xi^\prime|^2 > 0$, and $\zeta_3 = -i|\tau + |\xi^\prime|^2 |^{1/2} $ and 
$ \zeta_4 = i|\tau + |\xi^\prime|^2 |^{1/2} $ if $- \tau - |\xi^\prime|^2 < 0$
and $|\tau + |\xi^\prime|^2 | < c^2$. 
Then, in order to prove that the function 
\eqref{map:Gzparoboloid} is holomorphic is enough to check that \eqref{map:Gz} vanishes at these poles. Let us focus on proving that.

The item \eqref{ite:L1} in the \cref{lem:Fourier_hol-extension} and the fact that
$\supp G \subset \overline{\Sigma}$ ensure that, for every 
$(\sigma, \nu) \in \R \times \R^n$ such that  $|\nu| < c$, 
we have $e^{\sigma \tm + \nu \cdot \x} G \in L^1(\R \times \R^n)$.
Thus, the Fourier transform of $G$ can be extended to 
$\C \times \R^{n - 1} \times \C_{|\Im|<c}$  as the continuous function
\begin{equation}
\label{map:Ghatext}
\widehat{G} : (\gamma, \xi^\prime, \zeta) \mapsto \frac{1}{(2 \pi)^\frac{n+1}{2}} \int_{\R \times \R^n } e^{-i(\gamma t + (\xi^\prime, \zeta) \cdot x)} G(t, x) \, \dd (t, x).
\end{equation}
One can check that, for every $\xi^\prime \in \R^{n-1}$, the function 
\begin{equation}
\label{map:Gzz}
\zeta \in \C_{|\Im|<c} \mapsto \widehat{G}(-|\xi^\prime|^2 - \zeta^2, \xi^\prime, \zeta)
\end{equation}
is holomorphic in $\C_{|\Im|<c} $. We will see later that
$\widehat{G}(-|\xi|^2, \xi) = 0$ for all $\xi \in \R^n$, that is,
for every $\xi^\prime \in \R^{n-1}$ the function
\eqref{map:Gzz} vanishes on $\{ \zeta \in \C_{|\Im|<c} : \Im \zeta = 0 \}$.
Then, by analytic 
continuation we have that
\[ \widehat{G} (-|\xi^\prime|^2 - \zeta^2, \xi^\prime, \zeta) = 0  \qquad \forall (\xi^\prime, \zeta) \in \R^{n - 1} \times \C_{|\Im| < c},\]
which implies that \eqref{map:Gz} vanishes at the poles 
$\{ \zeta_1, \zeta_2, \zeta_3, \zeta_4\}$.
Indeed, if $- \tau - |\xi^\prime|^2 > 0$ we have that
\[ \widehat{G} (\tau, \xi^\prime, \zeta_j) =  \widehat{G} (-|\xi^\prime|^2 - \zeta_j^2, \xi^\prime, \zeta_j) = 0 \qquad j \in \{ 1, 2 \}; \]
however, $- \tau - |\xi^\prime|^2 < 0$ and $|\tau + |\xi^\prime|^2 | < c^2$
we have that
\[ \widehat{G} (\tau, \xi^\prime, \zeta_j) =  \widehat{G} (-|\xi^\prime|^2 - \zeta_j^2, \xi^\prime, \zeta_j) = 0 \qquad j \in \{ 3, 4 \}. \]
Therefore, the function \eqref{map:Gz} vanishes at the poles.

To conclude the proof of this corollary, we need to show that
$\widehat{G}(-|\xi|^2, \xi) = 0$ for all $\xi \in \R^n$. By the expression
\eqref{map:Ghatext}, we have that
\[ \widehat{G}(-|\xi|^2, \xi) = \frac{1}{(2 \pi)^\frac{n+1}{2}} \int_\Sigma e^{-i(-|\xi|^2 t + \xi \cdot x)} F(t, x) \, \dd (t, x). \]
Since every arbitrary function in $ \mathcal{S}(\R^n)$ can be written as 
$\widehat{\phi}$ with $\phi \in \mathcal{S}(\R^n)$, we have
\[ \int_{\R^n} \widehat{G}(-|\xi|^2, \xi) \overline{\widehat{\phi(\xi)}} \, \dd \xi = \frac{1}{(2 \pi)^\frac{1}{2}} \int_\Sigma (i\partial_\tm + \Delta) u \, \overline{[e^{i\tm \Delta} \phi]} \]
using the fact that $(i\partial_\tm + \Delta) u = F$ in $\Sigma$. The 
\cref{lem:integration_by_parts} and the fact that 
$u(0, \centerdot) = u(T, \centerdot) = 0$ imply that the right-hand side of
the previous identity vanishes. Hence $\widehat{G}(-|\xi|^2, \xi) = 0$ for all 
$\xi \in \R^n$, which concludes this proof.
\end{proof}

Finally, we focus on proving the \cref{lem:exp_decay}, which follows the lines of 
\cite[Lemma 4.6]{10.1063/5.0152372} but using the new inequality of the \cref{th:non-gain_Snu}.

\begin{proof}[Proof of the \cref{lem:exp_decay}]
Start by checking that the right-hand side of the 
inequality in the statement of this lemma is finite:
\begin{equation}
\label{in:exp_decay}
\| e^{\nu \cdot \x} F \|_{L^q ((0, T); L^r(\R^n))} \leq \| e^{-( c - |\nu|) |x|} \|_{L^\infty(\R^n)} \| e^{c|\x|} F \|_{L^q ((0, T); L^r(\R^n))}.
\end{equation}

For $\nu \in \R^n \setminus \{ 0 \}$ and $g \in \mathcal{S}(\R \times \R^n)$,
let us define the multiplier $T_\nu $:
\[T_\nu g (t, x) = \frac{1}{(2 \pi)^\frac{n+1}{2}} \int_{\R \times \R^n} e^{i(t \tau + x \cdot \xi)} \frac{\widehat{g} (\tau, \xi)}{-\tau - (\xi + i \nu)\cdot(\xi + i \nu)} \, \dd (\tau, \xi) \]
for all $(t, x) \in \R \times \R^n$. Note that $T_\nu$ is the Fourier multiplier with symbol
$p_{-\nu} (\tau - |\nu|^2, \xi)^{-1}$ with $p_\nu$ defined as in \eqref{id:symbol}.
Changing variables according to $ \sigma = \tau - |\nu|^2$, we have that
\begin{align*}
& T_\nu g (t, x) = e^{i|\nu|^2 t} S_{-\nu} f (t, x),\\
& f (s, y) = e^{-i|\nu|^2 s} g(s, y);
\end{align*}
where the multiplier $S_\nu$ is defined in \eqref{id:multiplier_definition}.
Using the \cref{th:non-gain_Snu} and the previous identities we know that there is a constant $C > 0$,
that only depends on $n$, $q$ and $r$, such that
\begin{equation}
\label{in:boundedness_Tnu}
\| T_\nu g \|_{L^{q^\prime} (\R; L^{r\prime}(\R^n))} \leq C \| g \|_{L^q (\R; L^r(\R^n))}
\end{equation}
for all $g \in \mathcal{S}(\R \times \R^n)$.

Let $G \in L^p(\R; L^r(\R^n))$ with $p \in [1, q]$ denote the function
\begin{equation*}
G (t, \centerdot) =
	\left\{
		\begin{aligned}
		& F(t, \centerdot) & & \textnormal{if} \, t \in (0,T), \\
		& 0 &  & \textnormal{if} \, t \in \R \setminus (0,T).
		\end{aligned}
	\right.
\end{equation*}
Obviously, $e^{c|\x|} G \in L^p(\R; L^r(\R^n))$ with $p \in [1, q]$.
As a consequence of the inequality \eqref{in:boundedness_Tnu}, $T_\nu$ can be 
extended to define $T_\nu (e^{\nu \cdot \x} G)$ for every 
$\nu \in \R^n \setminus \{ 0 \}$ such that $|\nu| < c$ since
\[\| e^{\nu \cdot \x} G \|_{L^q (\R; L^r(\R^n))} = \| e^{\nu \cdot \x} F \|_{L^q ((0, T); L^r(\R^n))} \]
and \eqref{in:exp_decay} holds. Furthermore,
\begin{equation}
\label{in:TnuExpG}
\| T_\nu (e^{\nu \cdot \x} G) \|_{L^{q^\prime} (\R; L^{r\prime}(\R^n))} \leq C \| e^{\nu \cdot \x} F \|_{L^q ((0, T); L^r(\R^n))}
\end{equation}
for all $\nu \in \R^n \setminus \{ 0 \}$ such that $|\nu| < c$.

For $\nu \in \R^n \setminus \{ 0 \}$ define
\[ \Gamma_\nu = \{ (\tau, \xi) \in \R \times \R^n : \tau = - (\xi + i \nu)\cdot(\xi + i \nu) \}. \]
Since the codimension of $\Gamma_\nu$ is $2$, we have by
the item \eqref{ite:L1} of \cref{lem:Fourier_hol-extension}
that---for every $\phi \in \mathcal{D}(\R \times \R^n)$ and every $\nu \in \R^n \setminus \{ 0 \}$ such 
that $|\nu| < c$---the function
\[ (\tau, \xi) \in (\R \times \R^n) \setminus \Gamma_\nu \mapsto \frac{\widehat{e^{\nu \cdot \x} G} (\tau, \xi)}{- \tau - (\xi + i \nu)\cdot(\xi + i \nu)} \, \widecheck{e^{-\nu \cdot \x} \phi}(\tau, \xi) \]
can be extended to $\R \times \R^n$ so that it represents an element in 
$L^1(\R \times \R^n)$. Here 
\[ \widecheck{e^{-\nu \cdot \x} \phi}(\tau, \xi) = \widehat{e^{-\nu \cdot \x} \phi}(-\tau, -\xi).  \]
Hence, we can write
\begin{align*}
\langle e^{- \nu \cdot \x}  T_\nu (e^{\nu \cdot \x} G), \phi \rangle & = \langle  T_\nu (e^{\nu \cdot \x} G), e^{-\nu \cdot \x} \phi\rangle \\
& = \int_{\R \times \R^n} \frac{\widehat{e^{\nu \cdot \x} G} (\tau, \xi)}{- \tau - (\xi + i \nu)\cdot(\xi + i \nu)} \, \widecheck{e^{-\nu \cdot \x} \phi}(\tau, \xi) \, \dd (\tau, \xi)
\end{align*}
for all $\phi \in \mathcal{D}(\R \times \R^n)$ and 
$\nu \in \R^n \setminus \{ 0 \}$ such that $|\nu| < c$. 
For every $\nu \in \R^n \setminus \{ 0 \}$, 
there exists a $Q \in \Orth(n)$ so 
that $\nu = |\nu| Q e_n$, with $e_n$ the $n^{\rm th}$ element of the standard 
basis of $\R^n$. Let $G_Q$ and $\phi_Q$ denote the functions 
$G_Q (t, x) = G (t, Qx)$ and $\phi_Q (t, x) = \phi (t, Qx)$ for all 
$(t, x) \in \R \times \R^n$. Changing variables, we have that
\[ \langle e^{- \nu \cdot \x}  T_\nu (e^{\nu \cdot \x} G), \phi \rangle = \int_{\R \times \R^n} \frac{\widehat{e^{|\nu| \x_n} G_Q} (\tau, \xi)}{- \tau - (\xi + i |\nu| e_n)\cdot(\xi + i |\nu| e_n)} \, \widecheck{e^{-|\nu| \x_n} \phi_Q}(\tau, \xi) \, \dd (\tau, \xi). \] 
By the item \eqref{ite:representation} of the \cref{lem:Fourier_hol-extension},
we have that
the right-hand side of the previous identity is equal to
\[ \int_{\R \times \R^{n-1} \times \R} \frac{\widehat{G_Q} (\tau, \xi^\prime, \xi_n + i |\nu|)}{- \tau - |\xi^\prime|^2 - (\xi_n + i |\nu|)^2} \, \widecheck{\phi_Q}(\tau, \xi^\prime, \xi_n + i |\nu|) \, \dd (\tau, \xi^\prime, \xi_n). \]
By the Fubini--Tonelli theorem the function
\[ \xi_n \in \R \mapsto \frac{\widehat{G_Q} (\tau, \xi^\prime, \xi_n + i |\nu|)}{- \tau - |\xi^\prime|^2 - (\xi_n + i |\nu|)^2} \, \widecheck{\phi_Q}(\tau, \xi^\prime, \xi_n + i |\nu|) \]
belongs to $L^1(\R)$ for almost every $(\tau, \xi^\prime) \in \R \times \R^{n-1}$,
the function
\[ H_Q^{|\nu|} : (\tau, \xi^\prime) \in \R \times \R^{n-1} \mapsto \int_\R \frac{\widehat{G_Q} (\tau, \xi^\prime, \xi_n + i |\nu|)}{- \tau - |\xi^\prime|^2 - (\xi_n + i |\nu|)^2} \, \widecheck{\phi_Q}(\tau, \xi^\prime, \xi_n + i |\nu|) \, \dd \xi_n \]
belongs to $L^1(\R \times \R^{n - 1})$, and
\[ \langle e^{- \nu \cdot \x}  T_\nu (e^{\nu \cdot \x} G), \phi \rangle = \int_{\R \times \R^{n - 1}} H_Q^{|\nu|}(\tau, \xi^\prime) \, \dd (\tau, \xi^\prime).  \]

Note that, by the \cref{cor:vanish} and the \cref{lem:Fourier_hol-extension}
the function
\begin{equation}
\label{map:GQxin}
\xi_n \in \R \mapsto \frac{\widehat{G_Q} (\tau, \xi^\prime, \xi_n)}{- \tau - |\xi^\prime|^2 - \xi_n^2} \, \widecheck{\phi_Q}(\tau, \xi^\prime, \xi_n)
\end{equation}
belongs to $L^1(\R)$ for all $(\tau, \xi^\prime) \in \R \times \R^{n-1}$ such that
$\tau \neq - |\xi^\prime|^2$. 

Our final goal will be to show that the function
\[H_Q : (\tau, \xi^\prime) \in \R \times \R^{n-1} \mapsto \int_\R \frac{\widehat{G_Q} (\tau, \xi^\prime, \xi_n)}{- \tau - |\xi^\prime|^2 - \xi_n^2} \, \widecheck{\phi_Q}(\tau, \xi^\prime, \xi_n) \, \dd \xi_n\]
satisfies that
\begin{equation}
\label{id:HQ-HQnu}
H_Q (\tau, \xi^\prime) = H_Q^{|\nu|} (\tau, \xi^\prime) \qquad \forall (\tau, \xi^\prime) \in \R \times \R^{n-1} : \tau \neq - |\xi^\prime|^2.
\end{equation}
With this at hand, we have that
\begin{equation}
\label{id:TnuHQ}
\langle e^{- \nu \cdot \x}  T_\nu (e^{\nu \cdot \x} G), \phi \rangle = \int_{\R \times \R^{n - 1}} H_Q (\tau, \xi^\prime) \, \dd (\tau, \xi^\prime). 
\end{equation}
As we will see at a later stage, the function
\begin{equation}
\label{map:GQ}
(\tau, \xi) \in \R \times \R^n \mapsto \frac{\widehat{G_Q} (\tau, \xi)}{- \tau - |\xi|^2} \, \widecheck{\phi_Q}(\tau, \xi)
\end{equation}
belongs to $L^1(\R \times \R^n)$. Hence, by the Fubini--Tonelli theorem we have 
that
\[ \int_{\R \times \R^{n - 1}} H_Q (\tau, \xi^\prime) \, \dd (\tau, \xi^\prime) = \int_{\R \times \R^n} \frac{\widehat{G_Q} (\tau, \xi)}{- \tau - |\xi|^2} \, \widecheck{\phi_Q}(\tau, \xi) \, \dd (\tau, \xi). \]
Using the identity \eqref{id:TnuHQ} and the change of variables that removes the 
orthogonal matrix $Q$, we have
\begin{equation}
\label{id:TnuG}
\langle e^{- \nu \cdot \x}  T_\nu (e^{\nu \cdot \x} G), \phi \rangle = \int_{\R \times \R^n} \frac{\widehat{G} (\tau, \xi)}{- \tau - |\xi|^2} \, \widecheck{\phi}(\tau, \xi) \, \dd (\tau, \xi)
\end{equation}
for all $\phi \in \mathcal{D}(\R \times \R^n)$ and 
$\nu \in \R^n \setminus \{ 0 \}$ such that $|\nu| < c$.

Let $v \in C(\R; L^2(\R^n))$ denote the function
\begin{equation*}
v (t, \centerdot) =
	\left\{
		\begin{aligned}
		& u (t, \centerdot) & & \textnormal{if} \, t \in [0,T], \\
		& 0 &  & \textnormal{if} \, t \in \R \setminus [0,T].
		\end{aligned}
	\right.
\end{equation*}
The continuity follows from the fact that 
$u (0, \centerdot) = u (T, \centerdot) = 0$. Using this,
one can check that
\[ (i\partial_\tm + \Delta) v = G \textnormal{ in } \R \times \R^n. \]
Taking Fourier transform we have that
\begin{equation}
\label{eq:RtimeRn}
(- \tau - |\xi|^2) \widehat{v}(\tau, \xi) = \widehat{G} (\tau, \xi) \enspace \dot{\forall} (\tau, \xi) \in \R \times \R^n.\footnote{The symbol $\dot{\forall}$ stands 'for almost every'.}
\end{equation}
This identity and \eqref{id:TnuG} yield
\[\langle e^{- \nu \cdot \x}  T_\nu (e^{\nu \cdot \x} G), \phi \rangle = \langle v, \phi \rangle \]
for all $\phi \in \mathcal{D}(\R \times \R^n)$, or equivalently
\[\langle T_\nu (e^{\nu \cdot \x} G), \psi \rangle = \langle e^{\nu \cdot \x}  v, \psi \rangle \]
for all $\psi \in \mathcal{D}(\R \times \R^n)$---it is enough to consider 
$\phi = e^{\nu \cdot \x} \psi$ for an arbitrary $\psi$. In particular,
\[ T_\nu (e^{\nu \cdot \x} G)|_\Sigma =  e^{\nu \cdot \x}  u.\]
Then, the inequality \eqref{in:TnuExpG} is the one stated in the 
\cref{lem:exp_decay}.

In order to conclude this proof, it remains to check two points:
the function \eqref{map:GQ} belongs to $L^1(\R \times \R^n)$, and the identity
\eqref{id:HQ-HQnu} holds. Start by showing the first of these points.
It is clear that \eqref{map:GQ} belongs to $L^1(\R \times \R^n)$ if, and only if,
the function
\begin{equation}
\label{map:G}
(\tau, \xi) \in \R \times \R^n \mapsto \frac{\widehat{G} (\tau, \xi)}{- \tau - |\xi|^2} \, \widecheck{\phi}(\tau, \xi)
\end{equation}
also belongs to $L^1(\R \times \R^n)$. In order to check that \eqref{map:G}
belongs to $L^1(\R \times \R^n)$, 
we derive from the identity \eqref{eq:RtimeRn} that
\[ \widehat{v}(\tau, \xi) = \frac{\widehat{G} (\tau, \xi)}{- \tau - |\xi|^2} \]
for almost every $ (\tau, \xi) \in \R \times \R^n$ such that $\tau \neq - |\xi|^2$.
Since the measure of the set 
$\{ (\tau, \xi) \in \R \times \R^n : \tau = - |\xi|^2 \}$ is zero, 
the above identity holds for almost every $ (\tau, \xi) \in \R \times \R^n$.
This together with fact that $v \in L^2(\R \times \R^n)$ ensure that
\eqref{map:G} belongs to $L^1(\R \times \R^n)$.

Finally, let us check the identity \eqref{id:HQ-HQnu}. By the \cref{cor:vanish} 
and the \cref{lem:Fourier_hol-extension}
the function
\[ \zeta \in \C_{|\Im|<c} \mapsto \frac{\widehat{G_Q} (\tau, \xi^\prime, \zeta)}{- \tau - |\xi^\prime|^2 - \zeta^2} \, \widecheck{\phi_Q}(\tau, \xi^\prime, \zeta)\]
is holomorphic. Then, by the Cauchy--Goursat theorem we have that
\begin{align*}
H_Q (\tau, \xi^\prime) &+ \lim_{\rho \to \infty} i \int_0^{|\nu|} \frac{\widehat{G_Q} (\tau, \xi^\prime, \rho + it)}{- \tau - |\xi^\prime|^2 - (\rho + it)^2} \, \widecheck{\phi_Q}(\tau, \xi^\prime, \rho + it) \, \dd t \\
& = H_Q^{|\nu|} (\tau, \xi^\prime) + \lim_{\rho \to \infty} i \int_0^{|\nu|} \frac{\widehat{G_Q} (\tau, \xi^\prime, - \rho + it)}{- \tau - |\xi^\prime|^2 - (-\rho + it)^2} \, \widecheck{\phi_Q}(\tau, \xi^\prime, -\rho + it) \, \dd t 
\end{align*}
for all $(\tau, \xi^\prime) \in \R \times \R^{n-1}$ such that 
$ \tau \neq - |\xi^\prime|^2$. The item \eqref{ite:representation} of the 
\cref{lem:Fourier_hol-extension}, and the decay as $\rho$ tends to infinity,
can be used to show that each of the limits in
the previous identity is zero. Hence, the identity \eqref{id:HQ-HQnu} holds.
This concludes the proof of the \cref{lem:exp_decay}.
\end{proof}

\begin{acknowledgements}
P.C. is supported by the project PID2024-156267NB-I00, as well as BCAM-BERC 
2022-2025 and the BCAM Severo Ochoa CEX2021-001142-S.
\end{acknowledgements}

\bibliography{references}{}
\bibliographystyle{plain}

\end{document}